\documentclass[11pt,a4paper,english]{amsart} 
\usepackage{amsmath,amsthm,amssymb,mathrsfs,mathtools,bm,eucal,tensor} 
\usepackage{stmaryrd}
\usepackage[all,cmtip]{xy}
\usepackage{dsfont}
\usepackage{microtype,lmodern} 
\usepackage[utf8]{inputenc} 
\usepackage[T1]{fontenc} 
\usepackage{enumerate,comment,braket,xspace,tikz-cd,csquotes} 
\usepackage[centering,vscale=0.7,hscale=0.7]{geometry}
\usepackage[hidelinks]{hyperref}
\usepackage[capitalize]{cleveref}

\numberwithin{equation}{section}
\theoremstyle{plain}
\newtheorem{thm}[equation]{Theorem}
\newtheorem{lem}[equation]{Lemma}
\newtheorem{prop}[equation]{Proposition}
\newtheorem{conj}[equation]{Conjecture}
\newtheorem{cor}[equation]{Corollary}

\theoremstyle{definition}
\newtheorem{defin}[equation]{Definition}

\newtheorem{rem}[equation]{Remark}

\newtheorem{ex}[equation]{Example}

\newcommand{\Hom}{\mathrm{Hom}}

\newcommand{\red}{\mathrm{red}}

\newcommand{\cF}{\mathcal F}

\newcommand{\cS}{\mathcal S}

\newcommand{\cV}{\mathcal V}

\newcommand{\cG}{\mathcal G}
\newcommand{\cM}{\mathcal M}
\newcommand{\cC}{\mathcal C}
\newcommand{\cO}{\mathcal O}
\newcommand{\cN}{\mathcal N}
\newcommand{\cI}{\mathcal I}

\newcommand{\cH}{\mathcal H}
\newcommand{\cD}{\mathcal D}
\newcommand{\cE}{\mathcal E}
\newcommand{\cR}{\mathcal R}
\newcommand{\cL}{\mathcal L}
\newcommand{\lra}{\longrightarrow}
\newcommand{\lla}{\longleftarrow}

\DeclareMathOperator{\Max}{Max}
\DeclareMathOperator{\Sl}{Sl}
\DeclareMathOperator{\rank}{rk}
\DeclareMathOperator{\rk}{rk}

\DeclareMathOperator{\Ker}{Ker}

\DeclareMathOperator{\nb}{nb}

\DeclareMathOperator{\irr}{irr}
\DeclareMathOperator{\Sol}{Sol}
\DeclareMathOperator{\DR}{DR}
\DeclareMathOperator{\End}{End}

\DeclareMathOperator{\Irr}{Irr}
\DeclareMathOperator{\divi}{div}
\DeclareMathOperator{\MIC}{MIC}
\DeclareMathOperator{\sing}{sing}

\DeclareMathOperator{\sm}{sm}
\DeclareMathOperator{\bCDiv}{\textbf{CDiv}}
\DeclareMathOperator{\bDiv}{\textbf{Div}}

\DeclareMathOperator{\CDiv}{CDiv}
\DeclareMathOperator{\Div}{Div}
\DeclareMathOperator{\ZR}{ZR}
\DeclareMathOperator{\divis}{divis}

\DeclareMathOperator{\an}{an}

\DeclareMathOperator{\RcHom}{R\cH om}
\DeclareMathOperator{\cHom}{\cH om}
\DeclareMathOperator{\RHom}{RHom}

\DeclareMathOperator{\Syme}{Sym}
\DeclareMathOperator{\gr}{gr}
\DeclareMathOperator{\TL}{TL}
\DeclareMathOperator{\ord}{ord}
\DeclareMathOperator{\Cart}{Cart}

\DeclareMathOperator{\Eu}{Eu}
\DeclareMathOperator{\codim}{codim}

\DeclareMathOperator{\Sym}{Sym}
\DeclareMathOperator{\sesi}{ss}
\DeclareMathOperator{\fdeg}{fdeg}
\DeclareMathOperator{\cEnd}{\cE nd}
\DeclareMathOperator{\Vect}{Vect}

\usepackage{xcolor}

\newcommand{\todo}[1]{\textcolor{red}{(Todo: #1)}}

\newcommand{\Remark}[1]{\textcolor{blue}{(Remark: #1)}}

\usepackage[english]{babel}
\usepackage[hidelinks]{hyperref}
\usepackage[capitalize]{cleveref}

\title[Cohomological boundedness for flat bundles on surfaces]{Cohomological boundedness for flat bundles on surfaces and applications }

\begin{document}

\author[H.Hu]{Haoyu Hu}
\address{Department of Mathematics, Nanjing University, Hankou Road 22, Nanjing 210000, China}
\email{huhaoyu@nju.edu.cn, huhaoyu1987@gmail.com}

\author[J.-B. Teyssier]{Jean-Baptiste Teyssier}
\address{Institut de Math\'ematiques de Jussieu, 4 place Jussieu, Paris, France}
\email{jean-baptiste.teyssier@imj-prg.fr}

\begin{abstract}
This paper explores the cohomological consequences of the existence of moduli spaces for flat bundles with bounded rank and irregularity at infinity and gives unconditional proofs.
Namely, we prove the existence of a universal bound for the dimension of  De Rham cohomology of flat bundles with bounded rank and irregularity on surfaces.
In any dimension, we  prove a Lefschetz recognition principle stating the existence of  hyperplane sections distinguishing flat bundles with bounded rank and irregularity after restriction.
We obtain in any dimension  a universal bound for the degrees of the turning loci of flat bundles with bounded rank and irregularity.
Along the way, we introduce a new operation on the group of $b$-divisors on a smooth surface (the partial discrepancy) and  prove a closed formula for the characteristic cycles of flat bundles on surfaces in terms of the partial discrepancy of the irregularity $b$-divisor attached to any flat bundle by Kedlaya.

\end{abstract}

\maketitle

\setcounter{tocdepth}{1}
\tableofcontents

The goal of this paper is to explore the cohomological consequences of the existence of moduli spaces for flat bundles with bounded rank and irregularity at infinity and to prove them unconditio\-nally in the surface case.\\ \indent
Let $U$ be a smooth complex projective  variety.
From Simpson's work \cite{SimpsonI}, flat bundles on $U$ with given rank form a complex variety, the \textit{De Rham space of $U$}.
If $U$ is quasi-projective, flat bundles acquire singularities at infinity making the situation more involved.
However, \textit{regular singular} flat bundles again give rise to moduli spaces at the cost of rigidifying the situation by a choice of logarithmic lattice   \cite{Nitsure}.
Let $X$ be a smooth compactification of $U$ such that $D:=X\setminus U$ has simple normal crossings and let $j : U\to X$ be the inclusion.
A flat bundle $\cE=(E,\nabla)$ on $U$ is regular singular \cite{Del} if $E$ extends to a vector bundle $F$  on $X$ with $\nabla(F)\subset F\otimes_{\cO_X}\Omega^1_X(\log D)$.
Regular singularity is independent of a choice of compactification.
When an extra condition on the residues is imposed, $F$ is   called a \textit{Deligne lattice}.
A natural question to ask is what lies beyond the regular singular case.
In general, the connection $\nabla$ no longer has simple poles at infinity, but Sabbah observed   in the nineties \cite{Sabbahdim} that $\nabla$ still has a simple expression, the so-called \textit{good formal decomposition} at the cost of working formally along $D$  away from a codimension $1$ subset of $D$.
The locus of $D$ where this simple expression  does not hold is the \textit{turning locus} of $\cE$.
Away from the turning locus, Deligne's lattices admit straightforward generalizations and Malgrange proved in \cite{Reseaucan} that they extend canonically to $X$.
Armed with Deligne-Malgrange lattices and conjectural bounds for their Chern classes, it is an expectation of Esnault and Langer dating back from 2014 that there should exist a moduli of finite type for flat bundles on $U$ with bounded rank and  irregularity at infinity.
See \cite{IAS_Video} and \cite[4.3.4]{Kedlaya3}.
Bounded irregularity  means very roughly that we bound the poles of the connection restricted to a Deligne-Malgrange lattice.
See \cref{Def_bounded} for a rigorous definition.
The bound on irregularity is thus embodied by an effective divisor $R$ supported on $D$.
Let us denote by  $\cM_r(X,D,R)$ Esnault-Langer's expected moduli of rank $r$ flat bundles on $U$ with irregularity bounded by $R$ along $D$.
As pointed out to the authors by A. Langer, there is a mismatch between Nitsure's construction and $\cM_r(X,D,R)$ for $R=0$.
Indeed, rigidifying the moduli problem with choices of logarithmic lattices makes Nitsure's construction depend on a choice of compactification for $X$.
On the other hand, having regular singularity at infinity is independent of a choice of compactification.
Furthermore, different points of Nitsure's construction may underlie the same flat bundle. 
\\ \indent
The existence of $\cM_r(X,D,R)$ has deep consequences for flat bundles which can be stated and studied independently.
The first consequence is \textit{cohomological}.
For a flat bundle $\cE$ on $U$, let $\DR\cE$ be the algebraic De Rham complex of $\cE$.
Following the construction of the jumping loci for character varieties, the following subsets
$$
\cV^{j}:=\{\cE\in \cM_r(X,D,R) \text{ such that } \dim H^\ast(U, \DR  \cE )\geq j\}
$$
should form a decreasing sequence of closed subsets of $\cM_r(X,D,R)$ when $j$ increases.
Since $\cM_r(X,D,R)$ is expected to be of finite type, its underlying topological space should be noetherian, so the sequence of $\cV^{j}$ should stabilize for $j$ big enough. 
On the other hand, De Rham cohomology is always finite dimensional.
Hence, there should exist an integer $j_0$ such that $\cV^{j_0}$ is empty. 
Put otherwise, there should exist a universal bound for the algebraic De Rham cohomology of rank $r$ flat bundles on $U$ with bounded irregularity at infinity.
Furthermore, this bound should depend only on $X,D,R$ and $r$.
See \cref{conjecture_coho} for a statement precising the dependency in $r$ and $R$.
One of the main result of this paper is an unconditional proof of the existence of this bound in the surface case (\cref{coho_boundedness_surface}).
To state it, we let  $\Div(X,D)$ be the group of divisors of $X$ supported on $D$ and denote by $k$ a field of characteristic $0$.
For every  effective divisor $R$ of $X$ supported on $D$, for every integer $r\geq 0$, we let $\MIC_r(X,D,R)$ be the category of flat bundles on $U$ with rank smaller than $r$ and irregularity bounded by $R$ along $D$.

\begin{thm}\label{coho_boundedness_surface_intro}
Let $X$ be a smooth projective surface over $k$.
Let $D$ be a normal crossing divisor of $X$.
Then, there exists a quadratic polynomial $C : \Div(X,D)\oplus \mathds{Z}\to\mathds{Z}$ affine in the last variable such that for every  effective divisor $R$ of $X$ supported on $D$, for every integer $r\geq 0$ and every object $\cE$ of $\MIC_r(X,D,R)$, we have
$$
\dim H^{*}(U,\DR \cE )\leq C(R,r)
$$
\end{thm}
A second consequence of the existence of  $\cM_r(X,D,R)$ is the \textit{Lefschetz recognition principle}.
It relies on the heuristic principle that the formation of $\cM_r(X,D,R)$ should come packaged with functorialities.
If $f : Y\to X$ is a morphism of smooth projective varieties over  $k$ such that $f^*R$ makes sense and such that $f^{-1}(D)$ is a normal crossing divisor,
one expects the existence of a morphism of varieties $\cM_r(X,D,R)\to \cM_r(Y,f^{-1}(D),f^*R)$ induced by the pull-back along $f: Y\to X$.
Functorialities would make  the study of $\cM_r(X,D,R)$ tractable through the curve case, where they should be much easier to construct since flat bundles on curves have no turning points.
If $\cE_0$ and $\cE_1$ are non isomorphic flat bundles on $U$,  one can find  a hyperplane $H$ such that the restrictions of $\cE_0$ and $\cE_1$ to $X\cap H$ are again non isomorphic.
When viewed as points of $\cM_r(X,D,R)$, this means geometrically that  $\cE_0$ and $\cE_1$  do not lie in the same fibre of 
$\cM_r(X,D,R)\to \cM_r(X\cap H,f_H^{-1}(D), f_H^*R)$
where $f_H : X\cap H  \to X$ is the inclusion.
One thus expects the induced morphism of schemes
$$
\cM_r(X,D,R)\lra \prod_H\cM_r(X\cap H,f_H^{-1}(D), f_H^*R)
$$
to be a closed immersion.
Since $\cM_r(X,D,R)$ should be of finite type, this implies the existence of hyperplanes $H_1,\dots, H_N$ depending only on $X,D,R$ and $r$ such that if we put $f: \bigsqcup X\cap H_i \to X$, then 
$\cM_r(X,D,R)\to  \cM_r(Y,f^{-1}(D), f^*R)$
is a closed immersion.
Put in the language of flat bundles, two objects of $\MIC_r(X,D,R)$ are isomorphic if and only their pull-back to $\bigsqcup X\cap H_i $ are isomorphic.
This is the \textit{Lefschetz recognition principle}.
We are led to introduce the following

\begin{defin}\label{realize_LRP}
Let $X$ be a smooth projective variety of dimension $n\geq 2$ over a field $k$ of characteristic $0$.
Let $D$ be a simple normal crossing divisor of $X$.
Let $X\to \mathds{P}$ be a closed immersion in some projective space.
Let $\cC$ be a class of flat vector bundles on $X-D$.
Let $\cH$ be a set of hyperplanes in $\mathds{P}$ meeting $X$ transversally.
We say that $\cH$ \textit{realizes the Lefschetz recognition principle for $\cC$} if for every  $\cM_1,\cM_2 \in  \cC$, the flat bundles $\cM_1$ and $\cM_2$ are isomorphic if and only if $\cM_1|_{X\cap H}$ and $\cM_2|_{X\cap H}$ are isomorphic for every $H\in \mathcal{H}$.
\end{defin}

Our second main result  is a  proof of the Lefschetz recognition principle for $\cC = \MIC_r(X,D,R)$  in any dimension :

\begin{thm}\label{Lefchetz_connections_surface_intro}
Let $X$ be a smooth projective variety of dimension $n\geq 2$ over $k$.
Let $D$ be a simple normal crossing divisor of $X$.
Let $X\to \mathds{P}$ be a closed immersion in some projective space.
Then, there exists a polynomial $K : \Div(X,D)\oplus \mathds{Z}\to \mathds{Z}$ of degree $4$ such that for every  effective divisor $R$ of $X$ supported on $D$, for every integer $r\geq 0$, there is a dense open subset of $\Omega(R,r) \subset (\mathds{P}^{\vee})^{K(R,r)}$ such that every $(H_1,\dots, H_{K(R,r)})\in \Omega(R,r)$ realizes the Lefschetz recognition principle for $\MIC_r(X,D,R)$.
\end{thm}
See \cref{Lefschetz_recognition_hyperplane} for explicit conditions defining the dense open set $\Omega(R,r)$.\\ \indent

Although the statements of cohomological boundedness and the Lefschetz recognition principle were derived separately from  the existence of $\cM_r(X,D,R)$, it turns out that \textit{for surfaces}, the latter is a consequence of the former. 
Let us explain how.
Let $j : U\to X$ be the open immersion and let $\cE$ be a flat bundle on $U$.
The $\cD_X$-module $j_*\cE$ is coherent and as such, a canonical $\dim U$-cycle $CC( j_*\cE)$  of the cotangent bundle $T^*X$ was attached to it by Kashiwara \cite{KS}. 
This is the \textit{characteristic cycle   of $j_* \cE$}.
The cycle $CC( j_*\cE)$ tells how far the $\cD_X$-module $j_*\cE$ is from being a flat connection on $X$ and thus provides a geometric measure of the complexity of the differential system underlying $j_*\cE$.
The cycle $CC( j_*\cE)$  is Lagrangian.
If $X$ is a surface, $CC( j_*\cE)$ is thus a combination of the zero section of $T^*X$, the conormal bundles of the components of $D$ and the conormal bundles of some points of $D$.
If $\cE'$  is another flat bundle on $U$, then any smooth curve transverse to $D$ and avoiding the points whose conormal bundle contributes to $CC(j_*\cHom(\cE_1,\cE_2))$ where $\cE_1,\cE_2\in \{\cE,\cE'\}$ distinguishes $\cE$ and $\cE'$ after restriction.
See \cref{isomorphism_between_Hom}.
This observation translates the statement of \cref{Lefchetz_connections_surface_intro} into the problem of finding a universal bound for the number of points of $D$ whose conormal bundle contributes to $CC( j_*\cE)$ for  $\cE$ in  $\MIC_r(X,D,R)$.
This question is  cohomological  since by Kashiwara-Dubson's formula \cite{Dubson}, every such point  contributes to De Rham cohomology.
This is how \cref{Lefchetz_connections_surface_intro} follows from  \cref{coho_boundedness_surface_intro} for surfaces.\\ \indent
The main result of \cite{teyConjThese} implies that the points of the smooth locus of $D$ whose conormal bundle contributes to $CC( j_*\cE)$ and $CC( j_*\cEnd\cE)$ are exactly the turning points of $\cE$ along $D$.
From the above discussion, cohomological boundedness thus yields a universal bound on the number of turning points of  objects in $\MIC_r(X,D,R)$.
We prove  in any dimension a stronger  universal bound  (\cref{bound_degree_turning_locus}) :

\begin{thm}\label{bound_degree_turning_locus_intro}
Let $X$ be a smooth projective variety over $k$.
Let $D$ be a normal crossing divisor of $X$.
Let $X\to \mathds{P}$ be a closed immersion in some projective space.
Then, there exists a quadratic polynomial $K : \Div(X,D)\oplus \mathds{Z}\to \mathds{Z}$ such that for every  effective divisor $R$ of $X$ supported on $D$, for every integer $r\geq 0$ and every object $\cE$ of $\MIC_r(X,D,R)$, the degree of the turning locus of $\cE$ along $D$ is smaller than $K(R,r)$.
\end{thm}
For a flat bundle $\cE$, Kedlaya \cite{Kedlaya2} and Mochizuki \cite{Mochizuki1} proved that the turning locus of $\cE$ along $D$ can be eliminated  by enough blow-up.
\cref{bound_degree_turning_locus_intro} is thus consistent with Esnault-Langer's expectation that there should exist a universal bound on the number of blow-up needed to achieve good formal decomposition for objects in $\MIC_r(X,D,R)$.
Their insight is that controlling the number of blow-up should give the required bound on the Chern classes of Deligne-Malgrange lattices  to construct  the moduli $\cM_r(X,D,R)$. 
By purity of turning loci proved by André for $D$ smooth \cite{andre} and  by Kedlaya in general \cite{Kedlaya3}, we deduce the following (\cref{bound_turning_locus})

\begin{thm}\label{bound_comp_turning_locus_intro}
Let $X$ be a smooth projective variety over $k$.
Let $D$ be a normal crossing divisor of $X$.
Let $X\to \mathds{P}$ be a closed immersion in some projective space.
Then, there exists a quadratic polynomial $K : \Div(X,D)\oplus \mathds{Z}\to \mathds{Z}$ such that for every  effective divisor $R$ of $X$ supported on $D$, for every integer $r\geq 0$ and every object $\cE$ of $\MIC_r(X,D,R)$, the set of irreducible components of the turning locus of $\cE$ along $D$ is smaller than $K(R,r)$.
\end{thm}
We summarize the interplay (at least for surfaces) of the above results and heuristic in the next diagram, to be understood with bounded rank and irregularity.
$$
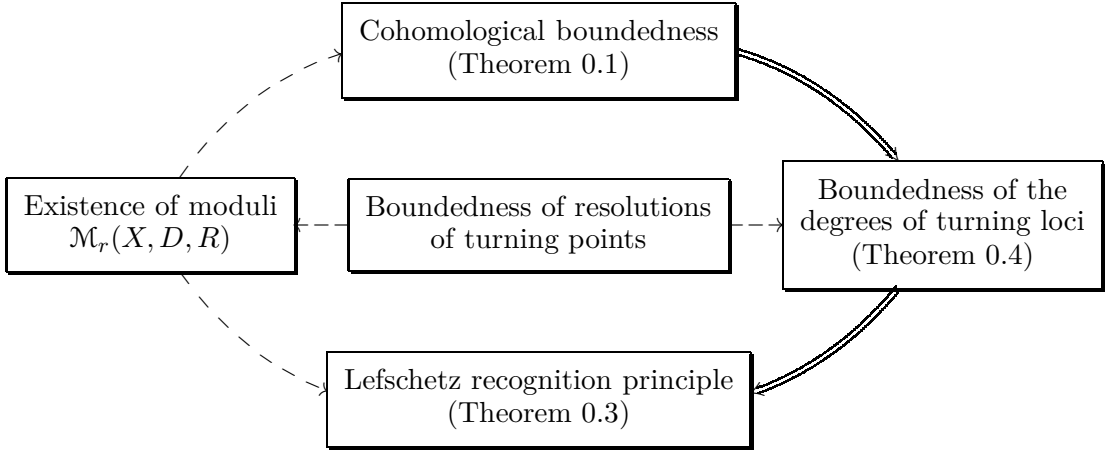
\xymatrix@C-12pt{
                                                               & *++[F-,]\txt{Cohomological boundedness \\ (\cref{coho_boundedness_surface_intro})}   \ar@{=>}@/^2.5pc/[rd]   &            \\  
 *++[F-,]\txt{Existence  of moduli \\ $\cM_r(X,D,R)$}   \ar@{-->}@/^2.5pc/[ru] \ar@{-->}@/_2.5pc/[rd]    &  *++[F-,]\txt{Boundedness of resolutions \\ of turning points}  \ar@{-->}[l]   \ar@{-->}[r]  &    *++[F-,] \txt{Boundedness of  the \\  degrees of turning loci \\ (\cref{bound_degree_turning_locus_intro}) }     \ar@{=>}@/^2.5pc/[ld]      \\
                                                                    &   *++[F-,] \txt{Lefschetz recognition principle \\  (\cref{Lefchetz_connections_surface_intro})}   &             
}
$$
We finally introduce the main tool of this paper and describe an extra application.
Let $X$ be a smooth connected variety over $k$. 
Let $D$ be a normal crossing divisor of $X$.
Let $\cE$ be a flat bundle on $ j :U:=X\setminus D\to X$.
For every proper birational morphism $p : Y\to X$ where $Y$ is smooth and $E:=p^{-1}(D)$ is a normal crossing divisor, Kedlaya  attached to $\cE$ an effective divisor $\Irr(Y,p^+\cE)$ supported on $E$ whose coefficient along a component $Z$  is the generic irregularity number of $\cE$  along $Z$, or equivalently the coefficient of $T^*_ZY$ in $CC(\jmath_{*}\cE)$, where $\jmath: U\to Y$ is the inclusion.
This collection of  divisors organizes into an element $\Irr \cE$ of the group $\bDiv(X):=\varprojlim_{Y\rightarrow X} \Div(Y)$ of $b$-divisors, where  $p : Y\to X$ runs over the poset of morphisms as above and where the transition maps are  push-forward.
In particular, an element of $\bDiv(X)$ is a $\mathds{Z}$-valued function on the set of divisorial valuations on $X$.
The $b$-divisor $ \Irr\cE$  is the \textit{irregularity $b$-divisor of  $\cE$ along $D$}.
In the group $\bDiv(X)$ lies the subgroup of \textit{Cartier $b$-divisor} defined as $\varinjlim_{Y\rightarrow X}\Div(Y)$
where the transition maps are pull-backs.
Fundamental to this paper is   Kedlaya's theorem  \cite{Kedlaya3} that  $\Irr\cE$ is a nef Cartier $b$-divisor with the property that $\Irr\cE$  and $\Irr\cEnd\cE$   lie in the subgroup $\Div(X)$   of $\bDiv(X)$  if and only if $\cE$ has good formal decomposition along $D$.
We show in \cref{H_Kato} that  $\Irr\cE$ and $\rk \cE$ determine $CC(j_*\cE)$, and it is an intriguing question to understand how $CC(j_*\cE)$ relates to $\Irr\cE$.
If $X$ is a surface, we  answer this question  using a new operation on $b$-divisors, the \textit{partial discrepancy}  $\delta : \bDiv(X)\to \bDiv(X)$.
The partial discrepancy measures the failure of a $b$-divisor $Z$ to lie in the image of $\Div(X)\to \bDiv(X)$ in the sense that if it does, then $\delta Z=0$.
We prove that if $Z$ is a nef Cartier $b$-divisor, then $\delta Z$ is an effective $b$-divisor with finite support when viewed as a function on the set of divisorial valuations on $X$.
Hence, the sum of its values, denoted by $\int_X \delta Z$ is a well-defined positive integer.
For a subset $A$ of $X$, we denote by $\int_A \delta Z$ the sum of the values of $\delta Z$ over the set of divisorial valuations whose centers on $X$ lie in $A$.
Using the partial discrepancy, we prove the following   formula for the characteristic cycle of flat bundles on surfaces (\cref{computation_CC_surface}) :

\begin{thm}\label{computation_CC_surface_intro}
Let $X$ be a smooth projective surface over an algebraically closed field of characteristic $0$.
Let $D$ be a normal crossing divisor of $X$.
Let $\cE$ be a flat bundle on $ j :U:=X\setminus D\to X$.
Then
$$
CC(j_*\cE) = \rk \cE \cdot  CC(\mathcal{O}_{X}(\ast D)) +  LC(\Irr(X,\cE)) +  \sum_{P\in D}  \left( \int_{P} \delta\Irr\cE\right) \cdot T_{P}^*X 
$$
\end{thm}
In the above formula, $LC(\Irr(X,\cE))$  is a Lagrangian cycle depending only on the divisor $ \Irr(X,\cE)$ supported on $D$ attached to $\cE$.
In particular it depends on $\cE$ only via generic data along the components of $D$.
A nice feature of the above formula is to make  explicit the means by which the lack of good formal decomposition  reflects in the characteristic cycle. 
By definition of the partial discrepancy, it shows in particular that the characteristic cycle is only sensitive to the turning points lying in the smooth locus of the successive inverse images of $D$ by the successive blow-up needed to achieve good formal decomposition.
From Kashiwara-Dubson formula  \cite{Dubson}, we deduce the following Grothendieck-Ogg Shafarevich type formula for flat bundles on surfaces (\cref{calculation_chi_surface}):

\begin{thm}\label{calculation_chi_surface_intro}
Let $X$ be a smooth projective surface over $\mathds{C}$.
Let $D$ be a normal crossing divisor of $X$.
Let $\cE$ be a flat bundle on $ j :U:=X\setminus D\to X$.
Then
$$
\chi(U,\DR\cE) = \rk \cE \cdot  \chi(U(\mathds{C})) +   (LC(\Irr(X,\cE)),T^*_X X)_{T^* X} +  \int_D \delta\Irr\cE
$$
\end{thm}

We now give a linear overview of the paper.
\cref{preparation} gathers some  general material on $\cD$-modules as well as Kedlaya's results on irregularity $b$-divisors.
\cref{coho_conjecture_section} introduces the cohomological boundedness conjecture for flat bundles with bounded rank and irregularity at infinity.
The main upshot of \cref{coho_conjecture_section}  is \cref{quasi_cor_reduction_to_bound_chi} stating that cohomological boundedness is equivalent to an a priori weaker conjecture, the \textit{$\chi$-boundedness conjecture} asking for a universal bound on the global Euler-Poincaré characteristic of the De Rham cohomology.
\cref{nearby_section} is the main technical core of this paper.
In a relative situation, it provides a mechanism for deducing $\chi$-boundedness out of $\chi$-boundedness for the generic fibre provided the turning locus is contained in a fibre.
See \cref{bound_chi_D}.
\cref{partial_discr} is devoted to the construction of the partial discrepancy $b$-divisor attached to a $b$-divisor on a smooth surface.
Its upshot is \cref{image_delta_of_nef_Cartier} ensuring that the partial  discrepancy of a nef Cartier $b$-divisor is a $b$-divisor with finite support when viewed as a function on the set of divisorial valuations.
\cref{formula_surfaces} is an application of \cref{partial_discr} to the proof of \cref{computation_CC_surface_intro} and \cref{calculation_chi_surface_intro}.
\cref{coho_bound_proof} gives the proof of cohomological boundedness for surfaces. 
\cref{Lefschetz} provides the proof of the Lefschetz recognition principle.
In \cref{Tannaka}, the techniques of this paper are used to obtain a Lefschetz theorem for the differential Galois group of flat bundles under some uncountability assumption of the base field.

\subsection*{Acknowledgement}
We thank N. Budur for pointing out the work of L. Xiao \cite{LXiao_CC_clean}.
We thank A. Langer for explaining to us the mismatch between Nitsure's construction and the expectation of what $\cM_r(X,D,R)$  should be when $R=0$.
We thank J. Sauloy for prompting the authors to look for genericity and not only existence of the hyperplanes from \cref{Lefchetz_connections_surface_intro}.
H.H. is supported by the National Natural Science Foundation of China (grant No. 11901287) and the Natural Science Foundation of Jiangsu Province (grant No. BK20190288).

\section{Geometric and $\cD$-module preparations}\label{preparation}
\subsection{Base field}
In this paper, $k$ will  denote a field of characteristic $0$.
If $X$ is a variety over $k$, we denote by $X^{\sm}$ the smooth locus of $X$ and by $X^{\sing}$ the singular locus of $X$.  
If $k\subset K$ is a field extension, we denote by $X_K$ the pull-back of $X$ over $K$.
\subsection{Pair of varieties}
 A \textit{pair over $k$} is the data $(X,D)$  of a smooth variety $X$ over $k$ with a reduced divisor $D$. 
An \textit{analytic pair} is the data $(X,D)$ of a complex manifold $X$ with a reduced divisor $D$.
In both situations, we denote by $\Div(X,D)$ the group of divisors of $X$ supported on $D$, that is, $\Div(X,D)$  is the free abelian  group over the set of  irreducible components of $D$. 
We denote by $\fdeg : \Div(X,D)\to \mathds{Z}$ the formal degree function, that is the group morphism  sending  each irreducible component of $D$ to $1$.
Note that if $k\subset K$ is a field extension, then $\fdeg R\leq \fdeg R_K$, where $R_K$ denotes the pull-back  of $R$ to $X_K$.
\\ \indent
If $P$ is a property of algebraic or analytic varieties, a $P$-pair $(X,D)$ will refer to a pair $(X,D)$ such that $X$ satisfies $P$.
Finally, a $P$-normal crossing pair $(X,D)$ is a $P$-pair such that $D$ has normal crossing.\\ \indent
A \textit{morphism of pairs $f: (Y,E)\to (X,D)$ over $k$} is a morphism $f : X\to Y$ of algebraic varieties over $k$ such that $f^{-1}(D)=E$.

\subsection{Transversality}
We recall the transversality conditions from \cite{Bei}.
\begin{defin}\label{transversal}
Let $f:Y\longrightarrow X$ be a morphism between smooth varieties over $k$. 
Let $C$ be a closed conical subset of $T^*X$.
Let  $\overline y$ be a geometric point above a point $y$ of $Y$.
We say that $f:Y\longrightarrow X$ is $C$-transversal at $y$ if 
$$
\Ker df_{\overline y}\bigcap C_{f(\overline y)}\subset \{0\}\subset T^*_{f(\overline y)}X
$$
where $df_{\overline y}: T^*_{f(\overline y)}X   \longrightarrow T^*_{\overline y}Y$ is the cotangent map of $f$ at $\overline y$. 
We say that $f:Y\longrightarrow X$ is $C$-transversal if it is $C$-transversal at every point of $Y$.
\end{defin}
 If $f:Y\longrightarrow  X$ is $C$-transversal, let  $f^\circ C$ be the scheme theoretic image of $Y\times_XC$ in $T^*Y$ by the canonical map $df: Y\times_XT^*X\longrightarrow T^*Y$. 
As proved in \cite[Lemma 1.2]{Bei}, the map $df:Y\times_XC\longrightarrow f^\circ C$ is finite and $f^\circ C$ is a closed conical subset of $T^*Y$.\\ \indent

In the next transversality criterion, we use an extra operation on closed conical subsets of cotangent bundles.
Let $f:X\longrightarrow Y$ be a proper morphism between smooth varieties over $k$. 
Let $C$ be a closed conical subset of $T^*X$.
We denote by $f_{\circ}C$ the closed conical subset of $T^*Y$ defined as the image of $df^{-1}(C)\subset X \times_Y T^*Y$ by the projection  $X \times_Y T^*Y\to T^*Y$, where $df : X\times_Y T^*Y\to T^*X$  is the canonical map.

\begin{lem}\label{transversality_after_closed_immersion}
Let 
\begin{equation}\label{diag_transversality_after_closed_immersion}
\xymatrix{
    Y     \ar[r]^-{j}  \ar[d]_-{f}    &     H \ar[d]^-{g}  \\
  X   \ar[r]^-{i}     &  P
}
\end{equation}
be a cartesian diagram of smooth varieties over $k$ whose arrows are  closed immersions.
Let $C$ be a closed conical subset of $T^*X$. 
Let $x\in Y$ and let us abuse notations by viewing $x$ as a point in $X$, $H$ and $P$.
Then, the following statements hold:
\begin{enumerate}\itemsep=0.2cm
\item If $g$ is $i_{\circ}C$-transversal at $x$, then $f$ is $C$-transversal at $x$.
\item Assume that $X$ and $H$ are transverse at $x$. 
If  $f$ is $C$-transversal at $x$, then $g$ is $i_{\circ}C$-transversal at $x$.
\end{enumerate}

\end{lem}
\begin{proof}
Assume that $f$ is not $C$-transversal at $x$.
This means that there exists a non zero form $\omega$ in $C_x$ vanishing on $T_xY \subset T_xX$.
Since (\ref{diag_transversality_after_closed_immersion}) is cartesian, so is the induced following diagram 
$$
\xymatrix{
    T_xY     \ar[r]  \ar[d]   &     T_xH \ar[d] \\
  T_xX   \ar[r]    &  T_xP
}
$$
Via the choice of a supplementary for $T_x Y$ in $T_xH$ and a supplementary for $T_xX  + T_xH$ in $T_xP$, we can extend $\omega$ to a non zero form $\eta$ on $T_xP$ vanishing on $T_xH$.
In particular, $\eta$ lies in $i_\circ C$ and $g$ is not $i_\circ C$-transversal at $x$ and $(1)$ is proven.\\ \indent
We now prove $(2)$. 
Assume that $T_xX$ and $T_x H$ generate $T_xP$ and that  $f$ is $C$-transversal at $x$.
We argue by contradiction.
Let $\eta$ be a non zero form of $i_\circ C$ above $x$ such that $\eta$ vanishes on $T_x H$.
If $i^*\eta=0$, the fact that  $T_xX$ and $T_x H$ generate $T_xP$ yields $\eta=0$, which is not possible.
Thus, $i^*\eta$ is a non zero form of $C$ above $x$.
By transversality assumption, $i^*\eta$ does not vanish on $T_xY$, so  $\eta$ does not vanish on $T_x H$. 
Contradiction.
This concludes the proof of \cref{transversality_after_closed_immersion}.
\end{proof}

We gather in \cref{easy_transversality} standard facts on transversality.
See \cite[1.2,2.2]{Bei} and \cite[3.4]{cc}.

\begin{lem}\label{easy_transversality}
Let $f : Y \to X$ be a morphism of smooth varieties over $k$.
Let $C,C' \subset T^*X$ be closed conical subsets.
Let $x\in X$.
Then the following hold :
\begin{enumerate}\itemsep=0.2cm
\item Assume $C_x\subset C'_x$. 
If $f$ is $C'$-transversal at $x$, then $f$ is $C$-transversal at $x$.
\item If $f$ is $C'$-transversal at $x$ and $C$-transversal at $x$, then $f$ is $C\cup C'$-transversal at $x$.
\item Let $g : Z \to Y$ be a morphism of smooth varieties over $k$.
The following conditions are equivalent :
\begin{enumerate}\itemsep=0.2cm
\item $f$ is $C$-transversal on a neighbourhood of $g(Z)$ and  $g$ is $f^{\circ}(C)$-transversal.
\item $f\circ g$ is $C$-transversal.
\end{enumerate}
\item If $C'' \subset T^*Y$ is a closed conical subset, the set of points $y\in Y$ at which $f$ is $C''$-transversal  is  open in $Y$.
\end{enumerate}
\end{lem}

\subsection{Universal hyperplane}\label{uni_hyperplane}
Let $V$ be a finite dimensional vector space over  $k$.
Let $V^{\ast}$ be its dual. 
Let $\mathds{P}=\mathrm{Spec}(\Syme V^{\ast})$ and $\mathds{P}^{\vee}=\mathrm{Spec}(\Syme V)$ be the associated projective spaces. 
The closed subscheme  of  $\mathds{P}\times_k\mathds{P}^{\vee}$ defined by
$$
Q:=\{(x,H)\in \mathds{P}\times_k\mathds{P}^{\vee} \;|\; x\in H\}
$$ 
is the universal family of hyperplanes of $\mathds{P}$. 
Let $(X,D)$ be a quasi-projective normal crossing pair over $k$.
Let $i:X\longrightarrow\mathds P$ be an immersion. 
Put $X_Q=X\times_{\mathds{P}}Q$.  
Denote by $p^{\vee}_X:X_Q\longrightarrow \mathds{P}^{\vee}$ the composition of $X_Q\longrightarrow Q$ with the canonical projection $p^{\vee}:Q\longrightarrow \mathds{P}^{\vee}$. 
Then, Bertini's theorem ensures the existence of a dense open subset $V\subset \mathds{P}^{\vee}$ of hyperplanes $H$ transverse to $X$ such that  $D\cap H$ has normal crossings.
In particular, if  $\eta$ denotes the generic point of $\mathds{P}^{\vee}$, we have the following commutative diagram with cartesian squares:
$$
\xymatrix{
X_{\eta }\ar[r]  \ar[d]  &  X_V   \ar[r]  \ar[d]   & X_Q   \ar[r]_-{p_X} \ar[d]   &     X \ar[d]\\
           Q_\eta     \ar[r] \ar[d]            &          Q_V    \ar[r] \ar[d]       &         Q      \ar[r]_-{p}   \ar[d]_-{p^{\vee}}     & \mathds{P}\\
\eta \ar[r]          &  V \ar[r]       & \mathds{P}^{\vee}   & 
}
$$
where $(X_\eta,D_\eta)$ is a  quasi-projective pair over $\eta$  of dimension $\dim X-1$ and where $D_\eta$ has normal crossing.
The pair $(X_\eta,D_\eta)$ is the \textit{generic hyperplane section of $(X,D)$.}


The following lemma provides a slight generalization of \cite[1.3.7 (2)]{SaitoConductorDirectImage} :

\begin{lem}\label{generic_transversality}
Let $\mathds{P}$ be a projective space over  $k$.
Let $X$ be a smooth  subvariety of $\mathds{P}$.
Let $C\subset T^*X$ be a closed conical subset of pure dimension $\dim X$.
Then there exists a dense open subset  of hyperplanes $H$ transverse to $X$ such that $X\cap H \to X$ is $C$-transversal.
\end{lem}
\begin{proof}
By Bertini's theorem, there exists a dense open subset $V_1\subset \mathds{P}^{\vee}$  of hyperplanes transverse to $X$.
Observe that $i_\circ C$ has pure dimension $\dim \mathds{P}$.
Thus \cite[1.3.7 (2)]{SaitoConductorDirectImage} ensures that there exists a dense open subset $V_2 \subset \mathds{P}^{\vee}$  of hyperplanes $H$ such that $H\to \mathds{P}$ is $i_\circ C$-transversal.
From  \cref{transversality_after_closed_immersion}-$(1)$,$V_1\cap V_2$  thus meets our requirements.
\end{proof}

\subsection{Characteristic cycle for coherent $\cD$-modules}
Let $k$ be a field of characteristic $0$.
Let $X$ be a smooth variety over $k$.
We endow $\cD_X$ with its filtration $F$ by the order of differential operators.
In particular, there is a canonical isomorphism $\gr^F \cD_X\simeq \pi_*\cO_{T^*X}$ where $\pi : T^*X \longrightarrow X$ is the canonical projection.
Every coherent $\cD_X$-module $\cM$ admits a filtration $F_{\cM}$ compatible with $F$ such that $\gr^{F_{\cM}}\cM$ is a coherent $\pi_*\cO_{T^*X}$-module.
Then, 
$$
\widetilde{\gr^{F_{\cM}}}\cM:=\mathcal O_{T^*X}\otimes_{\pi^{-1}\pi_*\cO_{T^*X}}\gr^{F_{\cM}}\cM
$$
is a coherent $\mathcal O_{T^*X}$-module.
Its characteristic cycle is a cycle of $T^*X$ which does not depend on a choice of  filtration as above.
It is the \textit{characteristic cycle of $\cM$}.
We denote it by $CC(\cM)$.
The support of $CC(\cM)$ is a closed conical subset of $T^*X$ called the \textit{singular support of $\cM$}.
We denote it by $SS(\cM)$.

\begin{rem}\label{CC_flat_base_change}
The formation of the characteristic cycle commutes with flat base change.
In particular, it commutes with base field extension and analytification.
\end{rem}
As a consequence of \cite{Gabber_CC}, we have the following

\begin{thm}\label{involutive_theorem_Gabber}
For any coherent $\mathcal \cD_X$-module $\cM$, the irreducible components of  $SS(\cM)$ have dimension $\geq \dim X$.
\end{thm}

\begin{defin}
We say that a coherent $\mathcal \cD_X$-module $\cM$ is \textit{holonomic} if $SS(\cM)$ has pure dimension $\dim X$.
\end{defin}
 
Holonomy is the correct finiteness condition for $\mathcal \cD_X$-modules, as will be clear from \cref{bound_cohomology} below.

\subsection{De Rham cohomology}
Let $k$ be a field of characteristic $0$.
Let $X$ be a smooth algebraic $k$-variety of dimension $d$.
Let $\cM$ be a $\cD_X$-module. 
We denote by 
$$
\DR \cM : \cM\lra \Omega^1_X \otimes_{\cO_X}\cM \lra \cdots \lra \Omega^d_X \otimes_{\cO_X}\cM
$$
the algebraic De Rham complex of $\cM$, where $\cM$ lies in degree $0$.
The \textit{algebraic De Rham cohomology of $\cM$} is the cohomology of $\DR \cM$.\\ \indent
If $k=\mathds{C}$, let $\cM^{\an}$ be the analytification of $\cM$ and define similarly the analytic De Rham complex of $\cM$. 
The \textit{analytic De Rham cohomology of $\cM$} is the cohomology of $\DR \cM^{\an}$.\\ \indent
In the proper complex setting, GAGA theorem for quasi-coherent cohomology identifies
algebraic and analytic De Rham cohomology. 
See \cite[6.6.1]{Del}.
This is the following

\begin{prop}\label{GAGA_DR}
Let $X$ be a smooth projective variety over $\mathds{C}$.
Let $\cM$ be a quasi-coherent $\cD_X$-module.
Then, the canonical comparison morphism
$$
H^{\ast}(X,\DR \cM)\lra H^{\ast}(X(\mathds{C}),\DR \cM^{\an})
$$ 
is an isomorphism.
\end{prop}
When \textit{holonomy} is imposed, algebraic De Rham cohomology is finite dimensional even when the ambient variety is not proper.
\begin{prop}\label{bound_cohomology}
Let $X$ be a smooth variety over $k$. 
Let $\cM$ be an holonomic $\cD_X$-module.
Then, for every integer $n$, the space $H^{n}(X,\DR \cM)$ is finite dimensional over $k$ and vanishes if $n\neq 0,\dots, 2\dim X$.
\end{prop}
\begin{proof}
Since the formation of the De Rham complex commutes with push-forward \cite[4.2.5]{HTT} and since holonomy is preserved under push-forward \cite[3.2.3]{HTT}, we can at the cost of replacing $X$ by a smooth compactification suppose that $X$ is proper over $k$.
From  \cref{coho_change_base_field},
we reduce to the case where $k=\mathds{C}$.
From \cref{GAGA_DR}, we are left to prove a variant of \cref{bound_cohomology} where now  $X$ is a smooth compact complex manifold and where $\cM$ is an holonomic $\cD_X$-module.
From Kashiwara's perversity theorem \cite{Ka2}, we are left to prove that for every perverse complex $\cF$ on $X$, the $\mathds{C}$-vector space $H^{n}(X,\cF)$ is finite dimensional and vanishes if $n\neq -\dim X,\dots ,\dim X$.
This is a standard fact  from the theory of perverse complexes \cite[4.2.4]{BBD}.
\end{proof}

\begin{rem} \label{coho_change_base_field}
As a consequence of the invariance of quasi-coherent cohomology under flat base change, De Rham cohomology is invariant under base field extensions.
\end{rem}

\subsection{The Solution and the Irregularity complexes}
Let $X$ be a complex manifold of dimension $d$.
Let $\cM$ be a $\cD_X$-module. 
The \textit{solution complex of $\cM$} is defined as 
$$
\Sol \cM := \RcHom_{\cD_X}(\cM,\cO_X)
$$
The following theorem is due to Kashiwara \cite{Ka2}.
\begin{thm}\label{DR_perverse}
Let $X$ be a smooth complex manifold.
Let $\cM$ be an holonomic $\cD_X$-module. 
Then the complexes $\Sol \cM[\dim X]$ and $\DR \cM[\dim X]$ are perverse complexes.
\end{thm}
Let $Z$ be  an hypersurface in $X$. 
Let $i: Z\to X$ be the inclusion. 
Let  $\cM(\ast Z)$ be the localization of $\cM$ along $Z$. 
Define the \textit{irregularity complex of $\cM$ along $Z$} by 
$$
\Irr_Z^*\cM := i_* i^*\Sol \cM(\ast Z)
$$
The following theorem is due to Mebkhout \cite{Mehbgro}.
\begin{thm}\label{Irr_pervers}
Let $X$ be a  complex manifold.
Let $\cM$ be an holonomic $\cD_X$-module. 
Let $Z$ be  an hypersurface in $X$. 
Then $\Irr_Z^*\cM[\dim X]$ is a perverse complex  supported on $Z$. 
\end{thm}


The De Rham and solution complexes are related under sheaf duality.  
See \cite{DualityMebkhout}.
\begin{thm}\label{dualityDR/Sol}
Let $X$ be a complex manifold.
Let $\cM$ be an holonomic $\cD_X$-module. 
There is a canonical isomorphism of complexes
$$
\DR \cM \lra \RcHom_{\mathds{C}}(\Sol \cM, \mathds{C})
$$
in the derived category of complexes of sheaves.
\end{thm}

\subsection{Characteristic cycle and the De Rham and solution complexes}
For a complex of sheaves $\cF$ with bounded and constructible cohomology, Kashiwa\-ra and Schapira defined the characteristic cycle $CC(\cF)$ by means of microlocal analysis \cite{KS}.  
The following theorem identifies the characteristic cycle of  an holonomic $\cD_X$-module with that of its solution complex.
See \cite[11.3.3]{KS} and \cite[Th 4]{Dubson}.
\begin{thm}\label{KS_equality_CC}
Let $X$ be a complex manifold.
Let $\cM$ be an holonomic $\cD_X$-module.
Then
$$
SS(\cM) = SS(\Sol \cM) \text{ and } CC(\cM) = CC(\Sol \cM)
$$
\end{thm}
In particular, the computation of $CC(\cM)$ can be reduced to a sheaf theoretic question.
If the singular support of $\cM$ is already known, the following theorem  \cite[Th. 3.5]{Ka2} tells how to chop off $X$ in order to compute $CC(\Sol \cM)$.
\begin{thm}\label{Kashiwara_loc_syst}
Let $X$ be a complex manifold. 
Let $\cM$ be an holonomic $\mathcal{D}_X$-module. 
Let $X_1,\dots,X_n$ be a Whitney stratification of $X$ such that $SS(\cM)$ lies in $\bigcup_{i=1}^n T^*_{X_i}X$. 
Then, the cohomology sheaves of $(\Sol \cM)|_{X_i}$  are local systems on $X_i$, $i=1,\dots,n$.
\end{thm}

From a cohomological perspective, computing the characteristic cycle is useful because of Kashiwara-Dubson's formula \cite{Dubson}.
See also \cite{Laumon_chi_D_module}.

\begin{thm}\label{Dubson}
Let $X$ be a proper complex manifold.
Let $\cM$ be an holonomic $\cD_X$-module. 
Then, we have 
$$
\chi(X,\DR\cM)=(CC(\cM),T^*_XX)_{T^*X}.
$$
where $\chi(X,\DR\cM)$ denotes the Euler-Poincaré characteristic of $\DR\cM$ and where $(-,-)_{T^*X}$ denotes the intersection number of cycles in $T^*X$.
\end{thm}

\subsection{Characteristic cycle and functorialities} 
 
 The characteristic cycle for $\cD$-modules commutes with proper push-forward.
 This is the following 
 
\begin{lem}\label{push_forward_CC}
Let $f: X\to Y$ be a proper morphism between smooth varieties over $k$.
Let $\cM$ be an holonomic $\cD_X$-module.
Then, $SS(f_+\cM)=f_\circ SS(\cM)$ and $CC(f_+\cM)=f_*CC(\cM)$.
\end{lem}
\begin{proof}
We argue for $CC$, the case of $SS$ being similar.
From \cref{CC_flat_base_change}, we can suppose $k=\mathds{C}$.
We have the following chain of equalities
\begin{align*}
 CC(f_+\cM)&= CC(f_+^{\an}\cM^{\an}) = CC(\Sol(f_+^{\an}\cM^{\an}))\\
       &  = CC(Rf_*\Sol\cM^{\an}) =f_*CC(\Sol\cM^{\an})  =f_*CC(\cM)
\end{align*}
The first equality follows from \cref{CC_flat_base_change} applied to $f_+\cM$.
The second one follows from \cref{KS_equality_CC}.
The third equality follows from the compatibility of the formation of $\Sol$ with proper push-forward.
The fourth equality follows from \cite[9.4.2]{KS}.
The last equality follows from the above arguments applied to $\cM$.
\end{proof}

 \begin{defin}
Let $f: Y\to X$ be a morphism between smooth varieties over $k$ or between complex manifolds.
Let $\cM$ be an holonomic $\cD_X$-module.
We say that $f: Y\to X$ is non characteristic for $\cM$ if it is $SS(\cM)$-transversal.
 \end{defin}

The following results are due to Kashiwara  \cite{TheseKashiwara} :

\begin{thm}\label{Cauchy-Kowaleska}
Let $f: Y\to X$ be a morphism between smooth varieties over $k$.
Let $\cM$ be an holonomic $\cD_X$-module such that $f: Y\to X$ is non characteristic for $\cM$.
Then the following statements hold : 
\begin{enumerate}\itemsep=0.2cm
\item The $\cD_Y$-module pull-back $f^+\cM$  is concentrated in degree $0$. 
That is, $f^+\cM\simeq f^* \cM$. 
\item We have $SS(f^+\cM )=f^\circ SS(\cM)$.
\item If $k=\mathds{C}$, the comparison morphisms 
$$
f^* \Sol \cM^{\an} \lra \Sol f^+\cM^{\an}  \text{ and }\DR f^+\cM^{\an}  \lra  f^* \DR\cM^{\an} 
$$ 
are isomorphisms in the derived category of sheaves.
\end{enumerate}
\end{thm}

\subsection{Meromorphic flat connections}
Let $X$ be a smooth algebraic variety over $k$.
A \textit{flat connection on $X$} or \textit{module with integrable connection on $X$} is a $\cD_X$-module $\cE:=(E,\nabla)$ whose underlying $\cO_X$-module $E$ is a vector bundle of finite rank on $X$.
We denote by $\MIC(X)$ the category of flat connections on $X$.
\\ \indent
If $D$ is a divisor in $X$, a \textit{meromorphic flat connection on $X$ with poles along $D$} is a $\cD_X$-module $\cM:=(M,\nabla)$ whose underlying $\cO_X$-module $M$ is a locally free sheaf of $\cO_X(\ast D)$-modules of finite rank.
We denote by $\MIC(X,D)$ the category of meromorphic flat connections on $X$ with poles along $D$. 

\begin{rem}
The above definitions make sense in the analytic setting, where the same notations will be used. 
\end{rem}
\begin{rem}
Meromorphic flat connections are holonomic $\cD$-modules. 
\end{rem}
In the algebraic setting, there is no difference between flat connections and meromorphic flat connections.
This is expressed by the following proposition \cite[5.3.1]{HTT}.

\begin{prop}\label{meromorphic_vs_non_meromorphic}
Let $(X,D)$ be a pair over $k$.
Let $j: U=X\setminus D\to X$ be the inclusion.
Then $(j_*, j^*)$ induce an equivalence of categories between $\MIC(U)$ and $\MIC(X,D)$.
\end{prop}
Since all the action in this paper happens at infinity, we will use the meromorphic viewpoint.
The next lemma says that this does not make any difference for cohomology.

\begin{lem}\label{same_coho}
Let $(X,D)$ be a pair over $k$.
Put $U:=X\setminus D$ and let $j: U\to X$ be the inclusion.
Let $\cE$ be an object in $\MIC(U)$.
Then the canonical restriction morphism 
$$
R\Gamma(X,\DR j_* \cE)\lra R\Gamma(U,\DR \cE)
$$
is an isomorphism in the derived category of vector spaces over $k$.
\end{lem}
\begin{proof}
We have 
$$
R\Gamma(U,\DR \cE)  \simeq R\Gamma(X,Rj_*\DR \cE)  \simeq R\Gamma(X,j_*\DR \cE) \simeq    R\Gamma(X,\DR j_*\cE)               
$$
where the second equality comes from the fact that $j$ is an affine morphism, and where the last equality follows from $j_*\DR \cE\simeq \DR j_*\cE$.

\end{proof}

\begin{prop}\label{RHom_alg_DR}
Let $(X,D)$ be a pair over $k$.
Let $\cM_1$ and $\cM_2$ be objects of $\MIC(X,D)$.
Then, there is a canonical isomorphism 
$$
\RHom_{\cD_X}(\cM_1,\cM_2)\lra R\Gamma(X,\DR\cHom(\cM_1,\cM_2))
$$
in the derived category of $k$-vector spaces.
\end{prop}
\begin{proof}
Combining \cite[VII 9.8]{Borel} and \cite[VI 5.3.2]{Borel} yields a canonical iso\-morphism
$$
\RHom_{\cD_X}(\cM_1,\cM_2)\simeq R\Gamma(X,\DR\cM_1^{\vee}\otimes_{\cO_X}^{L}\cM_2)
$$
where $\cM_1^{\vee}$ denotes the $\cD_X$-module dual to $\cM_1$.
Since $\cM_2$ is a locally free $\cO_X$-module localized along $D$, we have 
$$
\cM_1^{\vee}\otimes_{\cO_X}^{L}\cM_2\simeq \cM_1^{\vee}\otimes_{\cO_X} \cM_2\simeq\cM_1^{\vee}(\ast D)\otimes_{\cO_X} \cM_2
$$
Observe that $\cM_1^{\vee}(\ast D)$ is an object of $\MIC(X,D)$.
Let $\cM_1^*$ be the meromorphic connection dual to $\cM_1$.
Then \cite[2.6.10]{HTT} ensures that the restriction of $\cM_1^{\vee}(\ast D)$  and $\cM_1^*$ to  $U:=X\setminus D$ are canonically isomorphic.
From \cref{meromorphic_vs_non_meromorphic}, we deduce that $\cM_1^{\vee}(\ast D)$  and $\cM_1^*$  are canonically isomorphic.
We thus conclude the proof of \cref{RHom_alg_DR} using the identification 
$\cM_1^{\ast}\otimes_{\cO_X} \cM_2\simeq \cHom(\cM_1,\cM_2)$.
\end{proof}

\subsection{Good formal structure for connections}
In the next definition, we follow \cite[Def 2.1.3]{Kedlaya3} and fix a field $k$ of characteristic $0$.

\begin{defin} \label{defin_good_formal_structure}
Let $(X,D)$ be a normal crossing pair over $k$.
Let $\cM$ be an object of $\MIC(X,D)$. 
We say that  $\cM$ has \textit{good formal structure at $x\in D$} if there is a decomposition 
\begin{equation}\label{decomposition}
\cM_x\otimes_{\mathcal O_{X,x}}\mathcal S\cong \bigoplus_{\alpha \in I} \mathcal E^{\varphi_\alpha}\otimes_{\mathcal S} \mathcal R_{\alpha},
\end{equation}
where 
\begin{itemize}
\item[(1)]
the ring $\widehat{\mathcal {O}}_{X,x}(\ast D)$ is the localization along $D$ of the completion of ${\mathcal {O}}_{X,x}$ along the intersection of the irreducible components of $D$ containing $x$ and $\mathcal S$ is a finite \'etale $\widehat{\mathcal {O}}_{X,x}(\ast D)$-algebra;
\item[(2)]
$I$ is a finite set; 
\item[(3)]
each $\mathcal R_{\alpha}$ is a regular differential module over $\mathcal S$;
\item[(4)] $\varphi_\alpha$ $(\alpha\in I)$ are elements in $\mathcal S$ satisfying that, if $\varphi_\alpha$ does not lie in the integral closure $\mathcal S_0$ of $\mathcal O_{X,x}$ in $\mathcal S$, then $\varphi_\alpha$ is a unit of $\mathcal S$ and $\varphi_\alpha^{-1}\in \mathcal S_0$.
Furthermore, if $\varphi_\alpha-\varphi_\beta $ does not lie in $\mathcal S_0$, then $\varphi_\alpha-\varphi_\beta$ is a unit of $\mathcal S$ and $(\varphi_\alpha-\varphi_\beta)^{-1}\in \mathcal S_0$.
\end{itemize} 
\end{defin}

\begin{defin}
Let $\cM$ be an object of $\MIC(X,D)$. 
We say that a point $x$ of $D$ is a \textit{turning point of $\cM$} if $\cM$ does not admit a good formal structure at $x$. 
The set of turning points of $\cM$ is called the \textit{turning locus of $\cM$}.
We denote it by $\TL(\cM)$.
If $\TL(\cM)$ is empty, we say that \textit{$\cM$ admits good formal structure along $D$}.

\end{defin}

\begin{rem}\label{regular_base_change}
The formation of the turning locus commutes with regular base change. 
See \cite[4.3.2]{Kedlaya3} for a proof.
\end{rem}

The compatibility of good formal structure and pull-back follows from \cref{defin_good_formal_structure} :

\begin{prop}\label{pullback_good_is_good}
Let $f: (Y,E)\to (X,D)$ be a morphism of normal crossing pairs over $k$.
Let $\cM$ be an object of $\MIC(X,D)$. 
Let $x$ be a point in $E$ such that $\cM$ has good formal structure at $f(x)$.
Then, $f^+\cM$ has good formal structure at $x$.
In particular, if $\cM$ has good formal structure, so does $f^+\cM$.
\end{prop}

Purity of the turning locus was proved by André when $D$ is smooth \cite[3.4.3]{andre} and by Kedlaya when $D$ has normal crossing \cite[2.3.1]{Kedlaya3}.

\begin{thm}\label{purity}
Let $(X,D)$ be a normal crossing pair over $k$.
Then, for any object $\cM$ of $\MIC(X,D)$, the  turning locus of $\cM$ is a closed subscheme of $D$ of pure codimension $1$.
\end{thm}

The following fundamental theorem was proved by Kedlaya \cite[Th. 8.1.3]{Kedlaya2} and Mochizuki  \cite[Th. 19.5]{Mochizuki1}  via two completely different methods.

\begin{thm}\label{KM_theorem}
Let $(X,D)$ be a pair over $k$. 
Let $\cM$ be an object of $\MIC(X,D)$. 
Then, there exists a morphism of smooth algebraic varieties $f:Y\to X$ obtained as a composition of blow-up with smooth centres above $\TL(\cM)$ such that $f^{-1}(D)$ is a normal crossing divisor of $Y$ and $f^+\cM$ admits good formal structure along $f^{-1}(D)$.
\end{thm}

\subsection{Irregularity number}
Let $(X,D)$ be a normal crossing pair over $k$.
Let $\cM$ be an object of $\MIC(X,D)$. 
In the setting of \cref{defin_good_formal_structure}, good formal structure always holds at the generic point $\eta$ of an irreducible component $Z$ of $D$.
In that case, $\widehat{\mathcal {O}}_{X,\eta}(\ast D)$ has the form  $K\llbracket x_1,\dots ,x_d,y\rrbracket [y^{-1}]$ where $K$ is the function field of $Z$ and  $S$ has the form $L\llbracket x_1,\dots, x_d,y^{1/m}\rrbracket [(y^{1/m})^{-1}]$ where $m\geq 1$ and where $L$ is a finite extension of $K$.
\begin{defin}\label{defin_irregularity_number}
For an object $\cM$ in $\MIC(X,D)$, the \textit{generic slopes of $\cM$ along $Z$ } are the poles orders in the $y$-variable  of the $\varphi_\alpha$ contributing to (\ref{decomposition}).
Let us denote them by $\ord_y \varphi_\alpha$ where $\alpha \in I$ and put
$$
r(Z,\cM):=\max_{\alpha \in I} \ord_y \varphi_\alpha
$$
The \textit{generic irregularity number of $\cM$ along $Z$} is defined as 
$$
\irr(Z,\cM):=\sum_{\alpha \in I}   (\ord_y \varphi_\alpha )\cdot \rank \cR_\alpha
$$
\end{defin}

\begin{rem}
In the setting of \cref{defin_irregularity_number}, the generic irregularity number of $\cM$ along $Z$ is an integer independent of the choice of the finite étale $\widehat{\mathcal {O}}_{X,\eta}(\ast D)$-algebra $S$.
\end{rem}

In general, irregularity numbers and irregularity complexes are related through a theorem of  Malgrange \cite[Th 1.4]{Mal71}.
We record it  in the following form useful for us :

\begin{lem}\label{irr_and_Irr}
Let $(X,D)$ be an analytic pair.
Let $Z$ be an irreducible component of $D$.
Let $\cM$ be an object of $\MIC(X,D)$. 
Then, the generic irregularity number of $\cM$ along $Z$ is the generic rank along $Z$ of the perverse complex $\Irr^*_{D} \cM$.
\end{lem}

\subsection{$b$-divisors}\label{b-bivisor}
Let $X$ be a smooth variety over $k$. 
A \textit{modification of $X$} is the datum of a smooth variety $Y$ over $k$ and a map  $Y\to X$ which is proper, dominant and an isomorphism away from a nowhere dense closed subset of $X$. 
Following \cite{Kedlaya3}, we introduce the following definition :
\begin{defin}
The group of integral $b$-divisors $\bDiv(X)$ of $X$ is the limit
$$
\varprojlim_{Y\rightarrow X} \Div(Y)
$$
where $Y\to X$ runs over the category of modifications of $X$ and where the transition maps are push-forward.
\end{defin}
For $D$ in $\bDiv(X)$, we denote by $D(Y)$ the component of $D$ along a modification $Y\to X$.
For every irreducible divisor  $E$ of $Y$, we denote by $m(E,D)$ the multiplicity of $D(Y)$ along $E$.
\begin{defin}
The group of integral Cartier $b$-divisors $\bCDiv(X)$ on $X$ is the colimit
$$
\varinjlim_{Y\rightarrow X}\CDiv(Y)
$$ 
where $Y\to X$ runs over the category of  modifications of $X$ and where the transition maps are pull-back.
\end{defin}
Let $\ZR^{\divis}(X)$ be the subset of the Zariski-Riemann space of $X$ consisting of divisorial valuations centred on $X$. 
As explained in \cite{Kedlaya3}, the group of integral $b$-divisors identifies with the set of functions $m : \ZR^{\divis}(X)\longrightarrow \mathds{Z}$ such that for every  modification $Y\to X$, there is only a finite number of divisorial valuations $v\in \ZR^{\divis}(X)$ centred at an irreducible divisor of $Y$ such that $m(v)\neq 0$. 
Thus, the order on $\mathds{Z}$ induces an order $\leq$ on  $\bDiv(X)$.
Furthermore, there is a canonical injective map
$$
\bCDiv(X) \lra \bDiv(X)
$$

One of the main player of this paper is  a Cartier $b$-divisor with an extra property called \textit{nef} in \cite{Kedlaya3}.
In view of Lemma 1.4.9 and 1.4.10 from \cite{Kedlaya3}, we can define nef Cartier $b$-divisor as follows :

\begin{defin}
Let $X$ be a smooth variety over $k$. 
Let $D$ be a Cartier $b$-divisor on $X$.
We say that $D$ is \textit{nef} if for every modification $Y\to X$, we have $D\leq D(Y)$ in $\bDiv(X)$ where
$D(Y)$ is viewed as a $b$-divisor via $\bCDiv(X) \to \bDiv(X)$.
\end{defin}

\subsection{The Irregularity $b$-divisor}
Let $(X,D)$ be a pair over $k$.
We recall the following definition from \cite[3.1.1]{Kedlaya3}.
\begin{defin}\label{irr_b_divisor}
Let $\cM$ be an object of $\MIC(X,D)$.
We denote by  $\Irr \cM$ the unique $b$-divisor on $X$ such that for every  modification $p: Y\to X$, the multiplicity of $\Irr \cM$ along an irreducible divisor $E$ of $Y$ is the generic irregularity $\irr(E,p^+\cM)$ of $p^{+}\cM$ along $E$. 
We put
$$
\Irr(Y,p^+\cM):=(\Irr \cM)(Y) \text{ in } \Div(Y)
$$
\end{defin}

The following lemma is obvious :
\begin{lem}\label{Irr_divisor_base_change}
Let $(X,D)$ be a pair over $k$.
Let $k\subset K$ be a field extension.
Let $\cM$ be an object of $\MIC(X,D)$.
Then, $\Irr(X_K,\cM_K)=\Irr(X,\cM)_K$.
\end{lem}

\begin{rem}\label{R_divisor}
In the setting of \cref{irr_b_divisor}, one can also define $R(\cM)$ the unique rational $b$-divisor on $X$ such that for every  modification $p: Y\to X$, the multiplicity of $R(\cM)$ along an irreducible divisor $E$ of $Y$ is the highest generic slope $r(E,p^+\cM)\in \mathds{Q}_{\geq 0}$ of $p^{+}\cM$ along $E$, and put $R(Y,p^+\cM):=(R(\cM))(Y) \text{ in } \Div(Y)_{\mathds{Q}}$.
\end{rem}

\begin{rem}
Equivalently, the $b$-divisors $R(\cM)$ and $\Irr \cM$ will be viewed as $\mathds{Q}$-valued functions on the set of divisorial valuations of $X$.
\end{rem}

The following theorem is due to Kedlaya \cite[3.2.3]{Kedlaya3}. 
\begin{thm}\label{main_thm_Ked_III}
Let $(X,D)$ be a normal crossing pair over $k$.
For any object $\cM$ of $\MIC(X,D)$, the $b$-divisor $\Irr \cM$  is a nef Cartier $b$-divisor.
In particular, we have
$$
\Irr \cM \leq \Irr(X,\cM)
$$
in $\bDiv(X)$, where $\Irr(X,\cM)$ is viewed as a $b$-divisor via $\bCDiv(X) \to \bDiv(X)$.
\end{thm}

\begin{cor}\label{semi_continuity}
Let $(X,D)$ be a normal crossing pair over $k$.
Let $R$ be an effective divisor of $X$ supported on $D$.
Let $\cM$ be an object in $ \MIC(X,D)$.
Then the following conditions are equivalent;
\begin{enumerate}\itemsep=0.2cm
\item $\Irr(X,\cM)\leq R$ in $\Div(X)$.
\item For every point $0$ in $D$, for every locally closed smooth curve $C \to X$ in $X$ meeting $D$ at $0$ only, we have 
$$
\irr(0,\cM|_{C})\leq (C,R)_0
$$ 
where $(C,R)_0$  denotes the intersection number of $C$ with $R$ at $0$. 
\end{enumerate}
\end{cor}
\begin{proof}
If $(2)$ holds, we get $(1)$ by using generically enough smooth curves transverse to the irreducible components of $D$. 
We now show that $(1)$ implies $(2)$.
To do this, we can suppose that $R=\Irr(X,\cM)$.
Let $0$ be a point in $D$ and let $C \to X$ be a locally closed smooth curve meeting $D$ at $0$ only. 
Let $p  : Y\to X$ be a  modification of $X$ which is an isomorphism above $X\setminus D$. 
By valuative criterion for properness, the immersion $C\to X$ factors uniquely through an immersion $C\to Y$ followed by $p$.
In particular, we have
$$
\irr(0,\cM|_{C})=\irr(0,(p^+\cM)|_{C})
$$
On the other hand, the projection formula yields
$$
(C,\Irr(X,\cM))_0 = ( C,p^* \Irr(X,\cM))_{0}
$$
From \cref{main_thm_Ked_III}, we further have
$$
\Irr(Y,p^+\cM)\leq p^* \Irr(X,\cM)
$$
in $\Div(Y)$.
Hence, the sought-after inequality for $X, \cM, C\to X$ and $R=\Irr(X,\cM)$ follows from the analogous inequality for $Y, p^+\cM, C\to Y$ and $R=\Irr(Y,p^+\cM)$.
From \cref{KM_theorem}, we are thus left to suppose that $\cM$ has good formal structure and that $C$ is transverse to $D$ at $0$. 
In that case, the sought-after inequality is an equality.
\end{proof}

\begin{rem}
For surfaces, \cref{semi_continuity} was proved by Sabbah \cite[3.2.3]{Sabbahdim}.
\end{rem}

The property of having good formal structure can be read from the irregularity $b$-divisor. 
This is due to Kedlaya \cite[3.1.2]{Kedlaya3} as a consequence of \cref{main_thm_Ked_III}.
\begin{thm}\label{bdiv_and_good_dec}
Let $(X,D)$ be a normal crossing pair over $k$.
Let $\cM$ be an object of $\MIC(X,D)$.
Then,  $\cM$   has good formal structure if and only if 
$$
\Irr \cM = \Irr(X,\cM)  \text{ and } \Irr \End\cM = \Irr(X,\End\cM)
$$
in $\bDiv(X)$.
\end{thm}

The proof of \cref{semi_continuity} combined with \cref{bdiv_and_good_dec} yields the following curve-like criterion  for testing good formal structure :
\begin{cor}\label{Curve_and_goodness}
Let $(X,D)$ be a normal crossing pair over $k$.
Then, an object $\cM$ in $\MIC(X,D)$ has good formal structure if and only if for every point $0$ in $D$, for every locally closed smooth curve $C \to X$ in $X$ meeting $D$ at $0$ only, we have 
$$
\irr(0,\cM|_{C})= (C,\Irr(X,\cM))_0  \text{ and }\irr(0,(\End\cM)|_{C})= (C,\Irr(X,\End\cM))_0
$$ 
\end{cor}

Putting \cref{semi_continuity} and \cref{Curve_and_goodness} together yields the following
differential analogue of a result \cite{Hu_log_ram} of the first named author :

\begin{cor}\label{irr_via_sup}
Let $(X,D)$ be a pair over $k$ where $D$ is smooth and irreducible.
Let $\cM$ be an object in $ \MIC(X,D)$.
Then, 
$$
\irr(D,\cM)=\sup_{(0,C)}  \frac{\irr(0,\cM|_{C})}{(C,D)_0}
$$
where $0$ runs over the points of  $D$ and where $C$ runs over all smooth locally closed curves of $X$ meeting $D$ at $0$ only.
\end{cor}

\section{Cohomological and $\chi$-boundedness conjectures}\label{coho_conjecture_section}

\subsection{Connections with bounded irregularity}
\begin{defin}\label{Def_bounded}
Let $(X,D)$ be a normal crossing pair over $k$.
Let $\cM$ be an object in $ \MIC(X,D)$.
Let $R$ be an effective divisor of $X$ supported on $D$.
We say that $\cM$ has \textit{irregularity bounded by $R$} if the equivalent conditions of \cref{semi_continuity} are satisfied. \\ \indent
We denote by $\MIC(X,D,R)$ the full subcategory of $\MIC(X,D)$ consisting in connections with irregularity bounded by $R$. \\ \indent
For an integer $r\geq 0$, we denote by $\MIC_r(X,D,R)$ the full subcategory of $\MIC(X,D,R)$ spanned by connections with rank smaller than $r$.
\end{defin}

\subsection{Bounded irregularity and $\cHom$}

The $\cHom$ construction for connections preserves boundedness of irregularity in the following sense :

\begin{prop}\label{Irr_Hom}
Let $(X,D)$ be a normal crossing pair  over $k$.
Let $R$ be an effective divisor of $X$ supported on $D$.
Let $r\geq 0$ be an integer.
Let $\cM_1$ 	and $\cM_2$ be objects of $\MIC_r(X,D,R)$.
Then, $\cHom(\cM_1,\cM_2)$ is an object of $\MIC_{r^2}(X,D,2r^2\cdot R)$.
\end{prop}
\begin{proof}
We can suppose that $D$ is irreducible with generic point $\eta$.
Let $(\varphi_\alpha,\cR_\alpha), \alpha\in I$ and $(\psi_\beta,\cS_{\beta}), \beta\in J$  be the constituents of the good formal structures of $\cM_1$ and $\cM_2$  at $\eta$  as in (\ref{decomposition}).
With the notations from \cref{defin_irregularity_number}, we have
$$
\irr(D,\cHom(\cM_1,\cM_2))=\sum_{I\times J}  (\ord_y \psi_\beta-\varphi_\alpha )\cdot (\rank \cR_\alpha)\cdot (\rank \cS_\beta)  
$$
Since $(\ord_y \psi_\beta-\varphi_\alpha )\leq (\ord_y \varphi_\alpha )+(\ord_y \psi_\beta )$, we deduce 
\begin{align*}
\irr(D,\cHom(\cM_1,\cM_2))
& \leq |J|\rank \cM_2\sum_{I }  (\ord_y \varphi_\alpha )\cdot (\rank \cR_\alpha)+|I|\rank \cM_1    \sum_{ J}  (\ord_y \psi_\beta )\cdot (\rank \cS_\beta)\\
& \leq  (\rank\cM_2)^2 \cdot  \irr(D,\cM_1) +(\rank\cM_1)^2 \cdot  \irr(D,\cM_2) 
\end{align*}
The proof of \cref{Irr_Hom} thus follows.
\end{proof}

\subsection{Bounded irregularity and pull-back}

Bounded irregularity behaves well with respect to inverse image.

\begin{prop}\label{non_char_restriction_bounded_irr}
Let $f: (Y,E)\to (X,D)$ be a morphism of normal crossing pairs over $k$.
Let $R$ be an effective divisor of $X$ supported on $D$.
Then, for every object $\cM$ in $\MIC(X,D,R)$, the connection $f^+\cM$ has irregularity bounded by $f^*R$.
\end{prop}
\begin{proof}
By writing $f$ as a composition of a closed immersion followed by a smooth morphism, we are left to suppose that $f$ is either a closed immersion or a smooth morphism.
In both case, the sought-after bound is local around the maximal points of $E$, so we can suppose in each case that $E$ is smooth and irreducible.\\ \indent
Let us assume that $f$ is a smooth morphism.
Then, the pull-back $E\to D$ is smooth.
Since $E$ is smooth, so is the open set $f(E)$ of $D$.
Thus, $f(E)$ lies in a unique irreducible component $T$ of $D$.
Since $f$ is smooth, so is the pull-back $E\to T$.
Hence, the generic point $\eta_E$ of $E$ lies above the generic point $\eta_T$ of $T$ and a uniformizer of $\cO_{X,\eta_T}$ pulls-back to a uniformizer of $\cO_{Y,\eta_E}$.
Thus, the irregularity numbers of $\cM|_{\eta_T}$ and $(f^+\cM)|_{\eta_E}$ are equal.
Hence, $\Irr(Y,f^+\cM)=f^*\Irr(X,\cM)$ and \cref{non_char_restriction_bounded_irr} is proved if $f$ is smooth.
\\ \indent
We now assume that $f$ is a closed immersion with $E$ smooth and irreducible.
Let $0$ be a point in $E$ and let 
$C\to Y$ be a locally closed smooth curve of $Y$ meeting $E$ at $0$ only.
\cref{semi_continuity} applied to $f(C)$ and $\cM$ on $X$ gives
$$
\irr(0,(f^+\cM)|_{C})=\irr(f(0),\cM|_{f(C)})\leq   (f(C),R)_{f(0)}=  (C,f^*R)_0 
$$
where the last equality follows from the projection formula.
Hence, the condition $(2)$ of \cref{semi_continuity} is satisfied and \cref{non_char_restriction_bounded_irr} follows.
\end{proof}

\subsection{Bounded irregularity and change of compactification I}

Bounded irregularity behaves well with respect to compactification by normal crossing pairs.

\begin{prop}\label{change_compactification}
Let $U$ be a smooth variety over $k$. 
Let $(X_1,D_1)$ and $(X_2,D_2)$ be proper normal crossing pairs over $k$ compactifying $U$, that is $U=X_1\setminus D_1=X_2\setminus D_2$.
Let $j_1 : U\longrightarrow X_1$ and $j_2 : U\longrightarrow X_2$ be the inclusions. 
Let $R_1$ be an effective divisor of $X_1$ supported on $D_1$. 
Then, there exists an effective divisor $R_2$ of $X_2$ supported on $D_2$ depending only on $(X_1,D_1)$, $(X_2,D_2)$ and linearly on $R_1$ such that for every object $\cM_1$ of $\MIC(X_1,D_1,R_1)$, the connection $j_{2*}j_1^*\cM_1$ lies in $\MIC(X_2,D_2,R_2)$.
\end{prop}
\begin{proof}
Using resolution of singularities, we construct a pair $(X_3,D_3)$ dominating $(X_1,D_1)$ and $(X_2,D_2)$ such that $D_3$  has normal crossing. 
That is, there is a commutative diagram
$$
\xymatrix{
   &  U     \ar[ld]_-{j_1} \ar[d]^-{j_3}  \ar[rd]^-{j_2}  & \\
    X_1     & \ar[l]^-{p_1}  X_3 \ar[r]_-{p_2}    &     X_2  
}
$$
with modifications as horizontal arrows where $D_3$ lies over $D_1$ and $D_2$. 
Let $\cM_1\in \MIC(X_1,D_1,R_1)$. 
Put $\cM_2=j_{2*}j_1^*\cM_1$ and $\cM_3=j_{3*}j_1^*\cM_1$. 
From \cref{meromorphic_vs_non_meromorphic}, we have 
$$
p^+_1\cM_1=\cM_3=p^+_2\cM_2
$$ 
Let $E$ be an irreducible component of $D_2$. 
We have to bound $\irr(E,\cM_2)$ by means of $(X_i,D_i)$, $i=1,2$ and $R_1$ only.
Let $E'$ be the strict transform of $E$ in $X_3$. 
Then, 
$$
\irr(E,\cM_2)=\irr(E',p_2^+\cM_2)=\irr(E',\cM_3)=\irr(E',p_1^+\cM_1)
$$
From \cref{main_thm_Ked_III}, we further have 
$$
\Irr(X_3,p_1^+\cM_1)\leq p_1^*\Irr(X_1,\cM_1) \leq p_1^* R_1
$$
We deduce 
$$
\irr(E,\cM_2)\leq m(E',p_1^* R_1)=m(E,p_{2 \ast} p_1^* R_1)
$$
Hence, $R_2=p_{2 \ast} p_1^* R_1$ depends linearly on $R_1$ and satisfies the condition of  \cref{change_compactification}.
\end{proof}

\subsection{Cohomological boundedness conjecture}

Recall that if $G$ is an abelian group and $d\geq 0$ an integer, a \textit{polynomial $C : G\to \mathds{Z}$ of degree at most $d$} is an element of $\Sym^{\bullet}_{\mathds{Z}}(G^{\vee})$ of degree at most $d$, where $G^{\vee}$ is the abelian group dual to $G$.

\begin{conj}\label{conjecture_coho}
Let $(X,D)$ be a projective normal crossing pair of dimension $d$ over $k$.
There exists a  polynomial  $C : \Div(X,D)\oplus \mathds{Z}\to \mathds{Z}$ of degree at most $d$, affine in the last variable  such that for every  effective divisor $R$ of $X$ supported on $D$, for every integer $r\geq 0$ and every object $\cM$ of $\MIC_r(X,D,R)$, we have
$$
\dim H^{*}(X,\DR \cM)\leq C(R,r)
$$
\end{conj}

\begin{rem}\label{Rem1}
If a polynomial $C : \Div(X,D)\oplus \mathds{Z}\to \mathds{Z}$ as above exists, we say that \textit{cohomological boundedness holds for $(X,D)$ with bound $C$}.
\end{rem}

\begin{rem}\label{Rem2}
Let $d\geq 0$ be an integer.
Let $k$ be a field of characteristic $0$.
We say that \textit{cohomological boundedness holds in dimension $d$ over $k$} if it holds for every choice of $(X,D)$ as above with $(X,D)$ defined over $k$  and $\dim X=d$.
\end{rem}

\begin{rem}\label{Rem3}
Let $d\geq 0$ be an integer.
We say that \textit{cohomological boundedness holds in dimension $d$} if it holds in dimension $d$ over every field of characteristic $0$.
\end{rem}

\begin{lem}\label{curve_case}
Cohomological boundedness holds in dimension $1$.
\end{lem}
\begin{proof}
From \cref{coho_bound_reduction_complex}, it is enough to treat the case of a smooth connected projective curve $X$ of genus $g$ over $\mathds{C}$.
Let $D$ be a reduced divisor of $X$.
Let $\cM$ be an object of $\MIC(X,D)$.
From \cref{GAGA_DR} and \cref{dualityDR/Sol}, we are left to bound the cohomology of $\Sol(\cM^{\an})$.
Let $j : U=X\setminus D \to X$ be the open immersion and put $\cL :=\Sol(\cM^{\an})|_{U(\mathds{C})}$.
Observe that $\cL$ is a local system of rank $\rk \cM$ on $U(\mathds{C})$.
From \cref{irr_and_Irr}, there is a distinguished triangle
$$
\xymatrix{
   j_!\cL  \ar[r]    &   \Sol(\cM^{\an})  \ar[r]      &    \bigoplus_{P\in D}  \mathds{C}^{\irr(P,\cM)}_P[-1]\ar[r]^-{+1}      &
}
$$
We have  $H^{0}(X(\mathds{C}),   j_!\cL  )\simeq 0$.
If $\cL^*$ denotes the dual of $\cL$, Poincaré duality further gives 
$$
\dim H^{2}(X(\mathds{C}),   j_!\cL  )=\dim H^{0}(U(\mathds{C}),  \cL^{\ast}  ) \leq \rk \cM
$$
On the other hand, $\chi_c(U(\mathds{C}),  \cL  )=\rk \cM \cdot  \chi(U(\mathds{C}),  \mathds{C}  )=\rk \cM \cdot (2-2g-|D|)$.
Then, \cref{curve_case} follows from the long exact sequence in cohomology induced by the above distinguished triangle.
\end{proof}

\subsection{The $\chi$-boundedness conjecture}

We formulate an a priori weaker version of the cohomological boundedness conjecture.
\begin{conj}\label{conjecture_chi}
Let $(X,D)$ be a projective normal crossing pair of dimension $d$ over $k$.
There exists a  polynomial  $C : \Div(X,D)\oplus \mathds{Z}\to \mathds{Z}$ of degree at most $d$,  affine in the last variable  such that for every  effective divisor $R$ of $X$ supported on $D$, for every integer $r\geq 0$ and every object $\cM$ of $\MIC_r(X,D,R)$, we have
$$
|\chi(X,\DR \cM) |\leq C(R,r)
$$
\end{conj}

\begin{rem}
We  adopt for the $\chi$-boundedness conjecture the terminology from \cref{Rem1}, \cref{Rem2} and \cref{Rem3} for the cohomological boundedness conjecture.  
\end{rem}

\begin{rem}\label{reduction_chi_to_C}
Let $d\geq 0$ be an integer.
Similarly as in \cref{coho_bound_reduction_complex}, the $\chi$-boundedness conjecture holds in dimension $d$ if it holds in dimension $d$ over $\mathds{C}$.

\end{rem}

Cohomological boundedness trivially implies  $\chi$-boundedness. 
The goal of the next subsection is to show that they are equivalent.

\subsection{The cohomological and $\chi$-boundedness conjectures are equivalent}

\begin{prop}\label{non_char_restriction_coho_DR}
Let $\mathds{P}$ be a projective space over  $k$.
Let $X$ be a smooth subvariety of $\mathds{P}$.
Let $\cM$ be an holonomic $\cD_X$-module.
Let $H\in \mathds{P}^{\vee}(k)$ be a hyperplane such that $X\cap H$ is smooth and $X\cap H \to X$ is non-characteristic for $\cM$.
Then the canonical comparison morphism 
$$
H^n(X,\DR \cM)\lra H^n(X\cap H,\DR (\cM|_{X\cap H}))
$$
is an isomorphism if $n<\dim X-1$ and is injective if $n=\dim X-1$.
\end{prop}
\begin{proof}
From  \cref{coho_change_base_field},
we reduce to the case where $k=\mathds{C}$.
From \cref{GAGA_DR}, we are left to prove a variant of \cref{non_char_restriction_coho_DR}
where now   $X$ is the analytification of a smooth complex algebraic variety and where $\cM$ is the analytification of an holonomic $\cD$-module.
Then  \cref{DR_perverse} ensures that $\DR \cM[\dim X]$ is a perverse complex. 
From the Lefschetz hyperplane theorem for perverse complexes  \cite[2.4.2]{Cataldo}, the canonical comparison morphism 
$$
H^n(X,\DR \cM)\lra H^n(X\cap H,(\DR \cM)|_{X\cap H})
$$
is an isomorphism if $n<\dim X-1$ and is injective if $n=\dim X-1$.
Since $X\cap H \to X$ is non-characteristic for $\cM$, we conclude the proof of \cref{non_char_restriction_coho_DR} using \cref{Cauchy-Kowaleska}.
\end{proof}

\begin{rem}\label{rem_non_char_restriction_coho_DR}
When $X$ is a projective complex manifold,  \cref{non_char_restriction_coho_DR} is also valid via the same reasoning with either the analytic  De Rham complex or the solution complex.
\end{rem}

\begin{cor}\label{bound_in_terms_of_chi}
In the setting of \cref{non_char_restriction_coho_DR}, assume the existence of  $C\geq 0$ with
$$
\dim H^{\ast}(X\cap H,\DR (\cM|_{X\cap H}))  \leq C
$$
Then, if $d$ denotes the dimension of $X$, we have
$$
\dim H^{\ast}(X,\DR \cM)\leq |\chi(X,\DR \cM)| + 4d \cdot C
$$
\end{cor}
\begin{proof}
We can suppose that $k=\mathds{C}$.
From \cref{non_char_restriction_coho_DR}, we have for $n<d$,
$$
\dim H^{n}(X,\DR \cM)\leq C
$$
Assume that $n>d$.
Then,
\begin{align*}
\dim H^{n}(X,\DR \cM) &= \dim H^{n}(X(\mathds{C}),\DR \cM^{\an})   \\
                              & =\dim   H^{2d-n}(X(\mathds{C}),\Sol \cM^{\an})\\
                              & \leq \dim H^{2d-n}((X\cap H)(\mathds{C}),\Sol (\cM^{\an}|_{(X\cap H)(\mathds{C})}))\\
                              & \leq \dim H^{n-2}((X\cap H)(\mathds{C}),\DR(\cM^{\an}|_{(X\cap H)(\mathds{C})}))   \\
                              & \leq \dim H^{n-2}(X\cap H,\DR(\cM|_{X\cap H}))  \leq C
\end{align*}
where the first and fifth inequality follows from \cref{GAGA_DR}, the second follows from \cref{dualityDR/Sol} and Poincaré-Verdier duality on $X(\mathds{C})$, the third follows from \cref{rem_non_char_restriction_coho_DR} since $2d-n<d$ and the fourth inequality follows from \cref{dualityDR/Sol} and Poincaré-Verdier duality on $(X\cap H)(\mathds{C})$.
Furthermore, \cref{bound_cohomology} yields 
\begin{align*}
\dim H^{d}(X,\DR \cM)&  \leq  |\chi(X,\DR \cM)| + \displaystyle{\sum_{n\neq d}}   \dim H^{n}(X,\DR \cM)\\
                                      & \leq |\chi(X,\DR \cM)|  + 2d \cdot C
\end{align*}
Putting all the above inequalities together finishes the proof of \cref{bound_in_terms_of_chi}.
\end{proof}

\begin{prop}\label{reduction_to_bound_chi}
Let $(X,D)$ be a projective normal crossing pair of dimension $d$ over $k$.
Let $X\longrightarrow \mathds{P}$ be a closed immersion in some projective space. 
Assume that $\chi$-boundedness holds for $(X,D)$ with bound $K$ and that cohomological boundedness holds with bound $C$ for the generic hyperplane section  $(X_\eta,D_\eta)$ of $(X,D)$.
Then, cohomological boundedness holds for $(X,D)$ with bound $K+4d\cdot C'$ where $C'$ is the composition of $C$ with the linear map $\Div(X,D)\oplus \mathds{Z}\to \Div(X_\eta,D_\eta)\oplus \mathds{Z}$ deduced from $X_\eta\to X$. 
\end{prop}
\begin{proof}
Let $R$ be an effective divisor of $X$ supported on $D$. 
Let $r$ be an integer. 
By Bertini's theorem, there exists a dense open subset $V$ in $\mathds{P}^{\vee}$ such that for every hyperplane $H\in V(k)$, the pair $(X\cap H, D\cap H)$ is smooth  of dimension $d-1$ and $D\cap H$ has normal crossing.
Let $\eta $ be the generic point of $\mathds{P}^{\vee}$. 
Following the notations from \cref{uni_hyperplane}, consider the commutative diagram with cartesian squares 
$$
\xymatrix{
X_{\eta }\ar[r]  \ar[d]  &  X_V   \ar[r]  \ar[d]   & X_Q   \ar[r] \ar[d]   &     X \ar[d]\\
           Q_\eta     \ar[r] \ar[d]            &          Q_V    \ar[r] \ar[d]       &         Q      \ar[r]   \ar[d]     & \mathds{P}\\
\eta \ar[r]          &  V \ar[r]       & \mathds{P}^{\vee}   & 
}
$$
Let $\pi : X_Q\longrightarrow \mathds{P}^{\vee}$ be the composition morphism and let $\pi_V : X_V\longrightarrow V$ and $\pi_{\eta} : X_{\eta}\longrightarrow \eta$ be its pull-back above $V$ and $\eta$ respectively.
Let $D_V$ and $R_V$ be the pull-back of $D$ and $R$ to $X_V$ and let $D_\eta$ and $R_\eta$ be their pull-back  to $X_\eta$. 
Observe that $(X_\eta,D_\eta)$ is a pair of dimension $d-1$ over $\eta$ where $D_{\eta}$ has normal crossing. 
Since cohomological boundedness  holds for $(X_\eta,D_\eta)$ with bound $C$, we have for every object $\cN_\eta$ in $\MIC_r(X_\eta,D_\eta ,R_\eta)$,
$$
\dim H^{\ast}(X_\eta,\DR \cN_\eta)\leq C(R_\eta,r)
$$ 
On the other hand, for every  object $\cN$ in $\MIC_r(X_V,D_V ,R_V)$ we have
$$
(\pi _{V+}\cN)_\eta \simeq \pi _{\eta+}\cN_\eta\simeq  R\Gamma(X_\eta,\DR \cN_\eta) [d-1]
$$
Thus, the space $H^n(X_\eta,\DR \cN_\eta)$ is the generic fibre of the $\cD_V$-module 
$$
\mathcal{H}^{n-d+1} \pi _{V+}\cN
$$ 
for every $n\geq 0$.
Let $\cM$ be an object in $\MIC_r(X,D,R)$. 
From \cref{non_char_restriction_bounded_irr}, the pull-back $\cM_V$ of $\cM$ to $X_V$ lies in $\MIC_r(X_V,D_V ,R_V)$. 
Since $k$ is infinite,  \cref{generic_transversality} ensures the existence of a hyperplane $H\in V(k)$ such that
$X\cap H\to X$ is non-characteristic for $\cM$ and such that the cohomology modules of
$\pi_{V+}\cM_V$ are flat connections in a neighbourhood of $H\in V(k)$.
In particular, the inclusion $i : \{H\} \to V$ is non-characteristic for the cohomology modules of $\pi _{V+}\cM_V$.
Hence, \cref{Cauchy-Kowaleska} combined with base change for $\cD$-modules \cite[1.7.3]{HTT} yields
\begin{align*}
H^{n}(X_\eta,\DR (\cM_V)_\eta)& \simeq i^*\mathcal{H}^{n-d+1} \pi _{V+}\cM_V \\
                                                         &  \simeq    \mathcal{H}^{n-d+1} i^+ \pi _{V+}\cM_V     \\
                                                         & \simeq  \mathcal{H}^n \pi _{H+}(\cM_V)|_{X\cap H}    \\
                                                         & \simeq  H^n(X\cap H, \DR\cM|_{X\cap H} )
\end{align*}
where $\pi _{H} : X\cap H \to \{H\}$ is the pull-back of $\pi$ above $H\in V(k)$.
Thus,
$$
\dim H^\ast(X\cap H, \DR\cM|_{X\cap H} ) \leq C(R_\eta,r)
$$
Then, \cref{bound_in_terms_of_chi} yields
$$
\dim H^{\ast}(X,\DR \cM)\leq |\chi(X,\DR \cM)| + 4d \cdot C(R_\eta,r)\leq K(R,r)+ 4d \cdot C(R_\eta,r)
$$
which concludes the  proof of \cref{reduction_to_bound_chi}.
\end{proof}

\begin{cor}\label{quasi_cor_reduction_to_bound_chi}
Let $d\geq 2$ be an integer.
Then cohomological boundedness holds in dimension $d$   if it holds in dimension $d-1$   and if $\chi$-boundedness holds in dimension $d$.
\end{cor}

From \cref{curve_case} and  \cref{quasi_cor_reduction_to_bound_chi}, we deduce the following
\begin{cor}\label{cor_reduction_to_bound_chi}
The cohomological and $\chi$-boundedness conjectures are equivalent.
\end{cor}

\section{Nearby slopes and boundedness}\label{nearby_section}

\subsection{Resolution relative to a normal crossing divisor}

Let $(X,D)$ be a normal crossing pair over $k$. 
If  $C$ is a  closed subscheme of $X$, we say that \textit{$C$ and $D$ have simultaneously only normal crossing} if for every point $x$ of $C$, there exists a regular system of parameters $(x_1,\dots, x_n)$ for $\cO_{X,x} $ such that $\cI_{D}$ is generated at $x$ by  $x_1\cdots x_k$  and $\cI_{C}$ is generated at $x$ by some monomials in the $x_j$ for $j=1,\dots, n$ (therefore by some $x_j$ if $C$ is smooth at $x$).

\begin{defin}\label{sum_up}
Let $Z$ be a closed subscheme of $X$. 
We say that a blow-up $p : Y\to X$ with center $C$ is \textit{admissible with respect to $(Z,D)$} if the following conditions are satisfied :
\begin{enumerate}\itemsep=0.2cm
\item $C$ is smooth and contained in $Z_{\red}$. 
\item $C$ and $D$ have simultaneously only normal crossing.
\end{enumerate}
Then, we denote by $Z'$ the strict transform of $Z$ and put $D'=p^{-1}(D)\cup p^{-1}(C)$  endowed with its reduced structure.
The pair $(Z^{\prime},D^{\prime})$ is the \textit{transform of $(Z,D)$ by $p : Y\to X$}.
\end{defin}


The form of embedded resolution  needed in this paper is the following theorem below.
For a proof, see Theorem 8.4 and Theorem 8.6 in \cite{BMUniformization}.
\begin{thm}\label{BM}
Let $(X,D)$ be a normal crossing pair over $k$. 
Let $Z$ be a closed subscheme of $X$. 
Then, there exists a composition $p: Y\to X$ of admissible blow-up with respect to the successive transforms of $(Z,D)$ such that the final transform  $(Z',D')$  satisfies that $Z'_{\red}$ is smooth and $(Z',D')$ have simultaneously only normal crossing.
\end{thm}

\begin{defin}
A map  $p: Y\to X$ as above is a \textit{resolution of $Z$ relative to $D$}.
\end{defin}

The following lemma is obvious from the definitions.
\begin{lem}\label{composing_resolution}
Let $(X,D)$ be a normal crossing pair over $k$.
Let $Z,T$ be subschemes of $X$ with $T_{\red}\subset Z_{\red}$.
Then a blow-up of $X$ admissible with respect to $(T,D)$ is admissible with respect to 
$(Z,D)$.
\end{lem}


\begin{lem}\label{resolution_and_principalization}
Let $(X,D)$ be a normal crossing pair over $k$.
Let $Z,T$ be subschemes of $X$ with $T_{\red}\subset Z_{\red}$.
Then, there exists a resolution $p : Y\to X$ of $Z$ relative to $D$ such that the pull-back scheme $p^{-1}(T)$ is an effective Cartier divisor such that $p^{-1}(T)+D'$ has normal crossing, where $(Z',D')$ is the transform of $(Z,D)$ by $p : Y\to X$.
\end{lem}
\begin{proof}
Let $q : X'\to X$  be a blow-up admissible with respect to $(T,D)$.
From \cref{composing_resolution}, the map $q$ is admissible with respect to $(Z,D)$.
Let $(Z',T',D')$ be the transform of $(Z,T,D)$ by $q$.
We argue that if \cref{resolution_and_principalization} holds for $(X',D',Z',T')$, then it holds for $(X,D,Z,T)$.
Indeed, let $\rho : Y'\to X'$ be a resolution of $Z'$ relative to $D'$ such that the pull-back scheme $\rho^{-1}(T')$ is an effective Cartier divisor such that $\rho^{-1}(T')+D''$ has normal crossing, where $(Z'',D'')$ is the transform of $(Z',D')$ by $\rho$.
Then, the composition $\rho\circ q : Y'\to X$ is a resolution  of $Z$ relative to $D$.
On the other hand, we have $q^{-1}(T)=T'\bigcup E$ where $E$ is an effective Cartier divisor. 
Thus $\rho^{-1}(q^{-1}(T))=\rho^{-1}(T')\bigcup \rho^{-1}(E)$ is an effective Cartier divisor.
Observe that $(Z'',D'')$ is also the transform of $(Z,D)$ by $\rho\circ q$.
Since $D''$ contains $\rho^{-1}(E)$, the divisors  $\rho^{-1}(T')+D''$ and $\rho^{-1}q^{-1}(T)+D''$  have the same support. 
Hence, $\rho^{-1}q^{-1}(T)+D''$ has normal crossing and \cref{resolution_and_principalization} is indeed true for $(X,D,Z,T)$.\\ \indent
Using a resolution of $T$ relative to $D$ as given by \cref{BM}, we are thus left to prove \cref{resolution_and_principalization} in the case where 
$T_{\red}$ is smooth and  $T$ and $D$ have simultaneously only normal crossing.
In particular, $T_{\red}$ and $D$ have simultaneously only normal crossing.
Hence, the blow-up of $T_{\red}$ is admissible with respect to $(T,D)$.
We are thus left to prove \cref{resolution_and_principalization} in the case where $T$ is empty.
In that case, any resolution of $Z$ relative to $D$ as given by \cref{BM} does the job and the proof of \cref{resolution_and_principalization} is complete.
\end{proof}

\subsection{Bounded irregularity and change of compactification II}

 \cref{partial_compactification}  below provides a generalization of \cref{change_compactification} where part of the divisor at infinity is being kept while changing compactification. 
Before proving it, we need the following 
\begin{lem}\label{pull_back_resolution_irr_bounded}
Let $(X,D)$ be a normal crossing pair over $k$.
Let $Z$ be a closed subscheme  of $X$.
Let $p : Y\to X$ be a resolution of $Z$ relative to $D$ such that $(p^{-1}(Z))_{\red}$ is a normal crossing divisor.
Put $F:=(p^{-1}(Z\cup D))_{\red}$.
Let $R$ be an effective divisor supported on $D$.
Then for every object $\cM$ of $\MIC(X,D,R)$, the connection $(p^+ \cM)(\ast F)$ lies in $\MIC(Y,F,p^*R)$.
\end{lem}
\begin{proof}
Put $E=f^{-1}(D)$ .
By definition of a resolution relative to a normal crossing divisor, the divisors $E\subset F$ have normal crossing.
In particular, $p : (Y,E)\to (X,D)$ is a morphism of normal crossing pairs.
From \cref{non_char_restriction_bounded_irr}, we deduce that $p^+ \cM$ is an object of $\MIC(Y,E,p^*R)$.
Along a component of $F$ that is not a component of $E$, the generic irregularity of $(p^+ \cM)(\ast F)$ is $0$.
Thus,  $(p^+ \cM)(\ast F)$ is an object of $\MIC(Y,F,p^*R)$ and \cref{pull_back_resolution_irr_bounded} is proved.
\end{proof}

\begin{prop}\label{partial_compactification}
Let $(X,D)$ be a proper normal crossing pair over $k$.
Let $V$ be an open subset of  $X$ and put $D_V:=V\cap D$.
Let $j : (V,D_V) \to (Y,E)$ be a dense open immersion  where $(Y,E)$ is a proper normal crossing pair over $k$ with $Y\setminus E=V\setminus D_V$.
Then for every  effective divisor $R$ supported on $D$, there exists an effective divisor $S$ supported on $E$ depending only on $V$, on  $j : (V,D_V) \to (Y,E)$ and linearly on $R$ such that for every object $\cM$ of $\MIC(X,D,R)$, the connection $j_* \cM|_V$ is an object of $\MIC(Y,E,S)$.
\end{prop}
\begin{proof}
Put $Z:=X\setminus V$.
Let $p : Y'\to X$ be a resolution of $Z$ relative to $D$ such that $(p^{-1}(Z))_{\red}$ is a normal crossing divisor.
Put  $F:=(p^{-1}(Z\cup D))_{\red}$.
In particular, $p : Y\to X$ induces an isomorphism $p^{-1}(V)\to V$.
Furthermore, $Y'\setminus F=V\setminus D_V$.
Hence the normal crossing pair $(Y',F)$ is a compactification of $V\setminus D_V$.
By assumption, $(Y,E)$ is also a compactification of $V\setminus D_V$.
Let $\jmath : V\setminus D_V \to V$ be the inclusion.
Let $R$ be an effective divisor of $X$ supported on $D$.
Let $\cM$ be an object of $\MIC(X,D,R)$.
From \cref{pull_back_resolution_irr_bounded}, the connection $(p^+ \cM)(\ast F)$ is an object of $\MIC(Y',F,p^*R)$.
\cref{change_compactification} applied to $V\setminus D_V$, $(Y',F)$, $(Y,E)$ and $p^*R$ yields the existence of an effective divisor $S$ of $Y$ supported on $E$ depending only on $(Y',F)$, $(Y,E)$ and linearly on $p^*R$ such that 
$$
\cN :=j_*\jmath_* ((p^+ \cM)(\ast F))|_{V\setminus D_V}\simeq j_*\jmath_* \jmath^+\cM|_V\simeq j_* \cM|_V
$$ 
is an object of $\MIC(Y,E,S)$.
We thus have $\cN|_V\simeq \cM|_V$ and $S$ only depends on $V$, on  $j : (V,D_V) \to (Y,E)$ and on $R$.
This concludes the proof of \cref{partial_compactification}.
\end{proof}

\subsection{Resolution and multiplicity  estimate}
If $X$ is a smooth connected variety over $k$, we recall from \cref{b-bivisor} that any Cartier divisor on $X$ can be seen as a $\mathds{Z}$-valued function on $\ZR^{\divis}(X)$ via the injection $\bCDiv(X)\to \bDiv(X)$.
\begin{lem}\label{bound_multiplicity_one_blow_up}
Let $(X,D)$ be a normal crossing pair of dimension $d$ over $k$. 
Let $C$ be a  smooth subscheme of $X$ which has simultaneously only normal crossing with $D$.
Let $v$ be the divisorial valuation associated to the exceptional divisor of the blow-up $p : Y\to X$  along $C$.
Then, $D(v)\leq d$. 
If $D$ is smooth, then $D(v)$ is $0$ or $1$.
\end{lem}
\begin{proof}
The question is étale local around the generic point of $C$.
We can thus suppose the existence of a local system of coordinates $(x_1,\dots, x_d)$  such that 
$\cI_{D}$ is generated by $\prod_{i\in I} x_i$ where $I\subset\{1,\dots, d\}$ and $\cI_{C}$ is generated by some  $x_j$ for $j\in J\subset\{1,\dots, d\}$.
Then
$$
\cI_D=\left(\prod_{i\in I\cap J} x_i\prod_{i\in I\setminus I\cap J} x_i \right)
$$
Since $D(v)$ is simply the multiplicity of $p^*D$ along the exceptional divisor, we deduce
$$
D(v)=|I\cap J|\leq d
$$
If $D$ is smooth, then $I\cap J$ is either empty or a singleton depending on whether $D$ contains $C$, and \cref{bound_multiplicity_one_blow_up} follows.
\end{proof}

The  lemma below appeared as Proposition $4.3.2$ in \cite{NearbySlope} with the tacit assumption that $D$ has \textit{simple} normal crossing.
If some components of $D$ are singular, the coefficient $\fdeg R$ is not enough and has to be replaced by $d\cdot (\fdeg R)$ where $d$ is the dimension of the ambient variety.
We spell out the proof here for completeness.\\ \indent

\begin{lem}\label{lemme_from_nearby_slopes}
Let $(X,D)$ be a normal crossing pair of dimension $d$ over $k$. 
Let $R$ be an effective  divisor supported on $D$.
Let $Z$ be an effective Cartier divisor on $X$.
Let  $p: Y\to X$ be a composition of blow-up which are admissible with respect to the successive transforms of $(Z,D)$.
Let $v$ be a divisorial valuation centred at a component of $Z(Y)=p^*Z$.
Then, 
\begin{equation}\label{inequality_resolution}
R(v)\leq d  \cdot  (\fdeg R) \cdot  Z(v)
\end{equation}
where $R$ and $Z$ are viewed as $b$-divisors via $\bCDiv(X)\to \bDiv(X)$.
\end{lem}
\begin{proof}
We argue by induction on the number $n$ of admissible blow-up needed to write $p$.
If $n=0$, there is nothing to prove.
Let us now treat the case of $n+1$ blow-ups.
Let 
$$
\xymatrix{
p :Y \ar[r]^-{q} &  X' \ar[r]^-{r} & X
}
$$
where $r$ is a composition of $n$ blow-up admissible with respect to the successive transforms of $(Z,D)$, and where $q$ is admissible with respect to the final transform $(Z',D')$ of $(Z,D)$ by $r$.
Let $\cC$ be the set of strict transforms in $X'$ of the exceptional divisors that appear in $r$.
For $E\in \cC$, we let $v_E$ be the associated divisorial valuation.
In what follows, the notation $'$ indicates a strict transform in $X'$ and $''$ indicates a strict transform in $Y$.
Write $R= \alpha_1 D_1+\cdots + \alpha_m D_m$ with $\alpha_i>0$ and $Z= \beta_1 Z_1+\cdots +\beta_k Z_k$ with $\beta_i>0$.
We have
$$
Z(X')=\beta_1  Z'_{1}+\cdots +\beta_k Z'_{k}+\displaystyle{\sum_{E\in \mathcal{C}}} Z(v_E) E
$$
with $Z(v_E)>0$ by admissibility of $r$ with respect to the successive transforms of $(Z,D)$, and 
$$
R(X')=\alpha_1  D'_{1}+\cdots +\alpha_k D'_{k}+\displaystyle{\sum_{E\in \mathcal{C}}} R(v_E)  E
$$
We have to detail the effect of $q^*$  on each component contributing to the above equalities.
Let $C$ be the centre of $q$ and let $P$ be its exceptional divisor with associated valuation $v_P$.
Let $E\in \cC$.
Since $C$ and $E$ have simultaneously only normal crossing and since $E$ is smooth, \cref{bound_multiplicity_one_blow_up} gives
$q^{\ast} E=E''+ \epsilon_E P$ where $\epsilon_E\in \{0,1\}$.
Since $C$ and $D'_i$ have simultaneously only normal crossing, \cref{bound_multiplicity_one_blow_up} gives again
$$
q^{\ast} D'_i=D''_i+ \epsilon_i P \text{ where } \epsilon_i\in \{0,\dots,d\}
$$
Furthermore, let us write
$$
q^{\ast} Z'_i=Z''_i+ \mu_i P \text{ where }  \mu_i\in \mathds{N}
$$
We have
$$
Z(Y)=\sum  \beta_i Z''_{i}+\displaystyle{\sum_{E\in \mathcal{C}}} Z(v_E) E''+(\sum \beta_i \mu_i+\sum_{E\in \mathcal{C}} \epsilon_E Z(v_E) )P
$$
and
$$
R(Y)=\sum  \alpha_i D''_{i}+\displaystyle{\sum_{E\in \mathcal{C}}} R(v_E)  E''+(\sum \alpha_i \epsilon_i+\sum_{E\in \mathcal{C}} \epsilon_E R(v_E)   )P
$$
By recursion assumption, we have to check the inequality \eqref{inequality_resolution}  for $v=v_P$. 
By admissibility of $q$ with respect to $(Z',D')$, the centre $C$ lies in $Z'_{\red}$, that is in one of the $Z'_{i}$.
Hence, one of the $\mu_i$ is strictly positive.
Thus, 
\begin{align*}
d\cdot (\fdeg R)\cdot  Z(v_P) & =d\cdot (\fdeg R)\cdot\left(\sum \beta_i \mu_i+\sum \epsilon_E Z(v_E) \right) \\
&\geq d \cdot (\fdeg R)+\sum \epsilon_E \left(d \cdot (\fdeg R)\cdot Z(v_E)\right)  \\ 
&\geq \sum \alpha_i \epsilon_i+\sum \epsilon_E R(v_E)=R(v_P)
\end{align*}
and the proof of \cref{lemme_from_nearby_slopes} is complete.

\end{proof}


\subsection{Nearby slopes for $\cD$-modules}
Let $X$ be a smooth algebraic variety over $k$.
Let $f$ be a non constant function on $X$.
We denote by $\psi_f$ the nearby cycles functor associated to $f$.
For a general reference on nearby cycles, see \cite{MM}.
Inspired by a letter of Deligne to Malgrange \cite{DeligneLettreMalgrange}, the following notion of slopes was defined in \cite{NearbySlope} :

\begin{defin}
Let $\cM$ be an holonomic $\cD_X$-module.
The \textit{nearby slopes of $\mathcal{M}$ associated to $f$} are the rationals $r\geq 0$ such that there exists a  germ $N$ of meromorphic connection at $0$ in $\mathds{A}^1$ with slope $r$ verifying
\begin{equation}\label{annulation0}
\psi_{f}(\mathcal{M}\otimes f^{+}N)\neq 0
\end{equation}
We denote by $\Sl_{f}^{\nb}(\mathcal{M})$ the set of nearby slopes of $\mathcal{M}$ associated to $f$.
\end{defin}
\begin{rem}\label{rem_localization}
Since  $N$ is localized at $0$, the connection $f^+N$ is localized along $\divi f$.
Hence,  $\Sl_{f}^{\nb}(\mathcal{M})=\Sl_{f}^{\nb}(\mathcal{M}(\ast Z))$ where $Z$ denotes the support of $\divi f$.
\end{rem}
\begin{rem}
The nearby slopes are sensitive to the non reduced structure of $\divi f$.
\end{rem}

For $f$ as above, Deligne proved that the set $\Sl_{f}^{\nb}(\mathcal{M})$ is finite. 
The following theorem is the main result of \cite{NearbySlope}.
\begin{thm}
Let $X$ be a smooth algebraic variety over $k$.
Let $\cM$ be an holonomic $\cD_X$-module.
Then, there is an integer $C$ depending only on $\cM$ such that for every non constant function $f$ on $X$, the set $\Sl^{\nb}_f(\mathcal{M})$ is bounded by $C$.
\end{thm}
In this paper, we will need a boundedness of a different kind where $f$ is fixed but where $\cM$ is allowed to vary over the objects of $\MIC(X,D,R)$ where $D$ is a normal crossing divisor and where $R$ is an effective divisor supported on $D$.\\ \indent
Examples of this kind of bound were obtained in the étale setting  in \cite{HuTeyssier} and \cite{Hu_general_Leal} in a semi-stable situation where the horizontal ramification is tame.
In our case, the horizontal irregularity will be allowed to be non trivial.
The price to pay for this generality will be to require that the zero locus of $f$ contains the turning locus of $\cM$.
See \cref{boundedness_general} below.\\ \indent

Nearby slopes are compatible with proper push-forward.
See \cite[Th. 3]{NearbySlope}.

\begin{prop}\label{push_forward_Th}
Let $p : Y\to X$ be a proper morphism of smooth varieties over $k$.
Let $f$ be a non constant function on $X$ such that $fp=0$ is a divisor of $Y$.
Then, for every holonomic $\cD_Y$-module $\cM$, we have 
$$
\Sl^{\nb}_f(p_+\cM)\subset  \Sl^{\nb}_{fp}(\cM)
$$
\end{prop}

The following lemma is useful for dévissages.

\begin{lem}\label{pull_back_push_forward}
Let $p : (Y,E)\to (X,D)$ be a proper morphism of pairs over $k$ such that the induced morphism $Y\setminus E \to X \setminus D$ is an isomorphism.
Let $\cM$ be an object  in $\MIC(X,D)$.
Then, the canonical morphism $p_+ p^{+ }\cM \to \cM$ is an isomorphism.
\end{lem}
\begin{proof}
Since $p$ is proper, \cite[3.6-4]{Mehbsmf} ensures that the $\cD_X$-module $p_+ p^{+ }\cM$ is an object of $\MIC(X,D)$.
Then, \cref{pull_back_push_forward} follows from \cref{meromorphic_vs_non_meromorphic} and  base change for $\cD$-modules \cite[1.7.3]{HTT}.
\end{proof}

\begin{prop}\label{slope_push_forward}
Let $p : (Y,E)\to (X,D)$ be a proper morphism of pairs over $k$.
Let $f$  be a non constant function on $X$ such that $fp=0$ is a divisor of $Y$. 
Let $Z$ be the support of $\divi f$ and put $T:= p^{-1}(Z)$.
Assume that $Y\setminus T \to X\setminus Z$ is an isomorphism.
Let $\cM$ be an object of $\MIC(X,D)$.
Then, 
$$
\Sl^{\nb}_f(\cM) \subset \Sl^{\nb}_{fp}((p^+\cM)(\ast T)) 
$$
\end{prop}
\begin{proof}
We have 
$$
\Sl^{\nb}_f(\cM)   =\Sl^{\nb}_f(\cM(\ast Z)) 
                            =\Sl^{\nb}_{f}(p_+p^+(\cM(\ast Z)))   
                            \subset \Sl^{\nb}_{fp}((p^+\cM)(\ast T)) 
$$
The first equality follows from \cref{rem_localization}.
Since $Y\setminus T \to X\setminus Z$ is an isomorphism, the induced morphism $Y\setminus (T\cup E) \to X\setminus (Z\cup D)$ is an isomorphism with $\cM(\ast Z)$ being an object of $\MIC(X,Z\cup D)$.
Hence, the second equality follows from \cref{pull_back_push_forward}. 
The last inclusion follows from \cref{push_forward_Th}.
\cref{slope_push_forward} is thus proved.
\end{proof}

Over curves, nearby slopes coincide with the traditional notion of slopes recalled in \cref{defin_irregularity_number}.
This is given by the following lemma proved in \cite[3.3.1]{NearbySlope}.

\begin{lem}\label{slopes_and_nearby_slopes}
Let $C$ be a smooth curve over $k$.
Let $0$ be a point in $C$ and let   $t$ be a local uniformizer of $C$ at $0$.
Let $\cM$ be an object in $\MIC(C,0)$.
 Then the nearby slopes of $\cM$ associated to $t$ are the slopes of $\cM$ at $0$ in the sense of \cref{defin_irregularity_number}.
\end{lem}

\subsection{Boundedness of nearby slopes}

In this paragraph, we will need to resolve the turning locus of a connection $\cM$ while  resolving  a fixed subscheme  relatively to the pole locus.
Combined with Kedlaya's  \cref{main_thm_Ked_III}, this  resolution gives some control on the nearby slopes of $\cM$.
The existence of this resolution  is achieved by the following 

\begin{lem}\label{BM_KM}
Let $(X,D)$ be a normal crossing pair over $k$.
Let $\cM$ be an object of $\MIC(X,D)$.
Let $Z$ be a closed subscheme of $X$ such that $\TL(\cM)\subset Z_{\red}$.
Then, there is a morphism of normal crossing pairs $p : (Y,E)\to  (X,D)$ such that $p:Y\to X$ is a resolution of $Z$ relative to $D$ and  $p^+\cM$ has good formal structure.
\end{lem}
\begin{proof}
From Kedlaya-Mochizuki \cref{KM_theorem}, there exists a morphism of smooth algebraic varieties $q :Y'\to X$ obtained as a composition of blow-up above $\TL(\cM)$ such that $E'=q^{-1}(D)$ is a normal crossing divisor of $Y'$ and $q^+\cM$ admits good formal structure along $E'$.
From \cite[080B]{SPBLow_up}, there is a closed subscheme $T$ of $X$ with $T_{\red}\subset \TL(\cM)$ such that $q:Y'\to X$ identifies with the blow-up of $X$ along $T$.
By the universal property of the blow-up, $q:Y'\to X$ is a final object in the full subcategory of schemes over $X$ consisting in  morphisms $Y''\to X$ such that the inverse image of $T$ is an effective Cartier divisor on $Y''$.
By assumption, we have $T_{\red}\subset \TL(\cM)\subset Z_{\red}$.
From \cref{resolution_and_principalization},  there exists a resolution $p : Y\to X$ of $Z$ relative to $D$ such that the pull-back scheme $p^{-1}(T)$ is an effective Cartier divisor on $Y$.
Put $E=p^{-1}(D)$.
Then, there is a canonical commutative diagram of normal crossing pairs
$$
\xymatrix{
    (Y,E)     \ar[rr]^-{h}  \ar[dr]_-{p}    &        &    (Y',E')   \ar[ld]^-{q}  \\
    & (X,D)  &  
}
$$
From \cref{pullback_good_is_good}, the connection $h^+(q^+\cM)\simeq p^+\cM$ has good formal structure and the proof of \cref{BM_KM} is complete.

\end{proof}

We recall the following boundedness result for nearby slopes,  formulated in terms of the highest slope $b$-divisor from  \cref{R_divisor}.  
See \cite[3.4.1]{NearbySlope} for a proof.
\begin{lem}\label{boundedness_good_case}
Let $(X,D)$ be a normal crossing pair over $k$. 
Let $\cM$ be an object of $\MIC(X,D)$ with   good formal structure. 
Let $f$ be a non constant function on $X$ such that $\divi f$ is supported on $D$.
Let $r\in \mathds{Q}_{\geq 0}$ such that for every divisorial valuation $v$ centred at an irreducible component  of $\divi f$, we have 
$$
R(\cM)(v)\leq r \cdot (\divi f)(v)
$$
where $R(\cM)$ is the highest slope $b$-divisor from  \cref{R_divisor}.
Then the nearby slopes of $\cM$ associated to $f$ are smaller than $r$.
\end{lem}

\begin{prop}\label{boundedness_general}
Let $(X,D)$ be a normal crossing pair of dimension $d$ over $k$. 
Let $\cM$ be an object of $\MIC(X,D)$.
Let $f$ be a non constant function on $X$. 
Suppose that the turning locus of $\cM$ lies in the support of $\divi f$. 
Then, the nearby slopes of $\cM$ associated to $f$ are bounded by  $d\cdot \fdeg \Irr(X,\cM)$.
\end{prop}
\begin{proof}
Put $Z:=\divi f$.
Since $\TL(\cM)\subset Z_{\red}$, \cref{BM_KM} ensures the existence of a morphism of normal crossing pairs $p : (Y,E)\to (X,D)$ such that $p:Y\to X$ is a resolution of $Z$ relative to $D$ and  $p^+\cM$ has good formal structure.
Let $T:= p^{-1}(Z)=p^*\divi f=\divi fp$ be the pull-back scheme.
Let $(Z',D')$ be the transform of $(Z,D)$ by $p$.
Since $Z$ is a divisor, so is $Z'$.
By assumption, $Z'$ and $D'$ have only simultaneously normal crossing.
Thus, $Z'\cup D'$ is a normal crossing divisor of $Y$.
Since  $Z'\cup D'$ and $T\cup E$ have the same support, we deduce that $T\cup E$ is a normal crossing divisor of $Y$.
We are going to apply \cref{boundedness_good_case} on $(Y,T\cup E)$ to $(p^+\cM)(\ast T)$ and $fp$.
Since $p^+\cM \in \MIC(Y,E)$ has good formal structure, so does $(p^+\cM)(\ast T) \in \MIC(Y,T\cup E)$.
By assumption, $\divi fp$ is supported on $T\cup E$.
Let $v $ be a divisorial valuation centred at an irreducible component of  $T=\divi fp$.
If the centre of $v$ is not a component of $E$, we have 
$$
R((p^+\cM)(\ast T))(v)=0 \leq d  \cdot  (\fdeg \Irr(X,\cM)) \cdot  (\divi f)(v).
$$
Otherwise, $R((p^+\cM)(\ast T))(v)=R(\cM)(v)$.
On the other hand, we have   
\begin{align*}
R(\cM)(v)& \leq (\Irr \cM)(v)\\
                     & \leq (\Irr(X,\cM))(v)\\
                     & \leq d  \cdot  (\fdeg \Irr(X,\cM)) \cdot  (\divi f)(v)
\end{align*}
where the first inequality is trivial, the second follows from Kedlaya's \cref{main_thm_Ked_III} and where the last inequality comes from \cref{lemme_from_nearby_slopes}.
From \cref{boundedness_good_case}, we deduce that the nearby slopes of $(p^+\cM)(\ast T)$ associated to $fp$ are bounded by  $d  \cdot  (\fdeg \Irr(X,\cM))$.
Finally, \cref{slope_push_forward} applied to the proper morphism of pairs $p : (Y,E)\to (X,D)$ yields
$$
\Sl^{\nb}_f(\cM)  \subset \Sl^{\nb}_{fp}((p^+\cM)(\ast T)) 
$$
which concludes the proof of   \cref{boundedness_general}.
\end{proof}

\begin{prop}\label{bound_chi_D}
Let $(X,D)$ be a normal crossing pair of dimension $d$ over $\mathds{C}$. 
Let $C$ be a smooth connected curve over $\mathds{C}$. 
Let $f : X\longrightarrow C$ be a dominant proper morphism.
Let $0$ be a closed point in $C$. 
Suppose that  the reduced fibre $Z$ of $f$ over $0$ is not empty and is contained in $D$. 
Suppose that cohomological boundedness holds with bound $K$ for the generic fibre of $f : X\longrightarrow C$.
Then for every  effective divisor $R$ of $X$ supported on $D$, for every integer $r\geq 0$ and every object $\cM$ of $\MIC_r(X,D,R)$ with $\TL(\cM)\subset Z$, we have 
$$
|\chi(Z(\mathds{C}),\Sol \cM^{\an})|\leq d\cdot K(R_{\eta},r)  \cdot \fdeg R
$$
where $R_{\eta}$ is the pull-back of $R$ to the generic fibre of $f : X\longrightarrow C$.
\end{prop}
\begin{proof}
By the generic smoothness theorem, the generic fibre of $X$ over $C$ is smooth and the generic fibre of $(X,D)$ over $C$ is again a normal crossing pair.
Hence, the statement of cohomological boundedness makes sense over the generic fibre of $(X,D)$ over $C$.
We reproduce here a reasoning carried out in \cite[Th. 2]{teyConjThese} in a particular case. 
Let $t$ be a local uniformizer of $C$ around $0$.
Since $Z\subset D$, we have 
$$
(\Sol \cM^{\an})|_{Z(\mathds{C})}\simeq \Irr_{Z(\mathds{C})}^* \cM^{\an}
$$
Since the irregularity complex commutes with proper push-forward \cite{Mehbgro}, we deduce
$$
\chi(Z(\mathds{C}), \Sol \cM^{\an})  =\chi(Z(\mathds{C}), \Irr_{Z(\mathds{C})}^* \cM^{\an}) =(-1)^{d-1}\chi(0, \Irr_0^* f_+^{\an} \mathcal{M}^{\an})
$$
From \cref{Irr_pervers}, the irregularity complex of an holonomic $\cD_C$-module is a sky-scraper sheaf concentrated in degree $1$.
Hence, \cref{irr_and_Irr} gives
\begin{align*}
\chi(0, \Irr_0^* f_+^{\an} \mathcal{M}^{\an})&=    \sum (-1)^i \dim  \mathcal{H}^1 \Irr_0^* \mathcal{H}^{i-1}f_+^{\an} \mathcal{M}^{\an} \\
                                                                         & =\sum (-1)^i \irr(0,\mathcal{H}^{i-1} f_+^{\an} \mathcal{M}^{\an})
\end{align*}
On the other hand, the properness of $f$ yields from \cite[4.7.2]{HTT} a canonical identification $f_+^{\an} \mathcal{M}^{\an} \simeq (f_+ \mathcal{M})^{\an}$.
Since the irregularity number at $0$ is a formal invariant, we deduce 
$$
\chi(0, \Irr_0^* f_+^{\an} \mathcal{M}^{\an})=\sum (-1)^i \irr(0,\mathcal{H}^{i-1} f_+ \mathcal{M})
$$
For every integer $i$, \cref{slopes_and_nearby_slopes} gives
$$
 \irr(0, \mathcal{H}^i f_+\mathcal{M})\leq \rk\mathcal{H}^i f_+\mathcal{M} \cdot \Max \Sl^{\nb}_t(\mathcal{H}^i  f_+\mathcal{M})
$$
From \cref{push_forward_Th} and the exactness of nearby cycles, we have for every $i$
$$
\Sl^{\nb}_t(\mathcal{H}^i  f_+\mathcal{M})\subset \Sl^{\nb}_t(f_+\mathcal{M}) \subset \Sl^{\nb}_f(\mathcal{M})
$$
We deduce 
$$
|\chi(Z(\mathds{C}),\Sol \cM^{\an})|\leq \Max  \Sl^{\nb}_f(\mathcal{M}) \sum \rk\mathcal{H}^i f_+\mathcal{M} 
$$
Since $Z$ contains the turning locus of $\cM$, \cref{boundedness_general} implies that the nearby slopes of $\mathcal{M}$ associated to $f$ are bounded by  $d\cdot \fdeg R$. 
Since cohomological boundedness holds with bound $K $  for the generic fibre of $f : X\longrightarrow C$, \cref{bound_chi_D} follows.
\end{proof}

\section{Partial discrepancy $b$-divisors}\label{partial_discr}

Before proving cohomological boundedness for surfaces, we make a détour through the theory of $b$-divisor to introduce the partial discrepancy of a $b$-divisor.
As usual, $k$ denotes a field of characteristic $0$.

%

\subsection{Chain of blow-up}\label{Chain}

Let $X$ be a smooth surface over $k$.
Let $n\geq 0$.
Borrowing the terminology of \cite{Cossart_Piltant_Lopez}, a  \textit{chain of blow-up of $X$ of length $n$} is a sequence $\mathbf{p}$ of maps 
\begin{equation}\label{chain_blow_up}
\xymatrix{
\mathbf{p} : X':=X_n \ar[r]^-{p_{n-1}} &     X_{n-1} \ar[r]^-{p_{n-2}} &   \cdots \ar[r]^-{p_{0}}  &  X_0=X
}
\end{equation}
where $p_i$ is a blow-up at a single closed point $P_i$ of $X_i$ with $p_i(P_{i+1})=P_i$.
We let  $E_\mathbf{p}$ be the exceptional divisor of $p_{n-1}$.
We denote by $p : X'\to X$ the composition of the $p_i$, $0\leq i\leq n-1$ and by $\pi : X_{n-1}\to X$ the composition of the $p_i$, $0\leq i\leq n-2$.\\ \indent

The following theorem was proven by Zariski \cite{Zariski_surface} when $k$ is algebraically closed.
See  \cite{Cossart_surface} for arbitrary $k$ of characteristic $0$.

\begin{thm}\label{Zariski}
Let $X$ be a smooth connected surface over $k$.
Let $k(X)$ be the function field of $X$.
The map  associating to any chain of blow-up $\mathbf{p}$ of $X$ the valuation on $k(X)$ induced by $ E_\mathbf{p}$  is a bijection between the set of chains of blow-up of $X$ and $\ZR^{\divis}(X)$.
\end{thm}

The map $p : X'\to X$  corresponding to a divisorial valuation  $v$ via \cref{Zariski} is the smallest modification of $X$ on which the center of $v$ is a divisor. 
This is the following \cref{minimality_p} below.
To prove it, we need the following
\begin{lem}\label{point_or_pure_curve}
Let  $f : Y\to X$ be a  modification between smooth connected surfaces over $k$.
Let $P$ be a closed point of $X$.
Then either the scheme  $f^{-1}(P)$ has pure dimension $1$ or  $f^{-1}(P)$ is a single point $Q$.
In the latter case, there is a neighbourhood $U$ of $P$ in $X$ such that the induced morphism $f|_U: f^{-1}(U)\to U$ is an isomorphism. 
\end{lem}
\begin{proof}
This follows from the fact \cite[0C5R]{SPsurface} that a proper birational morphism  between smooth connected surfaces is a sequence of blow-up in closed points.
\end{proof}

\begin{prop}\label{minimality_p}
Let $X$ be a smooth connected surface over $k$.
Let $v$ be a divisorial valuation of $X$ and let $f : Y\to X$ be a  modification such that the center of $v$ on $Y$ is a divisor.
Let $\mathbf{p}$  be the chain of blow-up (\ref{chain_blow_up}) corresponding to $v$ via \cref{Zariski}.
Then, $f$ admits a unique factorization through  $p : X'\to X$.
\end{prop}
\begin{proof}
We argue by induction on the length $n\geq 0$ of $\mathbf{p}$.
If $n=0$, there is nothing to do.
Assume that $n>0$ and that $\mathbf{p}$ has length $n$ as in (\ref{chain_blow_up}). 
Let $E$ be the center of $v$ on $Y$.
Since the center of $v$ on $X$ is the point $P_0$, the map $f$ sends $E$ on $P_0$.
In particular the scheme $f^{-1}(P_0)$ contains $E$. 
From \cref{point_or_pure_curve}, the scheme  $f^{-1}(P_0)$ has pure dimension one.
It is thus an effective Cartier divisor on $Y$.
From the universal property of the blow-up, we deduce that $f$ factors uniquely through $p_0 : X_1\to X$ via a modification $g : Y\to X_1$.
Observe that the chain 
$$
\xymatrix{
\mathbf{q} : X':=X_n \ar[r]^-{p_{n-1}} &     X_{n-1} \ar[r]^-{p_{n-2}} &   \cdots \ar[r]^-{p_{1}}  &  X_1
}
$$
is the chain corresponding to $v$ via \cref{Zariski} applied to the smooth connected surface $X_1$.
Since $\mathbf{q}$ has length $n-1$, the recursion assumption applies to $g : Y\to X_1$ and $v$ and provides a unique factorization of $g$ through $q: X'\to X_1$.
We thus get the desired factorization of $f$ through $p$.
\end{proof}

The following general lemmas will be used in  the proof of \cref{discrepancy_is_b_divisor} to show that the partial discrepancy of a $b$-divisor is again a $b$-divisor.

\begin{lem}\label{elementarylemma_singular_point}
Let $X$ be a smooth variety over $k$.
Let $Z$ and $D$ be divisors on $X$ such that $Z_{\red}\subset D_{\red}$.
Then every singular point of $Z_{\red}$ is a singular point of $D_{\red}$.
\end{lem}
\begin{proof}
Let $P$ be a singular point of $Z_{\red}$.
If $P$ lies in   at least two distinct irreducible components of  $D$, then $P$ is necessarily a singular point of $D_{\red}$ and we are done.
Assume that $P$ lies in a single irreducible component $F$ of $D$.
Since $Z_{\red}$ is contained in $D_{\red}$, then $P$ also lies in a single irreducible component $F'$ of $Z$.
We necessarily have $F=F'$.
Hence, $Z_{\red}$  and $D_{\red}$  coincide in a neighbourhood of $P$ and \cref{elementarylemma_singular_point} follows.
\end{proof}

\begin{lem}\label{inclusion_support}
Let $X$ be a smooth connected surface over $k$.
Let $Z$ be a $b$-divisor of $X$.
Let $\mathbf{p}$ be a  chain of blow-up of $X$ of length $n\geq 1$.
Assume that $P_{n-1}$ lies in the support of $Z(X_{n-1})$.
Then, the following statements hold ;
\begin{enumerate}\itemsep=0.2cm
\item The support of $\pi^*Z(X)$ contains every curve of $X_{n-1}$ contracted to a point by $\pi : X_{n-1}\to X$.
\item We have  $Z(X_{n-1})_{\red}\subset (\pi^*Z(X))_{\red}$.
\item  If $E$ is a divisor of $X_{n-1}$ such that $(\pi^*Z(X))_{\red}\subset E_{\red}$, then every singular point of $Z(X_{n-1})_{\red}$ is a singular point of $E_{\red}$.
\end{enumerate}
\end{lem}
\begin{proof}
The second assertion follows from the first.
Indeed, away from the union of curves of $X_{n-1}$ contracted to a point by $\pi$, the map $\pi$ induces an isomorphism with $X\setminus \{P_0\}$.
Hence, $Z(X)=\pi_* Z(X_{n-1})$ implies that every component of $Z(X_{n-1})$ not contracted to a point  lies in the support of $\pi^*Z(X)$.\\ \indent
We thus prove the first assertion.
The curves of $X_{n-1}$ contracted by $\pi$ are the strict transforms of the exceptional divisors appearing in the chain $\mathbf{p}$.
Hence, the first assertion follows from the fact that since $P_{n-1}$ lies $Z(X_{n-1})$, the point  $P_0=\pi(P_{n-1})$ lies in $Z(X)=\pi_* Z(X_{n-1})$. \\ \indent
The last assertion follows from the inclusion of supports $Z(X_{n-1})_{\red}\subset (\pi^*Z(X))_{\red} \subset E_{\red}$ combined with \cref{elementarylemma_singular_point}.
\end{proof}

\subsection{Partial discrepancy divisor}

Armed with Zariski's \cref{Zariski}, we are now in position to define the partial discrepancy  of a $b$-divisor on a smooth connected surface endowed with a reduced divisor.
Although defined for every $\mathds{Z}$-valued function on the set of divisorial valuations, the partial discrepancy is better behaved on a subgroup of the group of $b$-divisors that we now introduce.
\begin{lem}\label{b_divisor_X_D}
Let $(X,D)$ be a connected pair of dimension $2$ over $k$.
Let $Z$ be a $b$-divisor on $X$ and let us view $D$ as a $b$-divisor via $\Cart(X)\to \bDiv(X)$.
The following conditions are equivalent ;
\begin{enumerate}\itemsep=0.2cm
\item The support of $Z : \ZR^{\divis}(X)\to \mathds{Z}$ lies in the support of $D : \ZR^{\divis}(X)\to \mathds{Z}$.
\item For every modification $f : Y\to X$, the support of $Z(Y)$ is contained in the support of $f^*D$.
\end{enumerate}
\end{lem}
\begin{proof}
We assume that $(1)$ holds.
Let $f : Y\to X$ be a modification and let $v$ be a divisorial valuation centred at a divisor in the support of $Z(Y)$.
Then $Z(v)=(Z(Y))(v)\neq 0$.
Hence, $D(v)=(f^*D)(v)\neq 0$ and $(2)$ holds.
On the other hand, assume that $(2)$ holds.
Let $v$ be a divisorial valuation on $X$ such that $Z(v)\neq 0$.
Then, there exists a modification $f : Y\to X$ such that the center $E$ of $v$ on $Y$ is a divisor.
Since $Z(v)= m(E,Z(Y))\neq 0$, the divisor $E$ is a component of $Z(Y)$.
By assumption, we deduce that $E$ is a component of $f^*D$.
Thus $D(v)= m(E,f^*D)\neq 0$.
\end{proof}

\begin{defin}\label{def_b_divisor_X_D}
Let $(X,D)$ be a connected pair of dimension $2$ over $k$.
Let $Z$ be a $b$-divisor on $X$.
If the equivalent conditions of \cref{b_divisor_X_D} are satisfied, we say that \textit{$Z$ is a $b$-divisor of $(X,D)$}.
We denote by $\bDiv(X,D)$ the subgroup of $\bDiv(X)$ formed by $b$-divisors of $(X,D)$.
\end{defin}
Let $Z$ be a $b$-divisor on $(X,D)$ as in  \cref{def_b_divisor_X_D}.
We define a new function 
$$
\delta Z : \ZR^{\divis}(X) \lra \mathds{Z}
$$ 
as follows.
Let $v\in  \ZR^{\divis}(X)$.
Let $\mathbf{p}$ be the  chain of blow-up of $X$ as in (\ref{chain_blow_up}) corresponding to $v$ via \cref{Zariski}.
If $n=0$ or if $n\geq 1$ and $P_{n-1}$ is a singular point of the support of $\pi^*D$, we put $(\delta Z)(v)=0$.
Otherwise, we put
$$
(\delta Z)(v):= (Z(X_{n-1}))(v)-Z(v)
$$
where $Z(X_{n-1})$ is viewed as a $b$-divisor via $\Cart(X_{n-1})\to \bDiv(X)$.

\begin{rem}\label{two_cases}
If $P_{n-1}$ is not a singular point of  the support of $\pi^*D$, only two things can happen.
If $P_{n-1}$ does not lie in the support of $Z(X_{n-1})$, we have $(Z(X_{n-1}))(v)=0$.
Otherwise, \cref{inclusion_support} ensures that $P_{n-1}$ is a smooth point of the support of $Z(X_{n-1})$.
Thus $(Z(X_{n-1}))(v)$ is  simply the multiplicity of the unique irreducible component of $Z(X_{n-1})$ passing through  $P_{n-1}$.
\end{rem}


\begin{ex}\label{discrepancy_Cartier}
If $Z$ lies in the image of $\Cart(X)\to \bDiv(X)$, then $\delta Z=0$.
\end{ex}

\begin{prop}\label{discrepancy_is_b_divisor}
Let $(X,D)$ be a connected pair of dimension $2$ over $k$.
Let  $Z$ be a $b$-divisor of $(X,D)$.
Then $\delta Z$ is a $b$-divisor of $(X,D)$.
\end{prop}
\begin{proof}
Let $f : Y\to X$ be a  modification of $X$.
Let $v$ be a divisorial valuation centred at a  divisor of $Y$.
Let $\mathbf{p}$ be the chain of blow-up (\ref{chain_blow_up}) corresponding to $v$ via \cref{Zariski}.
Assume that $n\geq 1$ and that $P_{n-1}$ is not a singular point of the support of $\pi^*D$.
From \cref{two_cases}, there are two cases.
If $P_{n-1}$ does not lie in support of $Z(X_{n-1})$, then
$$
(\delta Z)(v)= (Z(X_{n-1}))(v)-Z(v)=-Z(v)
$$
and since $Z$ is a $b$-divisor, there is only a finite number of $v$ as above for which $(\delta Z)(v)\neq 0$.
Otherwise $P_{n-1}$ is a smooth point of the support of $Z(X_{n-1})$.
If $w$ is the divisorial valuation corresponding to the unique irreducible component of  $Z(X_{n-1})$ passing through $P_{n-1}$, we thus have $(Z(X_{n-1}))(v)=Z(w)$, so that 
$$
(\delta Z)(v)= Z(w)-Z(v)
$$
On the other hand, \cref{minimality_p} implies the existence of  a commutative triangle 
$$
\xymatrix{
    Y      \ar[rd]_-{f} \ar[r]  &  X'  \ar[d]^-{p} \\
                        & X
}
$$
Since the centre of $w$ on $X'$ is a divisor, so is its centre on $Y$.
Since $Z$ is a $b$-divisor, we again conclude that there is only a finite number of $v$ as above for which $(\delta Z)(v)\neq 0$.
This proves that $\delta Z$ is a $b$-divisor of $X$.
We are left to prove that $\delta Z$ is a $b$-divisor of $(X,D)$.
Let $v$ be a divisorial valuation with $D(v)=0$.
Since $Z$ is a $b$-divisor of $(X,D)$, we have  $Z(v)=0$ and $(Z(X_{n-1}))(v)=0$.
This concludes the proof of \cref{discrepancy_is_b_divisor}.
\end{proof}

\begin{defin}
Let $(X,D)$ be a connected pair of dimension $2$ over $k$.
Let  $Z$ be a $b$-divisor of $(X,D)$.
Then $\delta Z$ is called the \textit{partial discrepancy of $Z$}.
\end{defin}

The following proposition expresses the partial discrepancy  without any reference to the chains of blow-up provided by Zariski's \cref{Zariski}.
\begin{prop}\label{computation_discrepancy}
Let $(X,D)$ be a connected pair of dimension $2$ over $k$.
Let  $Z$ be a $b$-divisor of $(X,D)$.
Let $f: Y\to X$ be a modification of $X$.
Let $Y'\to Y$ be the blow-up of $Y$ at a closed point $P$ and let  $E$ be its exceptional divisor.
Let $v$ be a divisorial valuation on $X$ such that the center of $v$ on $Y$ is $P$ and the center of $v$ on $Y'$ is $E$.
If $P$ is a singular point of the support of $f^*D$, then $(\delta Z)(v)=0$.
Otherwise, we have $(\delta Z)(v)=(Z(Y))(v)-Z(v)$ where  $Z(Y)$ is viewed as a $b$-divisor via $\Cart(Y)\to \bDiv(X)$.
\end{prop}
\begin{proof}
Let $\mathbf{p}$  be the chain of blow-up (\ref{chain_blow_up}) corresponding to $v$ via \cref{Zariski}.
We use the notations from \cref{Chain}.
Since the center of $v$ on $Y$ is a point, so is the center of $v$ on $X$.
Thus, we have $n\geq 1$.
From \cref{minimality_p} applied to the composition $Y'\to Y \to X$ and $v$, there is a commutative diagram
$$
\xymatrix{
 E   \ar[r]   \ar[d]&  Y'     \ar[r]   \ar[d]  &  X'  \ar[d]^-{p_{n-1}}     \\
P   \ar[r]   & Y      \ar[rd]_-{f} &  X_{n-1}  \ar[d]^-{\pi} \\
     &                    & X
}
$$
Let $w$ be the divisorial valuation on $X$ corresponding to the penultimate exceptional divisor $E_{n-1}$ of the chain $\mathbf{p}$.
In particular, the centres of $v$ and $w$ on $X_{n-1}$ are $P_{n-1}$ and $E_{n-1}$ respectively.
They are in particular distinct with that of $w$ being a divisor.
Thus, $E$ and the center of $w$ on $Y'$ are distinct divisors.
Since all what $Y'\to Y$ does is to contract $E$ to $P$, we deduce that the center of $w$ on $Y$ is again a divisor.
A second application of  \cref{minimality_p}  to $f : Y\to X$ and $w$ yields a commutative diagram 
$$
\xymatrix{
   E   \ar[r]  \ar[d] &  Y'     \ar[r]   \ar[d]  &  X'  \ar[d]^-{p_{n-1}}     \\
   P   \ar[r]   &  Y      \ar[rd]_-{f} \ar[r]^-{h}  &  X_{n-1}  \ar[d]^-{\pi} \\
             &                    & X
}
$$
Observe that the scheme $h^{-1}(P_{n-1})$ contains  $P$.
If $h^{-1}(P_{n-1})$ was an effective Cartier divisor, then the universal property of blow-up would provide a commutative diagram 
$$
\xymatrix{
  Y'     \ar[r]   \ar[d]  &  X'  \ar[d]^-{p_{n-1}}     \\
  Y      \ar[ru] \ar[r]^-{h}  &  X_{n-1}  
}
$$
In particular, the center of $v$ on $Y$ would be a divisor. 
Contradiction.
Hence, $h^{-1}(P_{n-1})$ is not an effective Cartier divisor.
From \cref{point_or_pure_curve}, we deduce the existence of a neighbourhood $U$ of $P_{n-1}$ in $X_{n-1}$ such that $h$ induces an isomorphism $h|_U : h^{-1}(U)\to U$ over $X$.
\cref{computation_discrepancy} then follows.

\end{proof}

\begin{rem}
The terminology \textit{partial} discrepancy divisor comes from the fact that we are looking at the failure of $Z$ to  lie in the image of $\Cart(Y)\to \bDiv(X)$ with $Y\to X$ a modification by only contemplating  \textit{the smooth points} of $Z(Y)$.
The justification for this comes from \cref{calculation_chi_surface}, showing that De Rham cohomology   only sees the partial discrepancy of the irregularity $b$-divisors.
\end{rem}

\begin{rem}
Let $(X,D)$ be a connected pair of dimension $2$ over $k$.
Let $f: Y \to X$ be a modification  and put $E:=(f^*D)_{red}$.
Then, the canonical isomorphism $\bDiv(X)\to \bDiv(Y)$ induces an  isomorphism $\bDiv(X,D)\to \bDiv(Y,E)$.
We thus have a square 
$$
\xymatrix{
    \bDiv(X,D)  \ar[r]^-{\sim}  \ar[d]_-{\delta}    &   \bDiv(Y,E) \ar[d]^-{\delta}  \\
   \bDiv(X,D)   \ar[r]^-{\sim}    &    \bDiv(Y,E)
}
$$
We stress on the fact that the above square does not commute.
The lemma below asserts however that the only obstruction lies at   valuations centred at divisors of $Y$.
\end{rem}

\begin{lem}\label{quasi_commutativity}
Let $(X,D)$ be a connected pair of dimension $2$ over $k$.
Let $f: Y \to X$ be a modification over $k$ and put $E:=(f^*D)_{red}$.
Let $Z$ be a $b$-divisor on $(X,D)$ and let $T$ be the image of $Z$ under $\bDiv(X,D)\to \bDiv(Y,E)$.
Let $v$ be a divisorial valuation on $X$ centred at a closed point of $Y$.
Then, $\delta Z$ and $\delta T$ coincide at $v$.
\end{lem}
\begin{proof}
We unveil the definition of $(\delta T)(v)$.
Let
$$
\xymatrix{
\mathbf{q} : Y':=Y_n \ar[r]^-{q_{n-1}} &     Y_{n-1} \ar[r]^-{q_{n-2}} &   \cdots \ar[r]^-{q_{0}}  &  Y_0=Y
}
$$
be the chain of blow-up obtained by applying \cref{Zariski} to $Y$ and $v$.
By assumption on $v$, there is at least one blow-up, that is $n\geq 1$.
Let $Q$  be the last blow-up point.
Let $\rho : Y_{n-1}\to Y$ be the composition of the $q_i$, $0\leq i\leq n-2$. 
If $Q$ is a singular point of $(\rho^*E)_\red = (\rho^* f^*D)_{\red}$, then by definition $(\delta T)(v)=0$.
On the other hand, \cref{computation_discrepancy} applied to the $Z$,$v$ and $f \rho : Y_{n-1}\to X$ yields $(\delta Z)(v)=0$ so that  $\delta Z$ and $\delta T$ coincide at $v$ in that case.
Assume that $Q$ is not a singular point of $(\rho^*E)_\red = (\rho^* f^*D)_{\red}$.
Then by definition $(\delta T)(v)=(T(Y_{n-1}))(v)-T(v)$.
On the other hand, \cref{computation_discrepancy} yields again $(\delta Z)(v)=(Z(Y_{n-1}))(v)-Z(v)$.
Since $Z(Y_{n-1})=T(Y_{n-1})$ and $Z(v)=T(v)$, we conclude that $\delta Z$ and $\delta T$ coincide at $v$ in that case.
This concludes the proof of \cref{quasi_commutativity}.
\end{proof}

\subsection{Integral of finitely supported $b$-divisors}
Let $X$ be a smooth connected surface over $k$.
Let $f : Y\to X$ be a modification and let $A \subset Y$ be a subset.
Let $\ZR^{\divis}(X,A)\subset \ZR^{\divis}(X)$ be the subset of divisorial valuations of $X$ whose center on $Y$ is a point of $A$.
For a $b$-divisor $Z$ of $X$ with finite support as a function on $\ZR^{\divis}(X)$, we put 
$$
\int_A Z := \sum_{v \in \ZR^{\divis}(X,A)}  Z(v)   \in \mathds{Z}\
$$
The goal of what follows is to prove that the partial discrepancy of a nef Cartier $b$-divisor is effective and finitely supported.

\begin{prop}\label{image_delta_of_nef_Cartier}
Let $(X,D)$ be a connected pair of dimension $2$ over $k$.
Let $Z$ be a Cartier $b$-divisor of $(X,D)$.
Let $f: Y\to X$ be a modification of $X$ such that $Z$ lies in the image of $\Cart(Y)\to \bDiv(X)$.
Then, $\delta Z$ is supported on the set of divisorial valuations whose center on $Y$ is an irreducible component of $f^*D$.
\end{prop}
\begin{proof}
Put $E:=(f^*D)_{red}$.
Let  $T$ be the image of $Z$ under $\bDiv(X,D)\to \bDiv(Y,E)$.
By assumption, $T$ is a Cartier $b$-divisor.
 \cref{discrepancy_Cartier} thus gives $\delta T=0$.
From \cref{quasi_commutativity}, we deduce that $\delta Z$ vanishes at every divisorial valuation centred at a closed point of $Y$.
This concludes the proof of \cref{image_delta_of_nef_Cartier}.
\end{proof}

\begin{prop}\label{delta_is_eeffective}
Let $(X,D)$ be a connected pair of dimension $2$ over $k$.
Let $Z$ be a Cartier $b$-divisor of $(X,D)$.
Then,  $\delta Z$ is a $b$-divisor with finite support.
In particular, the integral $\int_X \delta Z$ is a well-defined integer.
If $Z$ is furthermore nef, $\delta Z$ is effective. 
Hence, $\int_X \delta Z$ is a positive integer.
\end{prop}
\begin{proof}
That $\delta Z$ is a $b$-divisor with finite support follows from \cref{discrepancy_is_b_divisor} and \cref{image_delta_of_nef_Cartier}.
We are thus left to show that  $\delta Z$  is effective if $Z$ is nef.
Let $v$ be a divisorial valuation on $X$.
let 
$$
\xymatrix{
\mathbf{p} : X':=X_n \ar[r]^-{p_{n-1}} &     X_{n-1} \ar[r]^-{p_{n-2}} &   \cdots \ar[r]^-{p_{0}}  &  X_0=X
}
$$
be the chain of blow-up corresponding to $v$ via \cref{Zariski}.
Since $Z$ is a nef Cartier divisor, we have $Z\leq Z(X_{n-1})$ in $\bCDiv(X)$ and \cref{delta_is_eeffective} follows.
\end{proof}

\section{Formula for the characteristic cycle of connections on surfaces}\label{formula_surfaces}

\subsection{Local Euler-Poincaré characteristic and characteristic cycle }
Let  $X$ be  a smooth variety over $\mathds{C}$.
Let $D_c^b(X(\mathds{C}), \mathds{C})$ be the derived category of sheaves of $\mathds{C}$-vector spaces on $X(\mathds{C})$ with bounded and constructible cohomology.
For an object $K$ of $D_c^b(X(\mathds{C}), \mathds{C})$ and a point  $x$ of $X(\mathds{C})$, we put  
$$
\chi(x, K)=\sum_{i\in \mathds{Z}} (-1)^i \rank (\mathcal{H}^i K)_x
$$
We let $CC(K)$ be the characteristic cycle of $K$.
Let $\mathds{F}(X)$ be the group of $\mathds{Z}$-valued constructible functions on $X(\mathds{C})$.
Let $L_X$ be the group  of lagrangian cycles of $T^*X$.
Then, there is a canonical isomorphism $\Eu : L_X \to   \mathds{F}(X) $ called the \textit{Euler morphism} making the following triangle commutative : 
$$
\xymatrix{
D_c^b(X(\mathds{C}), \mathds{C})  \ar[rr]^-{\chi} \ar[rd]_-{CC} & &  \mathds{F}(X)  \\
 &   L_X   \ar[ru]_-{\Eu}  & 
}
$$
See \cite[9.7.1]{KS} and references therein.
If $S$ is a constructible subset of $X$, we denote by $\mathds{1}_S$ the function on $X(\mathds{C})$ which sends $x\in S$ to $1$ and $x\in X(\mathds{C})\setminus S(\mathds{C})$ to $0$.
By construction, if $S$ is a smooth closed subvariety of $X$, then $\Eu(T^*_S X)=(-1)^{\codim S}\mathds{1}_S$.\\ \indent

We now define a cycle that will show  up frequently in this paper.
\begin{defin}\label{CC_effective_divisor}
Let $(X,D)$ be a  simple normal crossing pair over $k$.
Let $D_1,\dots, D_n$ be the irreducible components of $D$.
Let  $R=\sum a_i D_i$ be an effective divisor supported on $D$.
For every $I \subset\{ 1,\dots,  n\}$, set 
$$
D_I := \bigcap_{i\in I}D_i \text{ and }  D_I^{\circ} := D_I \setminus \displaystyle{\bigcup_{i\notin I } }  D_i
$$
Define
$$
LC(R) := \displaystyle{\sum_{I\subset \{ 1,\dots,  n\}}}\left(\displaystyle{\sum_{i\in I}} a_i \right)[T^*_{D_I}X]
$$
\end{defin}

\begin{rem}\label{R_CC}
When $k=\mathds{C}$ in the above definition, $LC(R)$ is the unique lagrangian cycle in $T^*X$  such that
$$
 \Eu(LC(R))=\left\{
\begin{array}{cl}
-a_i &  \text{if $x\in D_i^{\circ}, i=1, \dots, n$}    \\
  0  & \text{otherwise }
\end{array}
\right.
$$
\end{rem}

\subsection{Characteristic cycle of a connection with good formal structure}

The easy direction of the main theorem of \cite{teyConjThese} gives the following

\begin{lem}\label{irr_good formal}
Let $(X,D)$ be an analytic pair where $D$ is smooth and connected.
Let $\cM$ be an object of $\MIC(X,D)$.
Assume that $\cM$ has good formal structure.
Then $\Irr_D^* \cM$ is a local system of rank $\irr(D,\cM)$ concentrated in degree $1$.
\end{lem}

\begin{prop}\label{vanishing}
Let $(X,D)$ be a normal crossing pair over $\mathds{C}$.
Let $D_1,\dots, D_n$ be the irreducible components of $D$.
Let $\cM$ be an objet of $\MIC(X,D)$. 
Assume that $\cM$ has good formal structure. 
Then, 
$$
 \chi(x, \Sol \cM^{\an})=\left\{
\begin{array}{cl}
 \rk \cM  &\text{if $x\in X\setminus D$} \\
-\irr(D_i,\cM) &  \text{if $x\in D_i^{\circ}, i=1, \dots, n$}    \\
  0  & \text{otherwise }
\end{array}
\right.
$$

\end{prop}
\begin{proof}
The restriction of $\Sol \cM^{\an}$ to the complement of $D$ is a local system concentrated in degree $0$. 
This gives the expected formula in that case.
The case where $x$ lies in one of the $D_i^{\circ}$ follows from \cref{irr_good formal}.
Let $I\subset \{1,\dots, n\}$ and suppose that $I$ contains at least two elements. 
Let $x$ be a point of $D_I^{\circ}$. 
Let $i\in I$ such that the number of irregular values of $\cM$ at $x$ having no pole along $D_i$ is maximal. 
Such an integer exists by goodness assumption on the irregular values.
If $i_I : D_I \longrightarrow X$ and $j_i : D_i^{\circ} \longrightarrow X$ are the inclusions, Corollary 3.4 of \cite{SabRemar} implies that the natural morphism 
$$
i_{I}^*\Irr_D^* \cM^{\an} \longrightarrow i_{I}^* Rj_{i*} j_{i}^*\Irr_D^*\cM^{\an}
$$ 
is an isomorphism. 
Hence, the germ of $\Sol \cM^{\an}$ at $x$ is the cohomology complex of a local system of rank $\irr(D_i,\cM)$ on a torus of  dimension $|I|-1>0$. 
Thus, we have $\chi(x, \Sol \cM^{\an})=0$  and this finishes the proof of \cref{vanishing}.
\end{proof}

\begin{prop}\label{CCgood}
Let $(X,D)$ be a simple normal crossing pair over $k$.
Let $\cM$ be an object of $\MIC(X,D)$. 
Assume that $\cM$ has good formal structure.
Then,
$$
CC(\cM)= \rank \cM \cdot CC(\mathcal{O}_X(\ast D)) + LC(\Irr(X,\cM))
$$
with $CC(\mathcal{O}_X(\ast D))=\displaystyle{\sum_{I\subset \{ 1,\dots,  n\}}} [T^*_{D_I}X]$.
\end{prop}
\begin{proof}
From \cref{CC_flat_base_change}, we can suppose $k=\mathds{C}$. 
Then, \cref{CCgood} follows from \cref{vanishing} after passing to the associated constructible functions.
\end{proof}

\begin{rem}
As pointed out by Budur, L. Xiao obtained in \cite[Th. 1.4.1]{LXiao_CC_clean} a variant of \cref{CCgood} where the characteristic cycle is replaced by the \textit{log-characteristic cycle} and where the good formal structure assumption is weakened into a \textit{cleanness} assumption.
For a proof of why good formal structure implies cleanness, see Theorems 3.2.11 and 3.3.12 in \cite[Th. 1.4.1]{LXiao_CC_clean}.
\end{rem}

\subsection{An application}\label{H_Kato}

\begin{lem}\label{same_Irr}
Let $(X,D)$ be a normal crossing pair over $k$.
Let $\cM_1$ and $\cM_2$ be objects of $\MIC(X,D)$.
Then, the following statements are equivalent;
\begin{enumerate}\itemsep=0.2cm
\item $\Irr \cM_1=\Irr \cM_2$ in $\bDiv(X)$.
\item For every point $0$ of $D$, for every locally closed smooth curve $C \to X$ in $X$ meeting $D$ at $0$ only, we have $\irr(0,\cM_1|_{C})=\irr(0,\cM_2|_{C})$. 
\end{enumerate}
\end{lem}
\begin{proof}
Suppose that $(1)$ holds.
Let $C \to X$ be a  locally closed smooth curve meeting $D$ at  $0$ only.
Let $p  : Y\to X$ be a resolution of the turning points of $\cM_1$ and $\cM_2$ as given by Kedlaya-Mochizuki \cref{KM_theorem}.
By valuative criterion for properness, the immersion $C\to X$ factors uniquely through an immersion $C\to Y$ followed by $p$.
At the cost of blowing-up further, we can suppose that $C\to Y$ intersects $E:=f^{-1}(D)$ transversely  at a point $0'$ of $E^{\sm}$. 
Let $Z$ be the component of  $E$ containing $0'$.
Then
$$
\irr(0,\cM_1|_{C})=\irr(0',(p^+\cM_1)|_{C})=\irr(Z,p^+\cM_1)=\irr(Z,p^+\cM_2)=\irr(0,\cM_2|_{C})
$$
Suppose that $(2)$ holds.
Let $p  : Y\to X$ be a modification and let $Z$ be an irreducible component of $p^{-1}(D)$.
Let $0'$ be a point of $Z^{\sm}$ such that $p^+\cM_1$ and $p^+\cM_2$ have good formal decomposition at $0'$.
Put $0:=p(0')$.
Let $C\to Y$ be a locally closed smooth curve meeting $p^{-1}(D)$ transversely at $0'$ only.
Let $C\to X$ be the induced locally closed smooth curve of $X$.
Then
$$
\irr(Z,p^+\cM_1)=\irr(0',(p^+\cM_1)|_{C})=\irr(0,\cM_1|_{C})=\irr(0,\cM_2|_{C})=\irr(Z,p^+\cM_2)
$$
and the proof of \cref{same_Irr} is complete.
\end{proof}

\begin{prop}\label{sameCC}
Let $(X,D)$ be a normal crossing pair over $k$.
Let $\cM_1$  and $\cM_2$ be objects of $\MIC(X,D)$ with the same rank.
Suppose that for every point $0$ of $D$, for every locally closed smooth curve $C \to X$ in $X$ meeting $D$ at $0$ only, we have 
$$
\irr(0,\cM_1|_{C})=\irr(0,\cM_2|_{C})
$$ 
Then $CC(\cM_1)=CC(\cM_2)$.
\end{prop}
\begin{proof}
From \cref{CC_flat_base_change}, we can suppose that $k=\mathds{C}$.
From Kedlaya-Mochizuki \cref{KM_theorem} combined with \cref{pull_back_push_forward} and \cref{push_forward_CC}, we can suppose that $\cM_1$ and $\cM_2$ have good formal structure and that $D$ has simple normal crossings.
From \cref{same_Irr}, we have $\Irr(X,\cM_1)=\Irr(X,\cM_2)$.
Hence, \cref{sameCC} follows from \cref{CCgood}.
\end{proof}

\begin{rem}
In the étale setting, \cref{sameCC}  is a Theorem of H. Kato \cite{H.Kato}.
\end{rem}

\subsection{Partial discrepancy of the irregularity divisor}

Applying the partial discre\-pancy construction to the irregularity $b$-divisor, \cref{partial_discr} can be summarized by the following proposition

\begin{prop}\label{irr_div_finitely_supported}
Let $(X,D)$ be a normal crossing connected surface over $k$.
Let $\cM$ be an object of $\MIC(X,D)$.
Then, $\delta \Irr\cM$ is an effective  $b$-divisor of $(X,D)$ with finite support.
In particular, $\int_X \delta \Irr\cM$ is a well-defined positive integer.
\end{prop}
\begin{proof}
From Kedlaya's \cref{main_thm_Ked_III}, the $b$-divisor  $\Irr\cM$ is nef and Cartier.
Then, \cref{irr_div_finitely_supported} follows from  \cref{delta_is_eeffective}.
\end{proof}

The non trivial contributions to the above integral are located above the turning points.
This is the content of the following

\begin{lem}\label{integral_over_Z_and_points}
Let $(X,D)$ be a  connected normal crossing surface over $k$.
Let $\cM$ be an object of $\MIC(X,D)$ and let $V(\cM)$ be the complement of $\TL(\cM)$.
Then, the restriction of $\delta \Irr \cM$ to $\ZR^{\divis}(V(\cM))$ is $0$.
In particular, for every subset $A\subset X$, we have
$$
\int_A \delta \Irr \cM = \int_{A\cap \TL(\cM)}\delta \Irr \cM
$$
\end{lem}
\begin{proof}
$\cM|_{V(\cM)}$ has good formal decomposition.
Thus  $(\Irr \cM)|_{\ZR^{\divis}(V(\cM))}=\Irr \cM|_{V(\cM)}$ lies in the image of $\Cart(X)\to \bDiv(X)$.
\cref{integral_over_Z_and_points} then follows from  \cref{discrepancy_Cartier}.
\end{proof}

\subsection{Local Euler-Poincaré characteristic and partial discrepancy divisor}

\begin{lem}\label{equality_integrals}
Let $(X,D)$ be a normal crossing connected surface over $k$.
Let $\cM$ be an object of $\MIC(X,D)$.
Let $q : Y\to X$ be a modification.
Let $Q$ be a closed point of $Y$.
Then, $\delta \Irr\cM$ and  $\delta \Irr q^+\cM$ coincide at every valuation of $X$ centred at $Q$.
In particular, 
$$
\int_Q \delta \Irr\cM  = \int_Q \delta \Irr q^+\cM
$$
\end{lem}
\begin{proof}
Put $E:=(q^*D)_{red}$.
Then, $\Irr q^+\cM$ is the image of $\Irr\cM $ under $\bDiv(X,D)\to \bDiv(Y,E)$.
Thus, \cref{equality_integrals} follows from \cref{quasi_commutativity}.
\end{proof}

\begin{thm}\label{formula_local_chi}
Let $(X,D)$ be a normal crossing surface over $\mathds{C}$.
Let $\cM$ be an object of $\MIC(X,D)$.
Let $P$ be a closed point of $D$.  
If $P$ is a singular point of $D$, we have
\begin{equation}\label{formula_chi_x_Sing_D}
\chi(P, \Sol \cM^{\an})=\int_{P} \delta \Irr \cM
\end{equation}
Otherwise, we have 
\begin{equation}\label{formula_chi_x_smooth_D}
\chi(P, \Sol \cM^{\an})=-\irr(Z, \cM)+\int_{P} \delta \Irr \cM
\end{equation}
where $Z$ is the component of $D$ containing $P$.
\end{thm}
\begin{proof}
We argue by recursion on the number of blow-up needed to resolve the turning points of $\cM$ lying above $P$.
If no blow-up is needed, this means that $\cM$ has good formal structure in a neighbourhood of $P$.
Then, \cref{bdiv_and_good_dec} ensures that $\Irr\cM$ lies in the image of $\Cart(X)\to \bDiv(X)$.
Hence, \cref{discrepancy_Cartier} yields  $\delta \Irr \cM=0$.
\cref{formula_local_chi} then follows from \cref{vanishing}.\\ \indent
Let $n>0$.
Let $p : Y\to X$ be a sequence of blow-up of length $n$ above $P$  such that $p^+\cM$ has good formal structure in a neighbourhood of $p^{-1}(P)$.
Then, $p=f\circ q$ where $q : X_1\to X$ is the blow-up at $P$ and where   $f:Y\to X_1$ is a sequence of $(n-1)$-blow-up above $P$.
Let $E$ be the exceptional divisor of $q$.
Let $D'$ be the strict transform of $D$.
Let $S$ be the set of points of $E$ which are either turning points for $q^+\cM$ or points of $E\cap D'$.
Put $U=E\setminus S$ and let $j : U\longrightarrow E$ and  $i : S\longrightarrow E$ be the inclusions. 
The localization triangle for $(\Sol q^+\cM^{\an})|_{E}$ reads
$$
\xymatrix{
j_! j^* (\Sol q^+\cM^{\an})|_{E}\ar[r] & (\Sol q^+\cM^{\an})|_{E} \ar[r] &  i_* i^*  (\Sol q^+\cM^{\an})|_{E}
}
$$
From \cref{pull_back_push_forward} and the compatibility of $\Sol$ with proper push-forward, we deduce
$$
\chi(P, \Sol \cM^{\an})=\chi_c(U,(\Sol q^+\cM^{\an})|_{U})+\displaystyle{\sum_{Q\in S}}\chi(Q, \Sol q^+\cM^{\an})
$$
Since $(\Sol q^+\cM^{\an})|_{U}$ is a local system of rank $\irr(E,q^+\cM)$ on $U$ concentrated in degree $1$, we have
$$
\chi_c(U, (\Sol q^+\cM^{\an})|_{U})=(|S|-2)\cdot \irr(E,q^+\cM)
$$
For a turning point $Q\in E$ for $q^+\cM$ not in $D'$,  the recursion hypothesis gives
\begin{align*}
\chi(Q, \Sol q^+ \cM^{\an})& =-\irr(E, q^+ \cM)+\int_{Q} \delta \Irr q^+\cM  \\
                                     & =-\irr(E, q^+ \cM)+\int_{Q} \delta \Irr \cM
\end{align*}
where the second equality follows from \cref{equality_integrals}.
For a point $Q$ in  $E\cap D'$, the recursion hypothesis gives
$$
\chi(Q, \Sol q^+ \cM^{\an})=\int_{Q} \delta \Irr q^+\cM=\int_{Q} \delta \Irr \cM
$$
where the second equality again follows from \cref{equality_integrals}.
Let $\nb_D(P)$ be the number of local branches of $D$ around $P$ passing through $P$. 
Note that $\nb_D(P)$ is equal to the cardinality of $E\cap D'$. 
Putting the above equalities together yields
$$
\chi(P, \Sol \cM^{\an})=(\nb_D(P) -2) \cdot \irr(E,q^+\cM)+\int_{E} \delta \Irr \cM
$$
Let $v$ be the valuation of $X$ corresponding to $E$, so that 
$$
\chi(P, \Sol \cM^{\an})=(\nb_D(P) -2) \cdot \irr(E,q^+\cM)-(\delta \Irr \cM)(v)+\int_{P} \delta \Irr \cM
$$
If $P$ is a singular point of $D$, then $\nb_D(P) =2$ and $(\delta \Irr \cM)(v)=0$.
Thus, \cref{formula_local_chi} is true in that case.
Otherwise, let $Z$ be the unique irreducible component of $D$ passing through $P$. 
Then,  $\nb_D(P)=1$ and $(\delta \Irr \cM)(v)=\irr(Z,\cM)-\irr(E,q^+\cM)$.
Hence,  \cref{formula_local_chi} is again true in that case. 
This finishes the proof of \cref{formula_local_chi}.
\end{proof}

\begin{cor}\label{computation_mult_TxX}
Let $(X,D)$ be a normal crossing surface over an algebraically closed field of characteristic $0$.
Let $\cM$ be an object of $\MIC(X,D)$.
Let $P$ be a closed point of $D$.  
If $P$ lies in the smooth locus of $D$, the multiplicity of $T_P^*X$ in $CC(\cM)$ is
$$
\int_{P} \delta\Irr\cM
$$
If $P$ lies in two irreducible components $Z_1$ and $Z_2$ of $D$, the multiplicity of $T_P^*X$ in $CC(\cM)$ is
$$
\rk \cM+\irr(Z_1, \cM)+ \irr(Z_2, \cM)+\int_{P} \delta\Irr\cM
$$
\end{cor}
\begin{proof} 
From \cref{CC_flat_base_change} and \cref{Change_base_field_integral}, we can suppose that the base field is $\mathds{C}$.
Suppose that $P$ lies in $D^{\sm}$.
Then, at the cost of shrinking $X$, we can suppose that $D$ is smooth and that $T_Q^*X$ does not contribute to  $CC(\cM)$ for $Q\neq P$.
Thus, we have 
$$
CC(\cM)=\rk \cM  \cdot  T_X^*X +(\rk \cM+\irr(D,\cM)) \cdot  T_D^*X+ m  \cdot T_P^*X
$$
where $m$ is the sought-after multiplicity.
Applying  \cref{KS_equality_CC} and passing to the associated constructible functions yields
$$
\chi(-, \Sol \cM^{\an}) = \rk \cM  \cdot  \mathds{1}_X -(\rk \cM+\irr(D,\cM)) \cdot  \mathds{1}_D + m \cdot  \mathds{1}_P
$$
From \cref{formula_local_chi}, evaluating at $P$ then gives $m= \int_{P} \delta\Irr\cM$.\\ \indent
Suppose that $P$  lies in two irreducible components $Z_1$ and $Z_2$ of $D$.
At the cost of shrinking X, we have 
$$
CC(\cM)=\rk \cM  \cdot  T_X^*X +(\rk \cM+\irr(Z_1,\cM)) \cdot  T_{Z_1}^*X+ (\rk \cM+\irr(Z_2,\cM)) \cdot  T_{Z_2}^*X+  m  \cdot T_P^*X
$$
where $m$ is the sought-after multiplicity.
Applying  \cref{KS_equality_CC} and passing to the associated constructible functions yields
$$
\chi(-, \Sol \cM^{\an}) = \rk \cM  \cdot  \mathds{1}_X -(\rk \cM+\irr(Z_1,\cM)) \cdot  \mathds{1}_{Z_1}  -(\rk \cM+\irr(Z_2,\cM)) \cdot  \mathds{1}_{Z_2}  + m \cdot  \mathds{1}_P
$$
From \cref{formula_local_chi}, evaluating at $P$ gives the expected formula.

\end{proof}


Putting together \cref{CCgood} and \cref{computation_mult_TxX} yields the following

\begin{thm}\label{computation_CC_surface}
Let $(X,D)$ be a simple normal crossing surface over an algebraically closed field of characteristic $0$.
Let $\cM$ be an object of $\MIC(X,D)$.
Then
$$
CC(\cM) = \rk \cM \cdot  CC(\mathcal{O}_{X}(\ast D)) +  LC(\Irr(X,\cM)) +  \sum_{P\in D}  \left( \int_{P} \delta\Irr\cM\right) \cdot T_{P}^*X 
$$
\end{thm}
\begin{rem}
In the étale setting, Yatagawa obtained an explicit description of the characteristic cycle for rank one étale sheaves on surfaces \cite{Yatagawa}.
It seems an interesting question to connect Yatagawa's work to the étale analogue of the irregularity $b$-divisor.
\end{rem}

In the complex setting, Kashiwara-Dubson's formula stated in  \cref{Dubson} yields the following Grothendieck-Ogg Shafarevich type formula for surfaces :

\begin{thm}\label{calculation_chi_surface}
Let $(X,D)$ be a proper simple normal crossing surface over $\mathds{C}$.
Let $\cM$ be an object of $\MIC(X,D)$.
Put $U:=X\setminus D$.
Then
$$
\chi(X,\DR\cM) = \rk \cM \cdot  \chi(U(\mathds{C})) +   (LC(\Irr(X,\cM)),T^*_X X)_{T^* X} +  \int_D \delta\Irr\cM
$$
\end{thm}

\section{Cohomological boundedness}\label{coho_bound_proof}

\subsection{Cohomological boundedness for surfaces}

\begin{lem}\label{reduction_relative situation}
Let $(X,D)$ be a proper simple normal crossing surface over $k$.
Let $V$ be an open subset of $X$ such that $D_V:=V\cap D$ is the vanishing locus of an algebraic function $f : V\to \mathds{A}^1$.
Then, there exists a commutative diagram of smooth varieties over $k$
\begin{equation}\label{eq_reduction_relative situation}
\xymatrix{
    V     \ar@{^{(}->}[r]^-{j}  \ar[d]_-{f}    &     Y   \ar[d]^-{h}  \\
     \mathds{A}^1  \ar@{^{(}->}[r]  &   \mathds{P}^1  
}
\end{equation}
satisfying the following conditions ; 
\begin{enumerate}\itemsep=0.2cm
\item  The map $j : V\longrightarrow Y$ is a dense open immersion.
\item The map $h$ is proper and dominant.
\item If  $E:=((Y\setminus V)\cup h^{-1}(0))_{\red}$, the pair $(Y,E)$ is a simple normal crossing surface with $Y\setminus E=V\setminus D_V$.
\item For every  effective divisor $R$ supported on $D$, there exists an effective divisor $S$ supported on $E$ depending only on $V$, on  $j : (V,D_V) \to (Y,E)$ and linearly on $R$ such that for every object $\cM$ of $\MIC(X,D,R)$, the connection $j_* \cM|_V$ is an object of $\MIC(Y,E,S)$.
\end{enumerate}

\end{lem}
\begin{proof}
By Nagata compactification theorem applied to the composition $V\to \mathds{P}^1$, there exists a commutative diagram (\ref{eq_reduction_relative situation}) where  $j : V\longrightarrow Y$ is a dense open immersion with $Y$ a variety over $k$ and where $h$ is proper.
Observe that $h$ is dominant because $f$ is dominant. 
By normalizing and blowing-up enough points above the singular  locus  $Y^{\sing}  \subset  Y\setminus V$, resolution of singularities for surfaces ensures that we can suppose $Y$ to be smooth.
Put $E:=((Y\setminus V)\cup h^{-1}(0))_{\red}$.
Then the equality $Y\setminus E=V\setminus D_V$ is automatic.
Furthermore $E\cap V = h^{-1}(0)\cap V =D_V$ is a simple normal crossing divisor of $V$.
Hence, at the cost of blowing-up enough points above $E\setminus E\cap V$, we can further suppose that $E$ is a simple  normal crossing divisor of $Y$.
The conditions $(1),(2),(3)$ are thus  satisfied.
Since $h$ and $\mathds{P}^1$ are proper, so is $Y$.
Hence, conditions $(4)$ is satisfied as a consequence of \cref{partial_compactification} and the proof of \cref{reduction_relative situation} is complete.
\end{proof}

\begin{lem}\label{CC_Sol_on_Z}
Let $(X,D)$ be a simple normal crossing  surface over $\mathds{C}$.
Let $Z$ be a reduced divisor of $X$ supported on $D$.
For a smooth point $P$ of $Z$ which is  singular in $D$, let $Z(P)$ be the component of $Z$ containing $P$.
Let $\cM$ be an object of $\MIC(X,D)$.
Let $\Irr(X,Z,\cM)$ be the effective divisor supported on $Z$ which coincides with $\Irr(X,\cM)$ on $Z$.
Then, 
\begin{align*}
CC((\Sol \cM^{\an})|_{Z(\mathds{C})}) & = LC(\Irr(X,Z,\cM)) +  \sum_{P\in Z}  \left( \int_{P} \delta\Irr\cM\right) \cdot T_{P}^*X \\
& + \sum_{P\in Z^{\sm}\cap D^{\sing}}  \irr(Z(P),\cM) \cdot T_{P}^*X
\end{align*}
where $(\Sol \cM^{\an})|_{Z(\mathds{C})}$ is viewed as a complex of sheaves on $X$ supported on $Z(\mathds{C})$.
\end{lem}
\begin{proof}
We conclude via \cref{formula_local_chi}  by looking at the associated constructible functions.
\end{proof}

\begin{prop}\label{bound_chi_Z}
Let $(X,D)$ be a geometrically connected simple normal crossing surface over $k$. 
Let $C$ be a smooth connected curve over $k$. 
Let $h : X\longrightarrow C$ be a dominant proper morphism.
Let $0$ be a closed point in $C$. 
Suppose that  the reduced fibre $Z$ of $h$ over $0$ is not empty and is contained in $D$. 
Then, there exists a quadratic polynomial $K : \Div(X,D)\oplus \mathds{Z}\to \mathds{Z}$ affine in the last variable such that for every  effective divisor $R$ of $X$ supported on $D$, for every integer $r\geq 0$ and every object $\cM$ of $\MIC_r(X,D,R)$, we have
$$
\int_{Z} \delta\Irr\cM \leq K(R,r)
$$
\end{prop}
\begin{proof}
Let $R$ be an effective divisor of $X$ supported on $D$. 
Let $r$ be an integer. 
Let $\cM$ be an object  of $\MIC_r(X,D,R)$.
From \cref{Change_base_field_integral_inequality}, we can suppose that $k$ is algebraically closed.
From \cref{Change_base_field_integral_general}, we can suppose that $k=\mathds{C}$.
The turning locus of $\cM$ consists in a finite set of closed points in $D$. 
At the cost of shrinking $C$, we can thus suppose $\TL(\cM)\subset Z$.
On the other hand, \cref{curve_case} ensures that cohomological boundedness holds with bound $K_1$ for the generic fibre of $h : X\longrightarrow C$.
Note that $K_1$ is an affine map.
Hence, \cref{bound_chi_D} implies
$$
|\chi(Z(\mathds{C}),\Sol \cM^{\an})|\leq 2K_1(R_{\eta},r)\cdot \fdeg R
$$
where $R_\eta$ is the pull-back of $R$ to the generic fibre of $h : X\to C$.
Since $h$ is proper, the complex manifold $Z(\mathds{C})$ is compact.
Hence, the index formula for constructible sheaves \cite[Th. 4.3]{KaIndex} yields
\begin{align*}
\chi(Z(\mathds{C}),\Sol \cM^{\an}) & =   (CC((\Sol \cM^{\an})|_{Z(\mathds{C})}),T^*_X X)_{T^*X}     \\
                                                         & =  (LC(\Irr(X,Z,\cM)),T^*_X X)_{T^*X} + \int_{Z} \delta\Irr\cM  \\
                                                         & +  \sum_{P\in Z^{\sm}\cap D^{\sing}}  \irr(Z(P),\cM) 
\end{align*}
where the second equality follows from  \cref{CC_Sol_on_Z}.
Hence, we have 
$$
\int_{Z} \delta\Irr\cM \leq 2K_1(R_{\eta},r)\cdot \fdeg R +| (LC(\Irr(X,Z,\cM)),T^*_X X)_{T^*X} |
$$
Let $Z_1,\dots, Z_n$ be the irreducible components of $Z$.
Then 
$$
|(LC(\Irr(X,Z,\cM)),T^*_X X)_{T^*X} |\leq \sum_{i=1}^n    m(Z_i,R)  |(Z_i,T^*_X X)_{T^*X}|
$$ 
If $K_2(R)$ denotes the right-hand side of the above inequality, we have 
$$
\int_{Z} \delta\Irr\cM \leq 2K_1(R_{\eta},r)\cdot \fdeg R+ K_2(R)
$$
and the proof of \cref{bound_chi_Z} is complete.
\end{proof}

Putting everything together yields the following absolute version of \cref{bound_chi_Z}.

\begin{thm}\label{strong_chi}
Let $(X,D)$ be a geometrically connected proper simple normal crossing surface over $k$.
Let $Z$ be a  subset of $D$.
Then, there exists a quadratic polynomial $K : \Div(X,D)\oplus \mathds{Z}\to \mathds{Z}$ affine in the last variable such that for every  effective divisor $R$ of $X$ supported on $D$, for every integer $r\geq 0$ and every object $\cM$ of $\MIC_r(X,D,R)$, we have
$$
\int_{Z} \delta\Irr\cM\leq K(R,r)
$$
\end{thm}
\begin{proof}
Let $R$ be an effective divisor supported on $D$.
Let $r$ be an integer. 
Let $\cM$ be an object  of $\MIC_r(X,D,R)$.
Let $\mathcal{V}$ be a finite cover of $X$ by open subsets such that for every $V\in \mathcal{V}$, the closed subset $V\cap D$ is either empty or is the vanishing locus of an algebraic  function $f_V : V\longrightarrow\mathds{A}^1$.
From \cref{irr_div_finitely_supported}, the $b$-divisor $\delta\Irr\cM$  is effective.
Hence, 
$$
\int_{Z} \delta\Irr\cM  \leq \sum_{V\in \cV} \int_{Z\cap V} \delta\Irr\cM   \leq \sum_{V\in \cV} \int_{D\cap V} \delta\Irr\cM
$$
Thus, we are left to prove  \cref{strong_chi} in the case where there exists an open set $V$ in $X$ such that $Z=D\cap V$ and such that $Z$ is the vanishing locus of an algebraic  function $f: V\longrightarrow\mathds{A}^1$.
Observe that the choices of $V$ and $f$ depend on $X$ and $D$ only.\\ \indent
From \cref{reduction_relative situation}, we can assume the existence of a proper map $h : X\to \mathds{P}^1$ such that $D$ contains $h^{-1}(0)$ and $Z$ is an open subset of $h^{-1}(0)$.
Since $\delta\Irr\cM$  is again effective, we can further assume that $Z$ is the reduced fibre of $h$ over $0$.
This case follows from \cref{bound_chi_Z}.
This concludes the proof of \cref{strong_chi}.

\end{proof}

We are now in position to prove cohomological boundedness for surfaces.

\begin{thm}\label{coho_boundedness_surface}
Let $(X,D)$ be a projective normal crossing surface over $k$.
Then, there exists a quadratic polynomial $C : \Div(X,D)\oplus \mathds{Z}\to \mathds{Z}$ affine in the last variable such that for every  effective divisor $R$ of $X$ supported on $D$, for every integer $r\geq 0$ and every object $\cM$ of $\MIC_r(X,D,R)$, we have
$$
\dim H^{*}(X,\DR \cM)\leq C(R,r)
$$
\end{thm}
\begin{proof}
From \cref{curve_case}, \cref{quasi_cor_reduction_to_bound_chi} and \cref{reduction_chi_to_C}, we are left to prove $\chi$-boundedness for surfaces over $\mathds{C}$.
From \cref{same_coho} and \cref{change_compactification}, we can suppose that $D$ is a simple normal crossing divisor.
We can furthermore suppose that $X$ is connected, and thus geometrically connected.
From \cref{calculation_chi_surface}, we are left to find a quadratic polynomial $K : \Div(X,D)\oplus \mathds{Z}\to \mathds{Z}$ affine in the last variable such that for every  effective divisor $R$ of $X$ supported on $D$, for every integer $r\geq 0$ and every object $\cM$ of $\MIC_r(X,D,R)$, we have we have $\int_D \delta\Irr\cM\leq K(R,r)$.
The existence of $K$ is ensured by \cref{strong_chi}, which finishes the proof of \cref{coho_boundedness_surface}.
\end{proof}

%

\subsection{Boundedness and turning locus}

To relate the turning locus to De Rham cohomology, we use the following :

\begin{thm}\label{Equivalent_condition_good_formal_structure}
Let $(X,D)$ be a  pair over $k$ where $D$ is smooth. 
Let $\cM$ be an object of $\MIC(X,D)$. 
Then, the following conditions are equivalent;
\begin{enumerate}\itemsep=0.2cm
\item $\cM$ has good formal structure along $D$.
\item The singular support of $\cM$ and $\End\cM$ is contained in $T^*_X X \cup T^*_D X$.
\end{enumerate}
If furthermore $k=\mathds{C}$, the above conditions are equivalent to the following condition.
\begin{enumerate}
\item[(3)] The complexes $\Irr^*_{D(\mathds{C})} \cM^{\an}$ and $\Irr^*_{D(\mathds{C})} \End\cM^{\an}$ are local systems on $D(\mathds{C})$ concen\-trated in degree $1$.
\end{enumerate}
\end{thm}
\begin{proof}
From \cref{CC_flat_base_change}, we can suppose that $k=\mathds{C}$.
The fact that $(1)$ implies $(2)$ follows from  \cref{CCgood}.
Suppose that $(2)$ holds. 
From  \cref{Kashiwara_loc_syst}, the cohomology sheaves of $\Irr^*_{D(\mathds{C})} \cM$ and $\Irr^*_{D(\mathds{C})} \End\cM$ are local systems on $D(\mathds{C})$.
From \cref{Irr_pervers}, the complexes $\Irr^*_{D(\mathds{C})}\cM$ and $\Irr^*_{D(\mathds{C})} \End\cM$ are thus necessarily concentrated in degree $1$ and $(3)$ follows.
The fact that $(3)$ implies $(1)$ is the main result of \cite{teyConjThese}.
\end{proof}

\begin{thm}\label{turning_point_and_integral}
Let $(X,D)$ be a normal crossing surface over an algebraically closed field of characteristic $0$.
Let $\cM$ be an object of $\MIC(X,D)$.
Let $P$ be a point in the smooth locus of $D$.
Then, $P$ is a turning point of $\cM$ if and only if 
$$
\int_P \left(\delta \Irr(\cM)+ \delta \Irr(\End\cM)\right)>0
$$
In particular,
$$
|\TL(\cM)|\leq |D^{\sing}|+\int_{D^{\sm}} \left(\delta \Irr(\cM)+ \delta \Irr(\End\cM)\right)
$$
\end{thm}
\begin{proof}
From \cref{Equivalent_condition_good_formal_structure}, $P$ is a turning point of $\cM$ if and only if $T^*_P X$ contributes to $CC(\cM)+CC(\End \cM)$.
Then, \cref{turning_point_and_integral} follows from \cref{computation_mult_TxX}.
\end{proof}

\begin{thm}\label{bound_degree_turning_locus}
Let $(X,D)$ be a projective normal crossing pair over $k$.
Let $X\to \mathds{P}$ be a closed embedding in some projective space.
Then, there exists a quadratic polynomial $K : \Div(X,D)\oplus \mathds{Z}\to \mathds{Z}$ such that for every  effective divisor $R$ of $X$ supported on $D$, for every integer $r\geq 0$ and every object $\cM$ of $\MIC_r(X,D,R)$, we have 
$$
\deg \TL(\cM)\leq K(R,r)
$$
\end{thm}
\begin{proof}
From \cref{change_compactification}, we can suppose that $D$ has simple normal crossings.
From \cref{regular_base_change}, we can suppose that $k$ is algebraically closed.
We can thus suppose that $X$ is geometrically connected.
If $X$ is a curve,  the turning locus is empty and there is nothing to prove.
If $X$ is a surface, $\TL(\cM)$ is a finite set of points.
Thus, $\deg \TL(\cM)=|\TL(\cM)|$.
On the other hand,  \cref{turning_point_and_integral} and \cref{strong_chi} ensures the existence of a quadratic polynomial $L : \Div(X,D)\oplus \mathds{Z}\to \mathds{Z}$ affine in the last variable such that for every  effective divisor $R$ of $X$ supported on $D$, for every integer $r\geq 0$ and every object $\cM$ of $\MIC_r(X,D,R)$, we have
$$
|\TL(\cM)|\leq |D^{\sing}|+L(R,r)+L(R,r^2)
$$
Hence, \cref{bound_degree_turning_locus} holds in that case.
Suppose that $X$ has dimension $d\geq 3$.
Let $N$ be the dimension of $\mathds{P}$.
Let $\cG$ be the Grassmannian of projective $N-d+2$-spaces in $\mathds{P}$.
Let $\eta$ be the generic point of $\cG$.
Let  $\mathcal{Q}$ be the universal family of projective $N-d+2$-spaces in $\mathds{P}$.
Then, there is  a commutative diagram with cartesian squares
$$
\xymatrix{
X_{\eta }\ar[r]  \ar[d]  &   X_{ \mathcal{Q}}  \ar[r] \ar[d]   &     X \ar[d]\\
          \mathcal{Q}_\eta     \ar[r] \ar[d]             &         \mathcal{Q}   \ar[r]   \ar[d]     & \mathds{P}\\
\eta \ar[r]             & \cG  & 
}
$$
Let $\eta$  as  subscript indicates a pull-back along $X_{\eta}\to X$.
Let $R$ be an effective divisor of $X$ supported on $D$. 
Let $r\geq 0$ be an integer. 
Let $\cM$ be an object of $\MIC_r(X,D,R)$.
From Kedlaya's purity \cref{purity}, the closed set $\TL(\cM)$ has pure dimension $d-2$.
Hence,  $\deg \TL(\cM)$  is the cardinal of $\TL(\cM)\cap H$ where $H$ is a generic projective $N-d+2$-spaces of $\mathds{P}$.
Thus,
$$
\deg \TL(\cM) = |\TL(\cM)_{\eta}|=|\TL(\cM_{\eta})|
$$
where the last equality follows by \cref{regular_base_change}.
Since $\cM_\eta$ lies in $\MIC_r(X_\eta,D_\eta,R_\eta)$ with $(X_\eta,D_\eta)$ a normal crossing surface over $\eta$,  \cref{bound_degree_turning_locus} follows from the surface case.
\end{proof}

As an application of the above theorem, we deduce the following

\begin{thm}\label{bound_turning_locus}
Let $(X,D)$ be a projective normal crossing pair over $k$.
Then, there exists a quadratic polynomial $K : \Div(X,D)\oplus \mathds{Z}\to \mathds{Z}$ such that for every  effective divisor $R$ of $X$ supported on $D$, for every integer $r\geq 0$ and every object $\cM$ of $\MIC_r(X,D,R)$, the number of irreducible components of $\TL(\cM)$ is smaller than $K(R,r)$.
\end{thm}
\begin{proof}
Let $R$ be an effective divisor of $X$ supported on $D$. 
Let $r\geq 0$ be an integer. 
Let $\cM$ be an object of $\MIC_r(X,D,R)$.
Let $Z_1,\dots, Z_n$ be the irreducible components of $\TL(\cM)$.
From Kedlaya's purity \cref{purity}, the $Z_i$ have the same dimension $\dim X-2$.
Hence, $\deg \TL(\cM)=\deg Z_1+\cdots+\deg Z_n$.
Since each $\deg Z_i $ is strictly positive, we deduce $n\leq \deg \TL(\cM)$.
Thus, \cref{bound_turning_locus} follows from \cref{bound_degree_turning_locus}.
\end{proof}

\section{Lefschetz recognition principle}\label{Lefschetz}

\begin{lem}\label{finding_good_curve}
Let $\mathds{P}$ be a projective space over $k$.
Let $X\subset \mathds{P}$ be a smooth  subvariety of pure dimension $n\geq 2$.
Let $C',C''\subset T^\ast X$ be  closed conical subsets of pure dimension $n$ where  the base of $C''$ has dimension at most $n-2$.
Put $C=C'\cup C''$.
Let $S\subset X$ be a closed subset of dimension $\leq n-2$ containing the base of $C''$.
Let $Y\subset \mathds{P}$ be a smooth hypersurface transverse to $X$ such that $X\cap Y \to X$ is $C'$-transversal and $X\cap Y \cap S$ has dimension $<n-2$.
Then, for every sufficiently generic hyperplanes $E_1,\dots, E_{n-2} \in \mathds{P}^{\vee}$, the following hold :
\begin{enumerate}\itemsep=0.2cm
\item 
The commutative diagram
$$
	\begin{tikzcd}
	X\cap E_1\cap \dots \cap E_{n-2} \arrow[rrrd, swap, bend right = 10, "i_{n-2}"]\arrow{r}  & \cdots \arrow{r}  & X\cap E_1 \cap E_2\arrow{r} \arrow[rd, swap, bend right = 5, "i_{2}"] & X\cap E_1    \arrow{d}{i_1}  	 \\
	&       &         &         X
			\end{tikzcd}
$$
is a diagram of smooth varieties such that for every $j=1,\dots, n-2$, the map
$$
X\cap E_1\cap \dots \cap E_{j} \to X\cap E_1\cap \dots \cap E_{j-1}
$$ 
is  $i_{j-1}^{\circ}(C)$-transversal.
\item The scheme $T:=X\cap Y\cap E_1\cap \dots \cap E_{n-2}$ is a smooth curve of $X$ avoiding $S$ such that  $T\to X$ is $C$-transversal.

\item The map $T\to X\cap E_1\cap \dots \cap E_{n-2}$ is $i_{n-2}^{\circ}(C)$-transversal.
\end{enumerate}
\end{lem}
\begin{proof}
The claim $(1)$ follows from an iterative use of \cref{generic_transversality} and \cref{easy_transversality}-(3).
Let us prove $(2)$. 
The scheme $X\cap Y$ is a smooth variety of dimension $n-1$ and 
$X\cap Y \cap S\subset X\cap Y $ is a closed subset of dimension $<n-2$.
For $E_1,\dots, E_{n-2} \in \mathds{P}^{\vee}$ sufficiently generic, Bertini theorem ensures that $T:=X\cap Y\cap E_1\cap \dots \cap E_{n-2} \subset X\cap Y$ is a smooth curve avoiding $S$.
An iterative use of \cref{generic_transversality} and \cref{easy_transversality}-(3) applied to  
$$
	\begin{tikzcd}
	T \arrow{r}  & \cdots \arrow{r}&  X\cap Y\cap E_1    \arrow{r} & X\cap Y \arrow{r}  & X
			\end{tikzcd}
$$
ensures that for $E_1,\dots, E_{n-2} \in \mathds{P}^{\vee}$ sufficiently generic, the map $T\to X$ is $C'$-transversal.
Since $T$ avoid $S$ and since $S$ contains the base of $C''$, the map $T\to X$ is also $C''$-transversal.
From \cref{easy_transversality}-(2), we deduce that $T\to X$ is $C$-transversal.
Claim $(2)$ is thus proved.
Claim $(3)$ follows from \cref{easy_transversality}-(3) applied to
$$
	\begin{tikzcd}
	T \arrow{r}  & X\cap E_1\cap \dots \cap E_{n-2}    \arrow{r}{i_{n-2}} &  X
			\end{tikzcd}
$$
\end{proof}

\begin{prop}\label{Iso_on_H0}
Let $(X,D)$ be a projective simple normal crossing pair of dimension  $n\geq 2$ over  $k$.  
Let $X\to \mathds{P}$ be a closed immersion in some projective space.
Let $\cM$ be an object of $\MIC(X,D)$.
Let $H$ be a hyperplane  such that
\begin{enumerate}\itemsep=0.2cm
\item $H$ is transverse to $X$ and $X\cap H\to X$ is $SS(\mathcal{O}_X(\ast D))$-transversal.

\item $X\cap H$ does not contain any irreducible component of  $\TL(\cM)$.
\end{enumerate}
Then, for every $E_1,\dots, E_{n-2} \in \mathds{P}^{\vee}$ sufficiently generic, the  restriction morphism
\begin{equation}\label{restriction_H0}
H^0(X,\DR \cM)\lra H^0(T,\DR \cM|_{T})
\end{equation}
is an isomorphism, where $T=X\cap H \cap E_1 \cap \dots \cap E_{n-2}$.

\end{prop}
\begin{proof}
Let $C''$ be the union of irreducible components of $SS(\cM)$ whose bases are subsets of $\TL(\cM)$.
Then,  \cref{CCgood} yields
$$
SS(\cM) \subset SS(\mathcal{O}_X(\ast D)) \bigcup C''
$$
Since  $\dim \TL(\cM)=n-2$, the conditions from \cref{finding_good_curve} are satisfied for $Y=H$, $C'=SS(\mathcal{O}_X(\ast D))$ and $C''$, and $S=\TL(\cM)$.
Thus, for every $E_1,\dots, E_{n-2} \in \mathds{P}^{\vee}$ sufficiently generic, 
the commutative diagram
$$
	\begin{tikzcd}
	T \arrow{r}  & 	X\cap E_1\cap \dots \cap E_{n-2} \arrow[rrrd, swap, bend right = 10, "i_{n-2}"]\arrow{r}  & \cdots \arrow{r}  & X\cap E_1 \cap E_2\arrow{r} \arrow[rd, swap, bend right = 5, "i_{2}"] & X\cap E_1    \arrow{d}{i_1}  	 \\
	& &       &         &         X
			\end{tikzcd}
$$
is a diagram of smooth varieties such that for every $j=1,\dots, n-2$, the map
$$
X\cap E_1\cap \dots \cap E_{j} \to X\cap E_1\cap \dots \cap E_{j-1}
$$ 
is  $i_{j-1}^{\circ}(SS(\cM))$-transversal and $T\to X\cap E_1\cap \dots \cap E_{n-2}$ is $i_{n-2}^{\circ}(SS(\cM))$-transversal.
From \cref{Cauchy-Kowaleska}, the horizontal and vertical arrows of the above diagram are thus non-characteristic for the successive restrictions of $\cM$.
From \cref{non_char_restriction_coho_DR}, we deduce that each arrow of the diagram 
$$
	\begin{tikzcd}
	H^0(X,\DR \cM) \arrow{r}  & H^0(X\cap H_1,\DR \cM|_{X\cap E_1})  \arrow{r}& \cdots  \arrow{r} & H^0(T,\DR \cM|_{T})
			\end{tikzcd}
$$
induced by the successive restrictions of $\cM$ are isomorphisms.
This concludes the proof of \cref{Iso_on_H0}.
\end{proof}

\begin{cor}\label{isomorphism_between_Hom}
Let $(X,D)$ be a projective simple normal crossing pair of dimension  $n\geq 2$ over  $k$.  
Let $X\to \mathds{P}$ be a closed immersion in some projective space.
Let $\cM_1,\cM_2$ be objects of $\MIC(X,D)$.
Let $H$ be a hyperplane such that
\begin{enumerate}\itemsep=0.2cm
\item $H$ is transverse to $X$ and $X\cap H\to X$ is $SS(\mathcal{O}_X(\ast D))$-transversal. 

\item   $X\cap H$ does not contain any irreducible component of  $\TL(\cHom(\cM_1,\cM_2))$.

\end{enumerate}
Then, for every $E_1,\dots, E_{n-2} \in \mathds{P}^{\vee}$ sufficiently generic, the restriction morphism
$$
\Hom_{\cD_X}(\cM_1,\cM_2)\lra \Hom_{\cD_{ T}}(\cM_1|_{ T},\cM_2|_{ T})
$$
is an isomorphism, where $T=X\cap H \cap E_1 \cap \dots \cap E_{n-2}$.

\end{cor}
\begin{proof}
Combine \cref{RHom_alg_DR} with \cref{Iso_on_H0}.
\end{proof}

\begin{thm}\label{Lefschetz_recognition_hyperplane}
Let $(X,D)$ be a projective simple normal crossing pair of dimension  $n\geq 2$ over  $k$.  
Let $X\to \mathds{P}$ be a closed immersion in some projective space.
Then, there exists a polynomial $K : \Div(X,D)\oplus \mathds{Z}\to \mathds{Z}$ of degree $4$ such that for every  effective divisor $R$ of $X$ supported on $D$, for every integer $r\geq 0$, every set $\mathcal{H}$ of  $K(R,r)$  hyperplanes satisfying the following conditions :
\begin{enumerate}\itemsep=0.2cm
\item $H$ is transverse to $X$ and $X\cap H\to X$ is $SS(\mathcal{O}_X(\ast D))$-transversal  for every  $H\in \mathcal{H}$
\item The $D\cap H, H\in \mathcal{H}$ are closed subsets of $D$ of pure dimension $n-2$ with two by two distinct irreducible components
\end{enumerate}
realizes the Lefschetz recognition principle for $\MIC_r(X,D,R)$ (\cref{realize_LRP}).
In particular, there is a dense open subset of $K(R,r)$-uples of hyperplanes realizing the Lefschetz recognition principle for $\MIC_r(X,D,R)$.
\end{thm}
\begin{proof}
From \cref{bound_turning_locus}, there exists a quadratic polynomial $L : \Div(X,D)\oplus \mathds{Z}\to \mathds{Z}$ such that for every  effective divisor $R$ of $X$ supported on $D$, for every integer $r\geq 0$ and every object $\cM$ of $\MIC_{4r^2}(X,D,2r^2 \cdot R)$, the number of irreducible components of $\TL(\cM)\subset D$ is smaller than $L(2r^2 \cdot R,4r^2)$.
Put
$$
K(R,r)=L(2r^2 \cdot R,4r^2) +1
$$
Choose a set $\mathcal{H}$ of $K(R,r)$   hyperplanes satisfying conditions $(1)$ and $(2)$ above.
We want to show that  $\mathcal{H}$ realizes the Lefschetz recognition principle for $\MIC_r(X,D,R)$.
Let $\cM_1,\cM_2 \in  \MIC_r(X,D,R)$ such that $\cM_1|_{X\cap H}$ and $\cM_2|_{X\cap H}$ are isomorphic for $H\in \mathcal{H}$.
In particular for every $H\in \mathcal{H}$ and every smooth subvariety $T\subset X\cap H$, the connections $\cM_1|_T$ and $\cM_2|_T$ are isomorphic.
We now want to show that $\cM_1$ and $\cM_2$ are isomorphic.
To do this, it is enough to show the existence of $H\in \mathcal{H}$ and a smooth subvariety  $T\subset X\cap H$ such that
$$
\Hom_{\cD_X}(\cM_{a},\cM_{b})\lra \Hom_{\cD_{  T}}(\cM_{a}|_{ T},\cM_{b}|_{T})
$$
is an isomorphism for every $a,b\in \{1,2\}$.
From \cref{isomorphism_between_Hom}, it is enough to show the existence of $H\in \mathcal{H}$ such that $X\cap H$ does not contain any irreducible component of  $\TL(\cHom(\cM_a,\cM_b)), a,b\in \{1,2\}$.
If we put 
$$
\cM := \bigoplus_{a, b\in \{1,2\}}\cHom(\cM_a,\cM_b)
$$
we have 
$$
\bigcup_{a,b\in \{1,2\}} \TL(\cHom(\cM_a,\cM_b)) \subset \TL(\cM)
$$
Since turning loci have pure dimension $n-2$ in virtue of  \cref{purity}, we deduce that 
the irreducible components of the $\TL(\cHom(\cM_a,\cM_b))$, $a,b\in \{1,2\}$ are irreducible components of $\TL(\cM)$.
Thus, it is enough to show the existence of $H\in \mathcal{H}$ such that $X\cap H$ does not contain any irreducible component of  $\TL(\cM)$.
Note from \cref{Irr_Hom} that $\cM$ is an object of $\MIC_{4r^2}(X,D,2r^2 \cdot R)$.
Hence, $\TL(\cM)$ has strictly less than $K(R,r)$ irreducible components.
Thus, condition $(2)$ above ensures the existence of the sought-after hyperplane.

\end{proof}

\section{Tannakian Lefschetz theorem} \label{Tannaka}

As usual, $k$ denotes a field of characteristic $0$.
For an abelian category $\cC$, we denote by $\cC^{\sesi}$ the full subcategory of $\cC$ spanned by the semisimple objects.
If $F : \cC\to \cD $ is an additive functor between abelian  categories, we let $F^{-1}(\cD^{\sesi})$ be the full subcategory of $\cC $ spanned by the objects sent  in $\cD^{\sesi}$ by $F$.\\ \indent

As a general reference for Tannakian categories, let us mention  \cite{Deligne_Milne}.
If $(\cC,\otimes_\cC)$ is a Tannakian category and if $M$ is an object of $\cC$, we denote by $\langle M\rangle$ the Tannakian subcategory of $\cC$ generated by $M$.
If $\omega : \cC \to \Vect_k$ is a neutralization for $(\cC,\otimes_\cC)$, we denote by $\pi_1( M,\omega)$ the Tannakian algebraic group of $\langle M\rangle$ at $\omega$.
\begin{ex}
Let $(X,D)$ be a   pair over $k$.
Then  $(\MIC(X,D),\otimes_{\cO_X(\ast D)})$ is an abelian  rigid tensor category over $k$.
For a morphism of pairs $f : (Y,E)\to (X,D)$  over $k$, the pull-back $f^+ : \MIC(X,D)\to \MIC(Y,E)$ is an exact tensor functor.
If furthermore $X$ is connected and $k$ is algebraically closed, the restriction to any closed point of $X\setminus D$ endows $(\MIC(X,D),\otimes_{\cO_X(\ast D)})$  with a structure of neutral Tannakian category.
\end{ex}

\begin{lem}\label{ss_object_tannakian_char_0}
Let $F: (\cC,\otimes_\cC) \to (\cD,\otimes_\cD)$ be an exact tensor functor between neutral Tannakian categories over $k$.
Then, $F^{-1}(\cD^{\sesi})$ is a neutral Tannakian subcate\-gory of $(\cC,\otimes_\cC)$.
\end{lem}
\begin{proof}
Let $f : X\rightarrow Y$ be a morphism in $F^{-1}(\cD^{\sesi})$ and let $K$ be its kernel in $\cC$. 
Since $F$ is exact, $F(K)\simeq \Ker F(f)$ is semisimple since $F(X)$ is semisimple \cite[4.14]{Milne_alg_gp}. 
Hence, $K$ lies in $F^{-1}(\cD^{\sesi})$ and similarly for the cokernel of $f$. 
Thus $F^{-1}(\cD^{\sesi})$ is an abelian subcategory of $\cC$.
Let $X,Y$ be objects of $F^{-1}(\cD^{\sesi})$.
Since $F$ is a tensor functor between rigid tensor categories, we have 
$$
F(\underline{\Hom}(X,Y))\simeq \underline{\Hom}(F(X),F(Y))\simeq  F(X)^{\vee}\otimes_{\cD} F(Y)
$$
where the first equivalence follows from  \cite[1.9]{Deligne_Milne} and the second from  \cite[1.7]{Deligne_Milne}.
Since the tensor product of two semisimple finite dimensional representations of a group is again semisimple  \cite[p.88]{Chevalley_Lie}, we know that $F(X)^{\vee}\otimes_{\cD} F(Y)$ is semisimple. 
Hence, $\underline{\Hom}(X,Y)$ lies in $F^{-1}(\cD^{\sesi})$.
Furthermore, the identity object of $\cC$ lies in $F^{-1}(\cD^{\sesi})$.
Hence, $F^{-1}(\cD^{\sesi})$ is stable under finite tensor product and dual.
Thus $F^{-1}(\cD^{\sesi})$  is a rigid abelian tensor subcategory of $(\cC,\otimes_\cC)$.
A neutralization of $(\cC,\otimes_\cC)$ induces a neutralization of $F^{-1}(\cD^{\sesi})$ and the proof of \cref{ss_object_tannakian_char_0} is complete.
\end{proof}

\begin{rem}
\cref{ss_object_tannakian_char_0} applied to the identity functor implies that $\cC^{\sesi}$ is a neutral tannakian subcategory of $(\cC,\otimes_\cC)$.
\end{rem}

The following ground-breaking theorem is due to Mochizuki \cite[13.2.3]{Mochizuki1} :
\begin{thm}\label{Mochizuki_semi_simplicity}
Let $(X,D)$ be a projective simple normal crossing pair over $k$.
Let $\cL$ be an ample line bundle on $X$.
Let $\cM$ be a simple object of $\MIC(X,D)$.
Then, there is an integer $m_0$ such that for every $m\geq m_0$, for every generic hyperplane  $H$  of $\cL^m$, the restriction $\cM|_{X\cap H}$ is a simple object of $\MIC(X\cap H,D\cap H)$.
\end{thm}

\begin{rem}
In the statement of \cref{Mochizuki_semi_simplicity}, the integer $m_0$ depends  on $\cM$.
\end{rem}

\begin{cor}\label{generic_restriction_sends_ss_to_ss}
Let $(X,D)$ be a  connected projective simple normal crossing pair over $k$.
Assume that $k$ is algebraically closed.
Let $\cL$ be an ample line bundle on $X$.
Let $\cM$ be an object of $\MIC(X,D)$.
Then, there is an integer $m_0$ such that for every $m\geq m_0$, for every generic hyperplane $H$  of $\cL^m$, for every semisimple object $\cN$ of $\langle \cM\rangle$, the restriction $\cN|_{X\cap H}$ is a semisimple object of $\MIC(X\cap H,D\cap H)$.
\end{cor}
\begin{proof}
Since $\langle \cM\rangle$ admits a tensor generator, the general yoga of neutral tannakian categories \cite[2.20 (b)]{Deligne_Milne} ensures that so does $\langle \cM\rangle^{\sesi}$.
Let $\cN$ be a tensor generator for $\langle \cM\rangle^{\sesi}$.
From Mochizuki's \cref{Mochizuki_semi_simplicity} applied to $\cN$, there exists an integer $m_0$ such that for every $m\geq m_0$, the set  of hyperplanes $H$ of $\cL^m$ with $\cN|_{X\cap H}$ semisimple contains a dense open subset $V$ of $\mathds{P}(\Gamma(X,\cL^m))^{\vee}$.
Let $H\in V(k)$.
Let $F :\langle \cM\rangle  \to \MIC(X\cap H, D\cap H)$ be the exact tensor functor given by the restriction to $X\cap H$.
Then, \cref{ss_object_tannakian_char_0} ensures that $F^{-1}(\MIC(X\cap H, D\cap H)^{\sesi})$ is a neutral tannakian subcategory of $\langle \cM\rangle$.
Since it contains $\cN$, it also contains $\langle \cM\rangle^{\sesi}$.
The proof of \cref{generic_restriction_sends_ss_to_ss} is thus complete.
\end{proof}

\begin{lem}\label{fully_faithful_restriction}
Let $(X,D)$ be a  projective pair of dimension at least $2$ over  $k$.
Assume that $k$ is uncountable.
Let $\cL$ be a very ample line bundle on $X$.
Let $\cC$ be a full subcategory of $\MIC(X,D)$ spanned by a countable number of objects.
Then, for every very  generic hyperplane $H$ of $\cL$, the restriction to $X\cap H$ induces a fully faithful functor on $\cC$. 
\end{lem}
\begin{proof}
Let $\cM_1$ and $\cM_2$ be objects in $\cC$.
From \cref{generic_transversality}, the set of hyperplanes $H$ of $\cL$ 
transverse to $X$ such that  $X\cap H\to X$ is non-characteristic for $\cHom(\cM_1,  \cM_2)$ contains a dense open subset $V(\cM_1,\cM_2)$.
From \cref{isomorphism_between_Hom}, the set of hyperplanes $H$ of $\cL$ transverse to $X$ such that 
$$
\Hom_{\cD_X}(\cM_1,\cM_2)\lra \Hom_{\cD_{X\cap H}}(\cM_1|_{X\cap H},\cM_2|_{X\cap H})
$$
is an isomorphism contains $V(\cM_1,\cM_2)$.
Then the intersection  of the $V(\cM_1,\cM_2)$ for $\cM_1, \cM_2 \in \cC$ gives the sought-after set of very generic hyperplanes.
\end{proof}

Before proving  \cref{Lefchetz_single_connection}, we recall the following  abstract lemma  \cite[1.1.4]{Stalder}.

\begin{lem}\label{Stalder}
Let $F : \cC\to \cD$ be an exact $k$-linear fully faithful functor between abelian categories over $k$.
Assume that for every object $M$ of $\cC$, the length of $M$ is finite and the endomorphism algebra of $M$ has finite dimension over $k$.
Assume that $F$ sends semisimple objects of $\cC$ to semisimple objects of $\cD$.
Then, the essential image of $F$ is closed under subquotients in $\cD$.
\end{lem}

\begin{rem}\label{finite_length}
Let  $X$ be a smooth variety over $k$.
Let $\cM$ be an holonomic $\cD_X$-module.
Then $\cM$ has finite length. 
Furthermore, the space
$$
\Hom_{\cD_X}(\cM,\cM)\simeq H^0(X,\DR \cM^\vee\otimes^L_{\cO_X}\cM)
$$ 
is finite dimensional over $k$ as a consequence of \cref{bound_cohomology}.
\end{rem}

\begin{thm}\label{Lefchetz_single_connection}
Let $(X,D)$ be a connected projective simple  normal crossing pair of dimension at least $2$ over $k$.
Assume that $k$ is uncountable and algebraically closed.
Let $\cL$ be an ample line bundle on $X$.
Let $\cM$ be an object of $\MIC(X,D)$.
Then, there is an integer $m_0$ such that for every $m\geq m_0$, for every very generic hyperplane $H$  of $\cL^m$ and every point $x$ of $(X\setminus D)\cap H$, the induced morphism of differential Galois groups 
$$
\pi_1(\cM|_{X\cap H},x)\lra \pi_1(\cM,x)
$$ 
is an isomorphism.
\end{thm}
\begin{proof}
We have to show that for every  large enough integer $m$, for every very generic hyperplane $H$  of $\cL^m$, the functor $F_H : \langle \cM \rangle\to \langle \cM|_{X\cap H}\rangle$ of neutralized Tannakian categories induced by the restriction to $X\cap H$ is an equivalence.
This amounts to show that for every large enough $m$ and for every very generic $H$, the functor $F_H $ is fully faithful and its essential image is closed under taking subquotients.
From \cref{finite_length},  \cref{Stalder} applies to $F_H$ and we are left to show that for every large enough $m$ and for every very generic $H$, the functor $F_H$ is fully faithful and sends semisimple objects to semisimple objects.
This follows from   \cref{generic_restriction_sends_ss_to_ss} and \cref{fully_faithful_restriction}.
\end{proof}

\appendix

\section{Base field extension lemmas}\label{base_field_section}

The goal of this purely technical appendix is to prove various compatibilities with base field extension of some  constructions considered in this paper.

\subsection{Cohomological boundedness and base field extension}

\begin{lem}\label{coho_bound_reduction_complex}
Let $d\geq 0$ be an integer. 
Let $k$ be a field of characteristic $0$.
Then, cohomological boundedness holds in dimension $d$ over $k$ if it holds in dimension d over $\mathds{C}$.
\end{lem}
\begin{proof}
Let $(X,D)$ be a normal crossing pair over $k$.
Let $R$ be an effective divisor of $X$ supported on $D$.
Let $\kappa$ be a finitely generated extension of $\mathds{Q}$ such that $(X,D)$ descends to a normal crossing pair $(X_\kappa, D_\kappa)$.
Put $S:=(\fdeg R)\cdot D_\kappa$.
Then,  $R\leq S_k$ and $S$ depends linearly on $R$.
Choose an embedding $\kappa \to \mathds{C}$.
Let $\cM$ be an object of $\MIC(X,D,R)$.
Then, there exists an intermediate extension $\kappa \subset \kappa_{\cM}\subset k$ such that $\kappa_{\cM}/\kappa$ is finitely generated and such that $\cM$ descends to an object $\cM_{\kappa_{\cM}}$ of $\MIC(X_{\kappa_{\cM}},D_{\kappa_{\cM}})$, where $(X_{\kappa_{\cM}},D_{\kappa_{\cM}})$ is the pull-back of $(X_{\kappa},D_{\kappa})$ over $\kappa_{\cM}$.
From \cref{Irr_divisor_base_change}, we have 
$$
\Irr(X_{\kappa_{\cM}},\cM_{\kappa_{\cM}})_k\leq R\leq S_k = (S_{\kappa_{\cM}})_k
$$ 
Thus, 
$$
\Irr(X_{\kappa_{\cM}},\cM_{\kappa_{\cM}}) \leq S_{\kappa_{\cM}}
$$
Hence, $\cM_{\kappa_{\cM}}$ is an object of $\MIC(X_{\kappa_{\cM}},D_{\kappa_{\cM}},S_{\kappa_{\cM}})$.
Choose an embedding $\kappa_{\cM}\to \mathds{C}$ over $\kappa$.
Observe that the complex  variety 
$$
\mathds{C}\times_{\kappa_{\cM}} X_{\kappa_{\cM}}\simeq \mathds{C}\times_{\kappa} X_{\kappa}
$$
 does not depend on $\kappa_{\cM}$, and thus does not depend on $\cM$.
 Furthermore, $\mathds{C}\times_{\kappa} D_{\kappa}$ is a normal crossing divisor and $\mathds{C}\times_{\kappa} X_{\kappa}$ is projective over $\mathds{C}$ if $X$ is projective over $k$.
 Let us assume that $X$ is projective of dimension $d$ over $k$.
By assumption, let $C : \Div(\mathds{C}\times_{\kappa} X_{\kappa},\mathds{C}\times_{\kappa} D_{\kappa})\oplus \mathds{Z}\to \mathds{Z} $ be a polynomial of degree at most $d$ affine in the last variable such that cohomological boundedness holds for $(\mathds{C}\times_{\kappa} X_{\kappa},\mathds{C}\times_{\kappa} D_{\kappa})$ with bound $C$.
From  \cref{Irr_divisor_base_change}, we know that $(\cM_{\kappa_{\cM}})_{\mathds{C}}$ is an object of $\MIC(\mathds{C}\times_{\kappa} X_{\kappa},\mathds{C}\times_{\kappa} D_{\kappa},S_{\mathds{C}})$ with $S_{\mathds{C}}=(\fdeg R)\cdot (\mathds{C}\times_{\kappa} D_{\kappa})$.
From \cref{coho_change_base_field}, we deduce
$$
\dim H^{\ast}(X,\DR \cM)\leq C(S_{\mathds{C}},r)
$$
and the proof of \cref{coho_bound_reduction_complex} is complete.
\end{proof}

\subsection{Partial discrepancy and base field extension}
The goal of this subsection is to prove \cref{Change_base_field_integral}.
\cref{Change_base_field_integral}  is used to reduce the computation of the characteristic cycle of a connection on a surface to the case where the base field is $\mathds{C}$.

\begin{lem}\label{Change_base_field}
Let $(X,D)$ be a geometrically connected normal crossing surface over $k$.
Let $k \subset K$ be a field extension.
Let $w$ be a divisorial valuation on $X$ and let $v$ be an extension of $w$ on $X_K$.
Let $\cM$ be an object of $\MIC(X,D)$.
Then, 
$$
(\Irr \cM_K)(v)=(\Irr \cM)(w) 
$$
\end{lem}
\begin{proof}
Let $f : Y\to X$ be a modification of $X$ such that $w$ is centred at a divisor $E$ of $Y$.
Then, $v$ is centred at an irreducible component $F$ of $E_K$.
Let $\eta_E$ be the generic point of $E$ and let $\eta_F$ be the generic point of $F$.
Observe that a uniformizer of $\mathcal{O}_{Y,\eta_E}$ pulls-back along $Y_K\to Y$ to a uniformizer of $\mathcal{O}_{Y_K,\eta_F}$.
Hence, $\irr(F, f_K^+\cM_K)=\irr(E,f^+\cM)$ and the equality  $(\Irr \cM_K)(v)=(\Irr \cM)(w)$ follows. 
\end{proof}

\begin{lem}\label{Change_base_field and partial discrepancy of irregularity}
Let $(X,D)$ be a geometrically connected normal crossing surface over $k$.
Let $k \subset K$ be a field extension.
Let $w$ be a divisorial valuation on $X$ and let $v$ be an extension of $w$ on $X_K$.
Let $\cM$ be an object of $\MIC(X,D)$.
Then, 
$$
(\delta \Irr \cM_K)(v)=(\delta\Irr \cM)(w)
$$
\end{lem}
\begin{proof}
If the centre of $w$ on $X$ is a divisor, then the centre of $v$ on $X_K$ is also a divisor.
In that case, $(\delta \Irr \cM_K)(v)=(\delta\Irr \cM)(w)=0$.
Hence, we can suppose that $w$ is centred at a closed point of $X$.
Then, there exists a modification  $p : Y\to X$ of $X$ such that $w$ is centred at a closed point $P$ of $Y$ and $w$ is centred at the exceptional divisor $E$ of the blow-up $Y'\to Y$ of $Y$ at $P$.
Then, $v$ is centred at a closed point $Q$ of $Y_K$ lying over $P$.
Let $S\subset Y_K$ be the set of closed points of $Y_K$ lying over $P$.
Let $Y_K'$ be the blow-up of $Y_K$ at $S$.
Since blowing-up commutes with flat base change, there is a canonical cartesian diagram of varieties
$$
\xymatrix{
Y_K' \ar[r]  \ar[d]  &  Y'   \ar[d]  \\
           Y_K    \ar[r]           &        Y  
}
$$
Hence, $v$ is centred at the exceptional divisor $F$ of $Y_K'$ lying over $Q$.
Thus, \cref{computation_discrepancy} applies to $v$ and $w$.
Consider the cartesian diagram 
$$
\xymatrix{
Y_K \ar[r]  \ar[d]_-{f_K}  &  Y   \ar[d]^-{f}  \\
           X_K    \ar[r]           &       X  
}
$$
Since $X_K\to X$ is étale, so is the map $f^*_K D_K\to f^*D$.
In particular the point $Q$ is a singular point of $f^*_K D_K$ if and only if $P$ is a singular point of $f^*D$.
Thus, if $Q$ is a singular point of $f^*_K D_K$, we have 
$$
(\delta \Irr \cM_K)(v)=(\delta\Irr \cM)(w)=0
$$
Otherwise, 
\begin{align*}
(\delta \Irr \cM_K)(v) & = ((\Irr \cM_K)(Y_K))(v)-(\Irr \cM_K)(v)  \\
                                    & =((\Irr \cM)(Y))(w)-(\Irr \cM)(w)\\
                                    & =(\delta\Irr \cM)(w)
\end{align*}
where the second equality follows from \cref{Change_base_field}.
\cref{Change_base_field and partial discrepancy of irregularity} is thus proved.
\end{proof}

\begin{lem}\label{Change_base_field_integral_inequality}
Let $(X,D)$ be a geometrically connected normal crossing surface over $k$.
Let $k \subset K$ be a field extension.
Let  $P$ be any point of $X$. 
Let $Q$ be a   point of $X_K$ lying over $P$. 
Then, 
$$
\int_P \delta\Irr \cM\leq \int_Q \delta\Irr \cM_K
$$
\end{lem}
\begin{proof}
Every divisorial valuation on $X$ centred at $P$ admits an extension  to a divisorial valuation on $X_K$ centred at $Q$.
Then, \cref{Change_base_field_integral_inequality} follows from \cref{Change_base_field and partial discrepancy of irregularity}.
\end{proof}

\begin{lem}\label{Change_base_field_integral_k_alg_closed}
Let $(X,D)$ be a  connected normal crossing surface over $k$.
Assume that $k$ is algebraically closed.
Let $k \subset K$ be a field extension.
Let  $P$ be a  rational point of $X/k$. 
Let $Q$ be the unique rational point of $X_K/K$ lying over $P$. 
Then
$$
\int_P \delta\Irr \cM= \int_Q \delta\Irr \cM_K
$$
\end{lem}
\begin{proof}
Let $f: Y\to X$ be a modification such that $f^+\cM$ has good formal structure. 
In particular, $f^+_K\cM_K$ has good formal structure. 
Thus, $\Irr \cM=\Irr(Y,f^+\cM)$ in $\bDiv(X)$ and  $\Irr \cM_K=\Irr(Y_K,f^+_K\cM_K)$ in $\bDiv(X_K)$.
Hence, \cref{image_delta_of_nef_Cartier} ensures that  $\delta\Irr \cM$ is supported on the set $A$ of divisorial valuations  associated with the irreducible components  of $f^*D$.
Similarly, $\delta\Irr \cM_K$ is supported on the set $B$ of divisorial valuations associated with the irreducible components of $f^*_K D_K$.
Let $F$ be an irreducible component of $f^*D$.
Since $k$ is algebraically closed, $F$ is geometrically irreducible.
Thus,  $F_K$ is irreducible.
Hence, the map $f^*_K D_K\to f^*D$ induces a bijection on the sets of irreducible components.
Thus, the map $B\to A$ induced by restriction is bijective.
Let $A_P$ be the subset of $A$ of valuations centred at $P$ on $X$.
Let $B_Q$ be the subset of $B$ of valuations centred at $Q$ on $X_K$.
Then, the bijection $B\to A$ induces a bijection $B_Q\to A_P$.
Hence, 
$$
\int_P \delta\Irr \cM = \sum_{w\in A_P}(\delta\Irr \cM )(w)= \sum_{v\in B_Q}(\delta\Irr \cM_K )(v) =  \int_Q \delta\Irr \cM_K
$$
where the second equality follows from \cref{Change_base_field and partial discrepancy of irregularity}.
\cref{Change_base_field_integral} is thus proved.
\end{proof}

\begin{lem}\label{Change_base_field_integral}
Let $(X,D)$ be a geometrically connected normal crossing surface over $k$.
Let $k \subset K_1$ and  $k \subset K_2$ be field extensions where $K_1$ and $K_2$ are algebraically closed.
Let  $P$ be a  closed point of $X$. 
For $i=1,2$, let $Q_i$ be a rational point of $X_{K_i}/K_i$ lying over $P$.
Then, 
$$
\int_{Q_1} \delta\Irr \cM_{K_1}= \int_{Q_2} \delta\Irr \cM_{K_2}
$$
\end{lem}
\begin{proof}
For $i=1,2$, let $k_i$ be the algebraic closure of $k$ in $K_i$.
Since $K_1$ and $K_2$ are algebraically closed, $k_1$ and $k_2$ are isomorphic over $k$.
Let $P_i$ be the image of $Q_i$ via $X_{K_i}\to X_{k_i}$.
Consider the following commutative diagram over $k$
$$
\xymatrix{
Q_1   \ar[r]  \ar[d]  & P_1  \ar[d]       \ar[r]      &        P    \ar[d]          &  P_2   \ar[d]       \ar[l]       &    Q_2   \ar[d]    \ar[l]  \\
           X_{K_1}    \ar[r]     & X_{k_1}     \ar[r]     &        X     &      X_{k_2}    \ar[l]        & X_{K_2}    \ar[l]   
}
$$
Choose an isomorphism between $k_1$ and $k_2$ over $k$.
Let $P_{21}$ be the pull-back of $P_2$ along the induced isomorphism $X_{k_1}\to X_{k_2}$.
Then, 
$$
\int_{P_2} \delta\Irr \cM_{k_2}= \int_{P_{21}} \delta\Irr \cM_{k_1}
$$
Since $P_{21}$ and $P_1$ are two rational points of $X_{k_1}/k_1$ lying over $P$, there exists an automorphism $\sigma$ of $k_1$ over $k$ such that $\sigma(P_{1})=P_{21}$.
Thus, 
$$
 \int_{P_{21}} \delta\Irr \cM_{k_1} =  \int_{P_1} \delta\Irr \sigma^+\cM_{k_1}= \int_{P_{1}} \delta\Irr \cM_{k_1} 
$$
From \cref{Change_base_field_integral_k_alg_closed}, we deduce
$$
\int_{Q_{1}} \delta\Irr \cM_{K_1} =\int_{P_{1}} \delta\Irr \cM_{k_1} =\int_{P_2} \delta\Irr \cM_{k_2}=\int_{Q_2} \delta\Irr \cM_{K_2}
$$
and \cref{Change_base_field_integral} is proved.
\end{proof}

\begin{rem}\label{invariance_rational_points}
Applying \cref{Change_base_field_integral} to the case where $K_1=K_2=K$, we conclude that the integral $\int_{Q} \delta\Irr \cM_{K}$ does not depend on the choice of a rational point $Q$ of $X_{K}/K$ lying over $P$.
\end{rem}

\begin{lem}\label{Change_base_field_integral_general}
Let $(X,D)$ be a geometrically connected normal crossing surface over $k$.
Let $k \subset K_1$ and  $k \subset K_2$ be field extensions where $K_1$ and $K_2$ are algebraically closed.
Let  $Z$ be a  locally closed subset of $X$. 
Then
$$
\int_{Z_{K_1}} \delta\Irr \cM_{K_1}= \int_{Z_{K_2}} \delta\Irr \cM_{K_2}
$$
\end{lem}
\begin{proof}
We are going to use \cref{Change_base_field_integral}.
We first argue that the points of $Z_{K_1}$ and $Z_{K_2}$ not sent to a closed point of $Z$ do not contribute to the above integrals.
From \cref{regular_base_change}, we have $\TL(\cM_{K_1})=\TL(\cM)_{K_1}$.
Hence, \cref{integral_over_Z_and_points} yields 
$$
\int_{Z_{K_1}} \delta\Irr \cM_{K_1}= \int_{(Z\cap \TL(\cM))_{K_1}}  \delta\Irr \cM_{K_1}
$$
Since $\TL(\cM)$ is a finite set of closed points of $X$, so is $Z\cap \TL(\cM)$.
In  particular, $(Z\cap \TL(\cM))_{K_1}$ is a finite set of rational points of $X_{K_1}$ lying above  closed points of $X$.
For every point $P$ of $Z\cap \TL(\cM)$, choose a rational point $Q_1(P)$ of $X_{K_1}$ above $P$ and let us denote by $\deg P$ the degree of $P$ over $k$.
Since $K_1$ is algebraically closed, $X_{K_1}$ admits exactly $\deg P$ rational points above $P$.
From \cref{invariance_rational_points}, we deduce 
$$
\int_{Z_{K_1}} \delta\Irr \cM_{K_1} = \sum_{Q_1\in(Z\cap \TL(\cM))_{K_1} } \int_{Q_1}  \delta\Irr \cM_{K_1}= \sum_{P\in Z\cap \TL(\cM)} \deg P \cdot \int_{Q_1(P)}  \delta\Irr \cM_{K_1}
$$
Since the same formula holds with  $K_2$, \cref{Change_base_field_integral_general} follows from  \cref{Change_base_field_integral}.
\end{proof}

\bibliographystyle{amsalpha}

\bibliography{biblio}

\providecommand{\bysame}{\leavevmode\hbox to3em{\hrulefill}\thinspace}
\providecommand{\MR}{\relax\ifhmode\unskip\space\fi MR }
\providecommand{\MRhref}[2]{%
  \href{http://www.ams.org/mathscinet-getitem?mr=#1}{#2}
}
\providecommand{\href}[2]{#2}
\begin{thebibliography}{BBDG18}

\bibitem[{And}07]{andre}
Y.~{André}, \emph{Structure des connexions méromorphes formelles de plusieurs
  variables et semi-continuité de l'irrégularité}, Invent. math.
  \textbf{170} (2007).

\bibitem[BBDG18]{BBD}
A.~{Beilinson}, J.~{Bernstein}, P.~{Deligne}, and O.~{Gabber},
  \emph{{{Faisceaux pervers}}}, {{Astérisque}}, vol. 100, {2018}.

\bibitem[{Bei}16]{Bei}
A.~{Beilinson}, \emph{{Constructible sheaves are holonomic}}, {{Sel. Math. New.
  Ser.}} \textbf{22} (2016).

\bibitem[BM89]{BMUniformization}
E.~{Bierstone} and P.~{Milman}, \emph{Uniformization of analytic spaces},
  Journal of the {A}merican {M}athematical {Society} \textbf{2} (1989).

\bibitem[{Bor}87]{Borel}
A.~{Borel et al.}, \emph{Algebraic $\mathcal{D}$-modules}, Perspectives in
  Math., vol.~2, Academic Press, 1987.

\bibitem[{Che}54]{Chevalley_Lie}
C.~{Chevalley}, \emph{Théorie des groupes de lie. vol. iii}, {A Series of
  Modern Surveys in Mathematics}, 1954.

\bibitem[{Cos}97]{Cossart_surface}
V.~{Cossart}, \emph{{Désingularisation des surfaces}}, Resolution of
  Singularities, {Progress in Mathematics}, vol. 181, Birkhäuser, 1997.

\bibitem[CPR02]{Cossart_Piltant_Lopez}
V.~{Cossart}, O.~{Piltant}, and Ana~J. {Reguera-Lopez}, \emph{{Divisorial
  Valuations Dominating Rational Surface Singularities}}, Valuation Theory and
  Its Applications, Volume I, {Fields Institute Communications}, vol.~32, Amer.
  Math. Soc., 2002.

\bibitem[{de }17]{Cataldo}
Mark Andrea~A. {de Cataldo}, \emph{{Perverse sheaves and the topology of
  algebraic varieties}}, Geometry of moduli spaces and representation theory,
  {IAS/Park City Math. Ser.}, vol.~24, Amer. Math. Soc., 2017.

\bibitem[{Del}70]{Del}
P.~{Deligne}, \emph{Equations différentielles à points singuliers
  réguliers}, Lecture {N}otes in {M}athematics, vol. 163, Springer-{V}erlag,
  1970.

\bibitem[{Del}07]{DeligneLettreMalgrange}
\bysame, \emph{Lettre à {M}algrange. 20 décembre 1983}, Singularités
  irrégulières (Société {M}athématique~de {F}rance, ed.), Documents
  {M}athématiques, vol.~5, 2007.

\bibitem[DM82]{Deligne_Milne}
P.~{Deligne} and J.S. {Milne}, \emph{{Tannakian Categories.}}, Hodge Cycles,
  Motives, and Shimura Varieties, {Lecture Notes in Mathematics}, vol. 900,
  Springer, 1982.

\bibitem[{Dub}84]{Dubson}
A.~{Dubson}, \emph{{Formule pour l'Indice des complexes constructibles et des
  Modules Holonomes}}, {{C. R. Acad. Sc. Paris}} \textbf{298} (1984).

\bibitem[{Esn}14]{IAS_Video}
H.~{Esnault}, \emph{The étale fundamental group of a smooth variety in
  characteristic p>0 and its relation to stratifications}, {{Workshop on
  fundamental groups and periods }} (2014), Video
  \url{https://www.ias.edu/video/fgp/2014/1013-HélèneEsnault}.
  
  
\bibitem[Ful98]{Fulton}
W.~{Fulton}, \emph{Intersection Theory},
  Springer-{V}erlag, 1998.  
  

\bibitem[{Gab}81]{Gabber_CC}
O.~{Gabber}, \emph{{The integrability of the characteristic variety}}, {{Amer.
  J. Math.}} \textbf{103} (1981).

\bibitem[HT21]{HuTeyssier}
H.~{Hu} and J.-B. {Teyssier}, \emph{Characteristic cycle and wild ramification
  for nearby cycles of étale sheaves}, J. für die Reine und Angew. Math.
  \textbf{776} (2021).

\bibitem[HTT00]{HTT}
R.~{Hotta}, K.~{Takeuchi}, and T.~{Tanisaki}, \emph{$\mathcal{D}$-{M}odules,
  {P}erverse {S}heaves, and {R}epresentation {T}heory}, vol. 236, Birkhauser,
  2000.

\bibitem[{Hu}19]{Hu_log_ram}
H.~{Hu}, \emph{Logarithmic ramifications of étale sheaves by restricting to
  curves}, {{IMRN}} \textbf{2019} (2019).

\bibitem[{Hu}21]{Hu_general_Leal}
\bysame, \emph{{Semi-continuity of conductors and ramification bound of nearby
  cycles. Preprint}}, 2021.

\bibitem[{Kas}75]{Ka2}
M.~{Kashiwara}, \emph{On the maximally overdetermined systems of linear
  differential equations {I}}, Publ. {RIMS} \textbf{10} (1975).

\bibitem[{Kas}85]{KaIndex}
\bysame, \emph{Index theorem for constructible sheaves}, Systèmes
  différentiels et singularités, Astérisque, vol. 130, Soc. Math. France,
  1985.

\bibitem[{Kas}95]{TheseKashiwara}
\bysame, \emph{{Algebraic study of systems of partial differential equations.
  Thesis, Tokyo University, 1970 (Translated by Andrea D'Agnolo and Pierre
  Schneiders, with a foreword by Pierre Schapira)}}, Mémoire de la {SMF}
  \textbf{63} (1995).

\bibitem[{Kat}21]{H.Kato}
H.~{Kato}, \emph{Wild ramification, the nearby cycle complexes, and the
  characteristic cycles of $\ell$-adic sheaves}, {Algebra and Number Theory}
  \textbf{15} (2021).

\bibitem[{Ked}11]{Kedlaya2}
K.~{Kedlaya}, \emph{Good formal structures for flat meromorpohic connexions
  {II}: excellent schemes}, J. Amer. Math. Soc. \textbf{24} (2011).

\bibitem[{Ked}19]{Kedlaya3}
\bysame, \emph{{Good formal structures on flat meromorphic connections, III:
  Irregularity and turning loci}}, {To appear in the Publications of RIMS}
  (2019).

\bibitem[KS90]{KS}
M.~{Kashiwara} and P.~{Schapira}, \emph{Sheaves on manifolds},
  Springer-{V}erlag, 1990.

\bibitem[{Lau}83]{Laumon_chi_D_module}
G.~{Laumon}, \emph{{Euler-Poincare characteristic of the De Rham complex of a
  $D$-module}}, Algebraic {Geometry}, {Lecture Notes in Mathematics}, vol.
  1016, Springer-{Verlag}, 1983.

\bibitem[{Mal}71]{Mal71}
B.~{Malgrange}, \emph{Sur les points singuliers des {\'e}quations
  diff{\'e}rentielles}, S{\'e}minaire d'{\'e}quations aux d{\'e}riv{\'e}es
  partielles ({Polytechnique}) (1971).

\bibitem[{Mal}96]{Reseaucan}
\bysame, \emph{Connexions méromorphes 2: Le réseau canonique}, Inv. {M}ath.
  \textbf{124} (1996).

\bibitem[{Meb}82]{DualityMebkhout}
Z.~{Mebkhout}, \emph{{Théorèmes de bidualité locale pour les
  $\mathcal{D}_X$-modules holonomes}}, {Arkiv für Math.} \textbf{20} (1982).

\bibitem[{Meb}90]{Mehbgro}
\bysame, \emph{Le théorème de positivité de l’irrégularité pour les
  $\mathcal{D}_{X}$-modules}, The {G}rothendieck {F}estschrift {III}, vol.~88,
  Birkhäuser, 1990.

\bibitem[{Meb}04]{Mehbsmf}
\bysame, \emph{Le théorème de positivité, le théorème de comparaison et le
  théorème d'existence de {R}iemann}, Éléments de la théorie des systèmes
  différentiels géométriques, Cours du {C.I.M.P.A.}, Séminaires et
  Congrès, vol.~8, SMF, 2004.

\bibitem[{Mil}17]{Milne_alg_gp}
J.S. {Milne}, \emph{{{Algebraic Groups: The Theory of Group Schemes of Finite
  Type over a Field}}}, Cambridge Studies in Advanced Mathematics, vol. 170,
  {{Cambridge University Press}}, 2017.

\bibitem[MM04]{MM}
P.~{Maisonobe} and Z.~{Mebkhout}, \emph{Le th{\'e}or{\`e}me de comparaison pour
  les cycles {\'e}vanescents}, Éléments de la théorie des systèmes
  différentiels géométriques, Cours du {C.I.M.P.A.}, Séminaires et
  Congrès, vol.~8, SMF, 2004.

\bibitem[{Moc}11]{Mochizuki1}
T.~{Mochizuki}, \emph{Wild {H}armonic {B}undles and {W}ild {P}ure {T}wistor
  $\mathcal{D}$-modules}, Ast{\'e}risque, vol. 340, SMF, 2011.

\bibitem[{Nit}93]{Nitsure}
N.~{Nitsure}, \emph{Moduli of semistable logarithmic connections}, {{J. Amer.
  Math. Soc.}} \textbf{6} (1993).

\bibitem[{Sab}00]{Sabbahdim}
C.~{Sabbah}, \emph{Equations diff{\'e}rentielles {\`a} points singuliers
  irr{\'e}guliers et ph{\'e}nom{\`e}ne de {S}tokes en dimension 2},
  Ast{\'e}risque, vol. 263, {SMF}, 2000.

\bibitem[{Sab}17]{SabRemar}
\bysame, \emph{{A remark on the {Irregularity Complex}}}, {Journal of
  Singularities} \textbf{16} (2017).


\bibitem[{Sai}17]{cc}
T.~{Saito},\emph{{The characteristic cycle and the singular support of a
  constructible sheaf}}, {{Invent. Math.}} \textbf{{207(2)}} ({2017}).

\bibitem[{Sai}21]{SaitoConductorDirectImage}
\bysame, \emph{The characteristic cycle and the singular support of a
  constructible sheaf}, {{J. Amer. Math. Soc.}} \textbf{34} (2021).
  
  \bibitem[SGA7]{SGA7}
A. Grothendieck {\it et al.}, \emph{ Groupes de monodromie en g\'eom\'etrie alg\'ebrique. S\'eminaire de G\'eom\'etrie Alg\'ebrique du Bois-Marie 1967-1969}, I dirig\'e par A. Grothendieck, II par P. Deligne et N. Katz, Lecture Notes in Math. 288, 340, Springer-Verlag, 1972, 1973.



\bibitem[{Sim}94]{SimpsonI}
C.~{Simpson}, \emph{{Moduli of representations of the fundamental group of a
  smooth projective variety I}}, {{Publ. Math. de l'IHES}} \textbf{79} (1994).

\bibitem[SPB]{SPBLow_up}
\emph{{Stack Project}}, ch.~Divisors.

\bibitem[SPs]{SPsurface}
\emph{{Stack Project}}, ch.~Resolution of surfaces.

\bibitem[{Sta}08]{Stalder}
N.~{Stalder}, \emph{{Scalar Extension of Abelian and Tannakian Categories.
  Preprint}}, 2008.

\bibitem[{Tey}16]{NearbySlope}
J.-B. {Teyssier}, \emph{A boundedness theorem for nearby slopes of holonomic
  $\mathcal{D}$-modules}, {{Compositio Mathematica}} \textbf{152} (2016).

\bibitem[{Tey}18]{teyConjThese}
\bysame, \emph{{Moduli of Stokes torsors and singularities of differential
  equations}}, {To appear at the J. Eur. Math. Soc.} (2018).

\bibitem[{Xia}15]{LXiao_CC_clean}
L.~{Xiao}, \emph{Cleanliness and log-characteristic cycles for vector bundles
  with flat connections}, {{Math. Ann.}} \textbf{362} (2015).

\bibitem[{Yat}20]{Yatagawa}
Y.~{Yatagawa}, \emph{Characteristic cycle of a rank one sheaf and ramification
  theory}, {{J. Algebraic Geom.}} \textbf{29} (2020).

\bibitem[{Zar}39]{Zariski_surface}
O.~{Zariski}, \emph{The reduction of the singularities of an algebraic
  surface}, {{Ann. Math.}} \textbf{40} (1939).

\end{thebibliography}

\end{document}